\documentclass[10pt]{amsart}
\usepackage{amssymb,amsfonts}
\usepackage[all,arc]{xy}
\usepackage{enumitem}
\usepackage{mathrsfs}
\usepackage{ stmaryrd }
\usepackage{url}
\usepackage[french,english]{babel}
\usepackage
[pdfauthor={Preston Wake and Carl Wang-Erickson},
 pdftitle={The Eisenstein ideal with squarefree level},
 bookmarks=false]
{hyperref}

\newcounter{qcounter}

\newcommand\define{\newcommand}

\define\isoto{\xrightarrow{\sim}}
\define\onto{\twoheadrightarrow}

\DeclareMathOperator{\Spec}{Spec}

\define\cQ{\mathcal{Q}}

\newcommand{\dia}[1]{{\langle #1 \rangle}}

\newcommand{\ttmat}[4]{\left( \begin{array}{cc}
#1 & #2 \\
#3 & #4
\end{array}
\right)}

\newcommand{\Z}{\mathbb{Z}}
\newcommand{\Q}{\mathbb{Q}}

\newcommand{\F}{\mathbb{F}}

\newcommand{\sO}{\mathcal{O}}

\newcommand{\m}{\mathfrak{m}}
\newcommand{\frt}{\mathfrak{t}}
\newcommand{\cP}{\mathcal{P}}

\newcommand{\varep}{\varepsilon}

\newcommand{\Hom}{\mathrm{Hom}}
\newcommand{\Gal}{\mathrm{Gal}}

\newcommand{\Ext}{\mathrm{Ext}}

\newcommand{\End}{\mathrm{End}}

\newcommand{\Fr}{\mathrm{Fr}}

\newcommand{\lb}{{[\![}}
\newcommand{\rb}{{]\!]}}

\newcommand{\red}{\mathrm{red}}

\define\cT{\mathcal{T}}

\define\ord{{\mathrm{ord}}}

\define\GL{{\mathrm{GL}}}

\define\kcyc{\kappa_{\mathrm{cyc}}}

\define{\Fitt}{\mathrm{Fitt}}
\define{\Ann}{\mathrm{Ann}}

\newtheorem{thm}{Theorem}[subsection] 
\newtheorem*{thm*}{Theorem}

\newtheorem{cor}[thm]{Corollary}
\newtheorem{prop}[thm]{Proposition}
\newtheorem{lem}[thm]{Lemma}

\theoremstyle{definition}
\newtheorem{defn}[thm]{Definition}

\newtheorem{eg}[thm]{Example}

\newtheorem{constr}[thm]{Construction}
\newtheorem{obs}[thm]{Observation}
\theoremstyle{remark}
\newtheorem{rem}[thm]{Remark}

\newcommand{\ra}{\rightarrow}

\newcommand{\lra}{\longrightarrow}
\newcommand{\lrisom}{\buildrel\sim\over\lra}
\newcommand{\risom}{\buildrel\sim\over\ra}
\newcommand{\rinj}{\hookrightarrow}

\newcommand{\rsurj}{\twoheadrightarrow}

\newcommand{\bQ}{\mathbb{Q}}

\newcommand{\bT}{\mathbb{T}}
\newcommand{\cE}{\mathcal{E}}
\newcommand{\cO}{\mathcal{O}}
\newcommand{\Db}{{\bar D}}

\newcommand{\cG}{{\mathcal{G}}}

\newcommand{\Jm}{{J^{\min{}}}}

\newcommand{\fl}{{\mathrm{flat}}}

\newcommand{\tf}{\tilde{f}}

\usepackage{color}

\newcommand{\Tr}{\mathrm{Tr}}

\newcommand{\sm}[4]{\ensuremath{\big(\begin{smallmatrix}#1 & #2 \\ #3 & #4\end{smallmatrix}\big)}}

\newcommand{\Soc}{\mathrm{Soc}}
\makeatletter
\let\c@equation\c@thm
\makeatother

\numberwithin{equation}{subsection}

\newcommand{\odia}[1]{\overline{\dia{#1}}}

\title{The Eisenstein ideal with squarefree level}

\author{Preston Wake}
\address{Department of Mathematics, Michigan State University \\
East Lansing, MI 48824}
\email{wakepres@msu.edu}

\author{Carl Wang-Erickson}
\address{Department of Mathematics, University of Pittsburgh \\
Pittsburgh, PA 15260}
\email{carl.wang-erickson@pitt.edu}

\begin{document}

\begin{abstract}
We use deformation theory of pseudorepresentations to study the analogue of Mazur's Eisenstein ideal with squarefree level. Given a prime number $p>3$ and a squarefree number $N$ satisfying certain conditions, we study the Eisenstein part of the $p$-adic Hecke algebra for $\Gamma_0(N)$, and show that it is a local complete intersection and isomorphic to a pseudodeformation ring. We also show that, in certain cases, the Eisenstein ideal is not principal and that the cuspidal quotient of the Hecke algebra is not Gorenstein. As a corollary, we prove that ``multiplicity one'' fails for the modular Jacobian $J_0(N)$ in these cases. In a particular case, this proves a conjecture of Ribet.
\end{abstract}

\maketitle

\tableofcontents

\section{Introduction}

In his landmark study \cite{mazur1978} of the Eisenstein ideal with prime level, Mazur named five ``special settings'' in which ``it would be interesting to develop the theory of the Eisenstein ideal in a broader context'' [pg.\ 39, \textit{loc.\ cit.}], the first of which is the setting of squarefree level. In this paper, we develop such a theory in certain cases. 

\subsection{Mazur's results and their squarefree analogues} 
\label{subsec:Mazur results}

Let $p \geq 3$ and $\ell$ be primes, and let $\bT_\ell$ be the $p$-adic Eisenstein completion of the Hecke algebra acting on modular forms of weight 2 and level $\ell$, and let $\bT_\ell\onto \bT_\ell^0$ be the cuspidal quotient. Let $I_\ell^0 \subset \bT_\ell^0$ be the Eisenstein ideal, and let $\m_\ell^0 = (p,I_\ell^0)$ be the maximal ideal. Mazur proved the following results \cite{mazur1978}: 
\begin{enumerate}[leftmargin=2em]
\item $\bT^0_\ell/I_\ell^0 \cong \Z_p/(\frac{\ell-1}{12})\Z_p$,
\item $I_\ell^0$ is principal,
\item $\bT^0_\ell$ is Gorenstein,
\item if $q\ne \ell$ is a prime such that $q \not \equiv 1 \pmod{p}$ and such that $q$ is not a $p$-th power modulo $\ell$, then $T_q-(q+1)$ generates $I_\ell^0$.
\end{enumerate}
Mazur calls a prime $q$ as in (4) a \emph{good prime for $(\ell,p)$}. We note that, of course, (4) implies (2) implies (3). We also note that (2) implies that $\bT_\ell$ is Gorenstein also.

The analogue of (1) has been proven for squarefree levels by Ohta \cite{ohta2014}. However, as has been noted by many authors, notably Ribet and Yoo \cite{ribet2015,yoo2015}, the statements (2)-(4) are not true in the squarefree setting. Still, in this paper, we prove, in certain cases, analogues of (2)-(4). Namely, we count the minimal number of generators of the Eisenstein ideal, count the dimension of the Eisenstein kernel of the Jacobian, and give sufficient (and sometimes also necessary) conditions for a list of elements $T_q-(q+1)$ to generate the Eisenstein ideal. As a corollary, we produce new level-raising results for modular forms congruent to Eisenstein series. 

\subsection{Motivation and applications} As applications of his results on the structure of the prime level Hecke algebra $\bT_\ell$, Mazur proves the following arithmetic results:
\begin{enumerate}[label=(\roman*), leftmargin=2em]
\item $J_0(\ell)(\Q)_{\mathrm{tors}}$ is a cyclic group of order $n$, where $n$ is the numerator of $\frac{\ell-1}{12}$, generated by the class of the divisor $(0)-(\infty)$.
%\item The maximal $\mu$-type subgroup of $J_0(\ell)$ over $\Q$ is the Shimura subgroup $\ker(J_0(\ell) \to J_1(\ell))$.
\item The dimension of $J_0(\ell)[\m_\ell^0]$ over $\F_p$ is $2$.
\end{enumerate}
Part (i) was conjectured by Ogg. As Mazur points out \cite[Remark, pg.\ 143]{mazur1978}, if one ignores the 2-torsion, part (i) is much easier and does not require the results (1)-(4) on the Hecke algebra. Indeed, Ohta has proven the squarefree analog of (i) (ignoring 2-torsion) \cite{ohta2014}. When we pass to squarefree level, the dimension in (ii) is no longer 2 in general;  Ribet and Yoo \cite{ribet2015,yoo2015} have partial results and conjectures as to what the dimension is. We count this dimension exactly, using our results on the Hecke algebra. 

Just as Mazur's results on $\bT_\ell$ have had many arithmetic applications, we expect that our results about the structure of $\bT_N$ for squarefree level $N$ will find more applications. We mention a few directions that are of particular interest to us:
\begin{itemize}[leftmargin=2em]
\item Connecting the rank of $\bT_N$ with Massey products, class groups, and Mazur-Tate $L$-functions, in analogy to our previous work \cite{WWE3} and the works of Merel \cite{merel1996} and Lecouturier \cite{lecouturier2020} in the prime level case. This should have application to Venkatesh's conjectures for derived Hecke algebras in the case of weight 1 forms with squarefree level, just as Merel's work is applied in the prime level case by Harris and Venkatesh \cite{HV2019}.
\item Implications of the Gorenstein property of $\bT_N$ for the arithmetic of cyclotomic fields and Iwasawa theory, as in the works of Ohta \cite{ohta2005} and Sharifi \cite{sharifi2011}.
\item Applications to the Iwasawa theory of residually reducible modular forms, esspecially conjectures of Greenberg \cite[Conj.~1.11]{greenberg1999} and Vatsal \cite[Conj.~1.14]{vatsal2005} on $\mu$-invariants.
\end{itemize}
It is also interesting to consider applications of our results in the setting of Hida theory (see \S \ref{subsec:hida} for a discussion of this). We hope to return to these applications in future work.

\subsection{Techniques of pseudomodularity} Our main technical result is an $R=\bT$ theorem, where $R$ is a deformation ring for Galois pseudorepresentations and $\bT$ is the Eisenstein part of the Hecke algebra. Although we consider this result to be secondary to our results on the structure of the Hecke algebra, we believe that the proof techniques we develop may be of independent interest, and are a step toward integral refinement of the modularity results of Skinner--Wiles \cite{SW1999}. Therefore we describe them here. The strategy is similar to that of our previous works \cite{WWE3, WWE4}, where we gave new proofs and refinements of Mazur's results. However, there are several points of interest that are new in this setting.
\begin{enumerate}[label=(\alph*), leftmargin=2em]
\item In the case of prime level $\ell$, Calegari and Emerton \cite{CE2005} have already applied deformation theory to study Mazur's Eisenstein ideal. Their method is to rigidify the deformation theory of Galois representations using auxiliary data coming from the prime level $\ell$. In the case of squarefree level, a similar strategy will not work: the deformation problem at prime level is already rigid, and cannot be further rigidified to account for the additional primes. 
\item
In the case of squarefree level, there are multiple Eisenstein series, and one has to account for the possibility of congruences among them. 
\item At squarefree level, unlike prime level, the Tate module of the Jacobian may not be free over the Hecke algebra. Since this Tate module is the natural way to construct Galois representations, it is really necessary to work with pseudorepresentations. 
\item We prove $R=\bT$ even in some cases where the Galois cohomology groups controlling the tangent space of $R$ are all non-cyclic (see Remark \ref{rem:no cyclicity assumption}). In this case, the universal pseudodeformation cannot arise from a representation. 
\end{enumerate}
To address issue (a), we have to develop a theory of Cayley--Hamilton representations and pseudorepresentations with squarefree level, which has the required flexibility; for this, we drew inspiration from our previous joint works \cite{WWE1,WWE3,WWE4} and the work of Calegari--Specter \cite{CS2016}. The ideas are discussed later in this introduction in \S \ref{subsec:psdef method}. To address issue (b), we make extensive use of an idea of Ohta \cite{ohta2014}: we use the Atkin--Lehner involutions at $\ell\mid N$ to define $\bT$, rather than the usual Hecke operators $U_\ell$.

\subsection{Setup} 
We introduce notation in order to state our main results precisely. Throughout the paper we fix a prime $p$ and let $N$ denote a squarefree integer with distinct prime factors $\ell_0, \ell_1, \dots, \ell_r$. The case $p \mid N$ is not excluded.

\subsubsection{Assumption on $p$} 
Throughout the paper we assume that $p>3$. The assumption that $p \ne 2$ is used crucially throughout the paper in several ways. First, we use the fact that there is no primitive $p$th root of unity in $\Q$, so the mod-$p$ cyclotomic character is non-trivial. Second, we use the fact that a local ring with residue characteristic $p$ cannot have a non-trivial involution, so $p$-adic modules with a Hecke action admit a direct sum decomposition according to the Aktin--Lehner eigenvalues. Finally -- and this is the only place where we also need $p \ne 3$ -- we use the fact that $\zeta(-1) = \frac{-1}{12}$ is a $p$-adic unit. This is reflected in the Galois cohomology computation that we quote from \cite{WWE3} as the fact that $K_i(\Z) \otimes \Z_p =0$ for $i=2,3$. It is also used to say that a non-zero constant cannot be a mod-$p$ modular form of weight 2. Because these do not seem to be crucial points, it is plausible that our techniques could be adapted to include the case $p=3$. However, we do not pursue this here.

\subsubsection{Eisenstein series and Hecke algebras} 
\label{sssec:eisen series and hecke alg defn}
The Eisenstein series of weight two and level $\Gamma_0(N)$ have a basis $\{E_{2,N}^\epsilon\}$, labeled by elements $\epsilon= (\epsilon_0, \dots, \epsilon_r)$ in the set $\cE =\{\pm 1\}^{r+1} \setminus \{(1,1, \dots, 1)\}$. The $E_{2,N}^\epsilon$ are characterized in terms of Hecke eigenvalues by the properties that
\begin{enumerate}[leftmargin=2em]
\label{eq:eisenstein eigenvalues}
\item $\displaystyle T_n E_{2,N}^\epsilon =\left (\sum_{0<t\mid n} t \right) E_{2,N}^\epsilon$ for all $n$ with $\gcd(n,N)=1$, and
\item $w_{\ell_i} E_{2,N}^\epsilon = \epsilon_i E_{2,N}^\epsilon$ for the Atkin--Lehner involutions $w_{\ell_0}, \dots, w_{\ell_r}$,
\end{enumerate}
together with the normalization $a_1(E_{2,N}^\epsilon)=1$. The constant coefficients satisfy
\begin{equation}
\label{eq:constant term}
a_0(E_{2,N}^\epsilon) = -\frac{1}{24} \prod_{i=0}^r (\epsilon_i\ell_i+1).
\end{equation}
(See \S\ref{subsec:eisenstein series} for more about these Eisenstein series.)
Based on the philosophy that congruences between Eisenstein series and cusp forms should happen when the constant term is divisible by $p$, we expect the most interesting congruences to occur when $\ell_i \equiv -\epsilon_i \pmod{p}$ for many $i$. (Note that we do not have to consider constant terms at other cusps: if a modular form $f$ of level $\Gamma_0(N)$ is an eigenform for all the Atkin--Lehner involutions, and $a_0(f)=0$, then $f$ is a cusp form.)

Consider the Hecke algebra of weight $2$ and level $\Gamma_0(N)$ generated by all $T_n$ with $\gcd(n,N)=1$ and by all Atkin--Lehner involutions $w_{\ell_0}, \dots, w_{\ell_r}$. Let $\bT_N^\epsilon$ denote the completion of this algebra at the maximal ideal generated by $p$ together with the annihilator of $E_{2,N}^\epsilon$. 

Let $I^\epsilon$ denote the annihilator of $E_{2,N}^\epsilon$ in $\bT_N^\epsilon$, so $\bT_N^\epsilon/I^\epsilon = \Z_p$, and let $\m^\epsilon =(I^\epsilon, p)$ be the maximal ideal of $\bT_N^\epsilon$. For a Hecke module $M$, let $M_\mathrm{Eis}^\epsilon$ denote the tensor product of $M$ with $\bT_N^\epsilon$ over the Hecke algebra. In particular, let $M_2(N)_\mathrm{Eis}^\epsilon$ (resp.\ $S_2(N)_\mathrm{Eis}^\epsilon$) denote the resulting module of modular forms (resp.\ cuspidal forms). Let $\bT_N^{\epsilon,0}$ denote the cuspidal quotient of $\bT_N^\epsilon$, and let $I^{\epsilon,0}$ be the image of $I^\epsilon$ in $\bT_N^{\epsilon,0}$.

\subsubsection{Another Hecke algebra}
\label{subsubsec:T_U}

In contrast with our approach, one often studies a different Hecke algebra $\bT^\epsilon_{N,U}$, containing the operators $U_\ell$ instead of $w_\ell$, and with Eisenstein ideal $I_U^\epsilon$ generated by $T_q-(q+1)$ for $q\nmid N$ and $U_{\ell_i}-\ell_i^\frac{\epsilon_i+1}{2}$ for $i=0,\dots, r$. We prove that $\bT_{N,U}^\epsilon = \bT^\epsilon_N$ in some of the cases that we consider --- see Appendix \ref{appendix:U and w}. Our main results together with Appendix \ref{appendix:U and w} can be used to prove results about $\bT^{\epsilon}_{N,U}$ that are closely related to the results of authors including Ribet \cite{ribet2010, ribet2015}, Yoo (\cite{yoo2015, yoo2017, yoo2017b} and others) and Hsu \cite{hsu2019}. 

We take the point of view that the reason to consider Hecke operators at primes dividing $N$ is to distinguish various oldforms modulo $p$. When $\bT_N^\epsilon \ne \bT_{U,N}^\epsilon$, it is because there are multiple oldforms that have congruent $U_\ell$-eigenvalues for some $\ell \mid N$. Because this multiplicity does not occur among $w_\ell$-eigenvalues, such multiplicity causes $\bT_{U,N}^\epsilon$ to have larger rank than $\bT_N^\epsilon$. Therefore, we think of $\bT^\epsilon_N$ as a superior to $\bT^\epsilon_{U,N}$ as a superior desingularization of the unramified Hecke algebra (that is, the Hecke algebra generated by $T_n$ for $(N,n) = 1$). We mostly consider $\bT^\epsilon_N$, but see Appendix \ref{appendix:U and w} for a comparison of $\bT^\epsilon_N$ and $\bT^\epsilon_{U,N}$. 

\subsubsection{The number fields $K_i$}
\label{subsubsec:defn of K_i}
Let $\ell$ be a prime such that $\ell \equiv \pm 1 \pmod{p}$. Then there is a unique degree $p$ Galois extension $K_\ell/\Q(\zeta_p)$ such that
\begin{enumerate}
\item $\Gal(\Q(\zeta_p)/\Q)$ acts on $\Gal(K_\ell/\Q(\zeta_p))$ via the character $\omega^{-1}$,
\item the prime $(1-\zeta_p)$ of $\Q(\zeta_p)$ splits completely in $K_\ell$, and
\item only the primes above $\ell$ ramify in $K_\ell/\Q(\zeta_p)$.
\end{enumerate}
For each $i$ such that $\ell_i \equiv \pm 1 \pmod{p}$, let $K_i=K_{\ell_i}$ (see also Definition \ref{defn:K_i}).

\subsection{Structure of the Hecke algebra}
\label{subsec:main thms}
Our main results concern the structure of the Hecke algebra $\bT_N^\epsilon$.
\begin{thm}
\label{thm:main r primes}
Assume that $\epsilon=(-1,1,\dots,1)$. Let 
\[
\mathcal{S}=\{i\in\{1,\dots, r\} \mid \ell_i \equiv -1 \pmod{p}\}
\]
and let $s=\#\mathcal{S}$. Then
\begin{enumerate}
\item $\bT_N^\epsilon$ is a complete intersection ring.
\item $\bT_N^{0,\epsilon}$ is Gorenstein if and only if $I^{\epsilon}$ is principal.
\item There is a short exact sequence 
\begin{equation}
\label{eq:main SES}
0 \to  \bigoplus_{i=1}^r \Z_p/(\ell_i+1)\Z_p \to I^\epsilon/{I^\epsilon}^2 \to \Z_p/(\ell_0-1)\Z_p \to 0.
\end{equation}
\item The minimal number of generators of $I^\epsilon$ is $s+\delta$ where 
\[
\delta =\left\{
\begin{array}{lc}
1 & \text{ if } \ell_0 \text{ splits completely in } K_i \text{ for all } i \in \mathcal{S}, \text{ or} \\
0 & \text{ otherwise.}
\end{array}\right.
\]
\end{enumerate}
\end{thm}
\begin{proof}
Parts (1) and (3) are proved in \S\ref{sec:case1} (see especially Theorem \ref{thm:star main}). It is known to experts that Part (2) follows from (1) (see Lemma \ref{lem:goren and I principal}). Part (4) is Theorem \ref{thm:star split}.
\end{proof}

\begin{rem}
In fact, we show that, unless $s=r$, there are no newforms in $M_2(N)_\mathrm{Eis}^\epsilon$, so we can easily reduce to the case $s=r$ (i.e.~the case that $\ell_i \equiv -1 \pmod{p}$ for all $i>0$). When $s=r$, one could use this theorem to prove that there are newforms in $M_2(N)_\mathrm{Eis}^\epsilon$, but this is known (see \cite{ribet2015}, \cite[Thm.\ 1.3(3)]{yoo2017}).
\end{rem}

\begin{rem}
The criterion of Part (4) determines whether or not the extension class defined by the sequence \eqref{eq:main SES} is $p$-cotorsion. In fact, one can describe this extension class exactly in terms of algebraic number theory, but we content ourselves with the simpler statement (4).
\end{rem}

\begin{thm}
\label{thm:main 2 primes no new}
Assume $r=1$ and $\epsilon=(-1,-1)$ and that $\ell_0 \equiv 1 \pmod{p}$ but $\ell_1 \not \equiv 1 \pmod{p}$. If $\ell_1$ is not a $p$-th power modulo $\ell_0$, then there are no newforms in $M_2(N)_\mathrm{Eis}^\epsilon$. In particular, $I^\epsilon$ is principal, and generated by $T_q-(q+1)$ where $q$ is a good prime (of Mazur) for $(\ell_0,p)$.
\end{thm}
\begin{proof}
This is Theorem \ref{thm:one interesting prime}. 
\end{proof}
\begin{rem}
In the case $\ell_1 \ne p$, this is a theorem of Ribet \cite{ribet2010} and Yoo \cite[Thm.\ 2.3]{yoo2017}. Yoo has informed us that the method should work for the case $\ell_1 =p$ as well. In any case, our method is completely different.
\end{rem}

\begin{thm}
\label{thm:main 2 primes new}
Assume $r=1$ and $\epsilon=(-1,-1)$ and that $\ell_0 \equiv \ell_1 \equiv 1 \pmod{p}$. Assume further that
\begin{center}
$\ell_i$ is not a $p$-th power modulo $\ell_j$ for $(i,j) \in \{(0,1),(1,0)\}$.
\end{center}
 Then
\begin{enumerate}
\item there are newforms in $M_2(N)_\mathrm{Eis}^\epsilon$.
\item $\bT_N^\epsilon$ is a complete intersection ring.
\item $\bT_N^{\epsilon,0}$ is not a Gorenstein ring.
\item $I^{\epsilon,0}/{I^{\epsilon,0}}^2 \cong \Z_p/(\ell_0-1)\Z_p \oplus \Z_p/(\ell_1-1)\Z_p$.
\end{enumerate}
\end{thm}
\begin{proof}
Parts (2) and (4) are proven in Theorem \ref{thm:thm2}. Part (1), the precise meaning of which is given in Definition \ref{defn:no newforms}, follows from Part (2) by Theorem \ref{thm:newforms in case 2}. Part (3) follows from (2) and (4) by Lemma \ref{lem:goren and I principal}.
\end{proof}

\begin{rem}
\label{rem:no cyclicity assumption}
The proof of this theorem may be of particular interest for experts in the deformation theory of Galois representations. The proof is the first (as far as we are aware) example of an $R=\bT$ theorem, where $R$ is a universal pseudodeformation ring, and where we do not rely on certain Galois cohomology groups being cyclic. (This cyclicity ensures that the pseudorepresentations come from true representations.) In fact, with the assumptions of the theorem, the relevant cohomology groups are not cyclic. However, see \cite[Thm.\ 8.2]{BerKlo2015}, where $R' = \bT$ is proved, where $R'$ is a certain quotient of a universal pseudodeformation ring. 
\end{rem}

\begin{rem} 
Outside of the cases considered in these theorems, we cannot expect that $\bT_N^\epsilon$ is a complete intersection ring, as Remark \ref{rem: counting relations} and the examples in \S \ref{subsec:examples} below illustrate. Our method, which applies Wiles's numerical criterion \cite{wiles1995}, proves that $\bT_N^\epsilon$ is a complete intersection ring as a byproduct. A new idea is needed to proceed beyond these cases. The authors along with C.\ Hsu are currently working out such an idea \cite{HWWE2021}. 
\end{rem}

\begin{rem} 
\label{rem: counting relations}
Consider the case $\epsilon=(-1,-1,\dots,-1)$ with $\ell_i \equiv 1 \pmod{p}$ for $i=0,\dots, r$. There is a numerological reason why our arguments work for $r=1$, but not for $r>1$.  To see that $\bT_N^\epsilon$ satisfies the numerical criterion, its cotangent module $I^\epsilon/{I^\epsilon}^2$ must not be any bigger than its reducible quotient contributed by Lemma \ref{lem:computation of Rred}. In order for the irreducible submodule of $I^\epsilon/{I^\epsilon}^2$ to vanish, we have to show that there are $(r+1)^2$ relations which kill off all of the $(r+1)^2$ generators. We can always see that $(r+1)$ of them hold, and when certain additional conditions (like the assumptions in Theorem \ref{thm:main 2 primes new}) on the $\ell_i$ hold, we show that another $(r+1)$ relations hold (see Lemma \ref{lem:b_sig and b_gam generate}). This gives a total of $2(r+1)$, and only when $r=1$ do we have $(r+1)^2=2(r+1)$.
\end{rem}

\subsection{Applications to multiplicity one}
For an application of the main result, we let $J_0(N)$ be the Jacobian of the modular curve $X_0(N)$.
\begin{cor}
\label{cor:mult one fails}
In the following cases, we can compute $\dim_{\F_p} J_0(N)(\overline{\Q}_p)[\m^\epsilon]$: 
\begin{enumerate}[leftmargin=2em]
\item With the assumptions of Theorem \ref{thm:main r primes}, we have 
\[
\dim_{\F_p} J_0(N)(\overline{\Q}_p)[\m^\epsilon] =1+s+\delta,
\]
where $s$ and $\delta$ are as in Theorem \ref{thm:main r primes}.
\item With the assumptions of Theorem \ref{thm:main 2 primes no new}, we have $\dim_{\F_p} J_0(N)(\overline{\Q}_p)[\m^\epsilon] =2$.
\item With the assumptions of Theorem \ref{thm:main 2 primes new}, we have $\dim_{\F_p} J_0(N)(\overline{\Q}_p)[\m^\epsilon] =3$.
\end{enumerate}
\end{cor}
\begin{proof}
This follows from the named theorems together with Lemma \ref{lem:I and Jacobian kernel} (which is known to experts).
\end{proof}
One says that ``multiplicity one holds'' if $\dim_{\F_p} J_0(N)(\overline{\Q}_p)[\m^\epsilon] =2$. This corollary implies that multiplicity one holds in case (1) if and only if $s+\delta=1$, always holds in case (2), and always fails in case (3).

\subsubsection{Ribet's Conjecture} Previous works on multiplicity one have used a different Hecke algebra $\bT_{N,U}^\epsilon$, defined in \S \ref{subsubsec:T_U} (see, for example, \cite{yoo2015}). Let $\m_U^\epsilon=(I_U^\epsilon,p) \subset \bT^\epsilon_{N,U}$ be its maximal ideal. The previous corollary together with Proposition \ref{prop:T=T_U first case} give the following.
\begin{cor}[Generalized Ribet's Conjecture]
\label{cor: gen Ribet conj}
With the assumptions of Theorem \ref{thm:main r primes}, assume in addition that $\ell_i \not \equiv 1 \pmod{p}$ for $i>0$. Then 
\[
\dim_{\F_p} J_0(N)(\overline{\Q}_p)[\m_U^\epsilon]=1+s+\delta,
\]
where $s$ and $\delta$ are as in Theorem \ref{thm:main r primes}.
\end{cor}
The case $s=r=1$ of Corollary \ref{cor: gen Ribet conj} was conjectured by Ribet \cite{ribet2015} (see also \cite[pg.~4]{yoo2017b}).
\begin{rem}
After we told Yoo about the results of this paper, he found an alternate proof of this corollary in the case $s=r=1$. His proof involves a delicate study of the geometry of $J_0(N)$, and relies on the following particular results of this paper:
\begin{enumerate}[label=(\roman*), leftmargin=2em]
\item  $I^\epsilon$ is principal if and only if $\bT^{0,\epsilon}_{N}$ is Gorenstein, from Theorem \ref{thm:main r primes}(2), and
\item $\bT^{0,\epsilon}_{N,U}=\mathbb{T}_N^{0,\epsilon}$ under the assumption $s=r=1$, from Proposition \ref{prop:T=T_U first case}, so that the conclusion of (i) can be applied to the ideal $I_U^\epsilon \subset \bT^{0,\epsilon}_{N,U}$. 
\end{enumerate}
In particular, Yoo's proof does not make use of our formula for the number of generators for $I^\epsilon$ in Theorem \ref{thm:main r primes}(4), and we believe that his methods could be used to give a new proof of that result in this case.

In contrast, our proof is immediate from the ring-theoretic properties given in Theorem \ref{thm:main r primes} and a standard argument (found in \cite{mazur1978}, for example), and no additional geometric argument is needed. The fact that our proof is almost completely ring-theoretic demonstrates the power of the Gorenstein property and is a reason for our interest in using $\mathbb{T}_N^\epsilon$ rather than $\mathbb{T}_{N,U}^\epsilon$.
\end{rem}

\subsubsection{Gorensteinness, and multiplicity one for the generalized Jacobian} The following observations are not used (nor proven) in this paper (although they are familiar to experts), but we include them to illustrate the the arithmetic significance of the Gorenstein property for $\bT_N^\epsilon$ proved in Theorems \ref{thm:main r primes}, \ref{thm:main 2 primes no new} and \ref{thm:main 2 primes new}. We learned this point of view from papers of Ohta, especially \cite{ohta2005}. 

As is well-known, and as we explain in \S \ref{subsec:gorenstein and jacobian}, multiplicity one holds if and only if $\bT_N^{0,\epsilon}$ is Gorenstein. The nomenclature ``multiplicity one" comes from representation theory. It is related to the question of whether $H^1_{\mathrm{\acute{e}t}}(X_0(N)_{\overline{\Q}},\Z_p(1))_\mathrm{Eis}^\epsilon$ is a free $\bT_N^{0,\epsilon}$-lattice in the free $\bT_N^{0,\epsilon}[\frac{1}{p}]$-module $H^1_{\mathrm{\acute{e}t}}(X_0(N)_{\overline{\Q}},\Q_p(1))_\mathrm{Eis}^\epsilon$. 

There is another natural lattice to consider, namely $H^1_{{\mathrm{\acute{e}t}}}(Y_0(N)_{\overline{\Q}},\Z_p(1))_{\m^\epsilon,\mathrm{DM}}$, the image of $H^1_{{\mathrm{\acute{e}t}}}(Y_0(N)_{\overline{\Q}},\Z_p(1))_\mathrm{Eis}^\epsilon$ under the Drinfeld-Manin splitting
\[
H^1_{{\mathrm{\acute{e}t}}}(Y_0(N)_{\overline{\Q}},\Q_p(1))_\mathrm{Eis}^\epsilon \lra H^1_{{\mathrm{\acute{e}t}}}(X_0(N)_{\overline{\Q}},\Q_p(1))_\mathrm{Eis}^\epsilon.
\]
In a similar manner to the proof of Lemma \ref{lem:failure of mult one = gorenstein defect}, one can show that $\bT_N^\epsilon$ is Gorenstein if and only if $H^1_{{\mathrm{\acute{e}t}}}(Y_0(N)_{\overline{\Q}},\Z_p(1))_{\m^\epsilon,\mathrm{DM}}$ is a free $\bT_N^{0,\epsilon}$-module, if and only if 
\[
\dim_{\F_p} GJ_0(N)(\overline{\Q}_p)[\m^\epsilon]=2,
\]
where $GJ_0(N)$ is the generalized Jacobian of $J_0(N)$ relative to the cusps (see e.g.\ \cite[\S3]{ohta1999} for a discussion of generalized Jacobians). Hence our result that $\bT_N^\epsilon$ is Gorenstein can be thought of as a multiplicity one result for $GJ_0(N)$.

Finally, we note that these ideas illustrate why the failure of multiplicity one in Corollary \ref{cor:mult one fails} is related to the failure of $I^\epsilon$ to be principal: if $\bT_N^\epsilon$ is Gorenstein, 
\[
H^1_{{\mathrm{\acute{e}t}}}(X_0(N)_{\overline{\Q}},\Z_p(1))_\mathrm{Eis}^\epsilon \rinj H^1_{{\mathrm{\acute{e}t}}}(Y_0(N)_{\overline{\Q}},\Z_p(1))_{\m^\epsilon,\mathrm{DM}}
\]
has the form, as $\bT^{0, \epsilon}_N$-modules, of
\[
\bT_N^{0,\epsilon} \oplus I^{0,\epsilon} \rinj \bT_N^{0,\epsilon} \oplus \bT_N^{0,\epsilon}.
\]
Hence $H^1_{{\mathrm{\acute{e}t}}}(X_0(N)_{\overline{\Q}},\Z_p(1))_\mathrm{Eis}^\epsilon$ is free if and only if $I^{0,\epsilon}$ is principal.

\subsection{Good primes} We also prove analogues of Mazur's good prime criterion (statement (5) of \S \ref{subsec:Mazur results}).

In the situation of Theorem \ref{thm:main 2 primes new}, our good prime criterion is necessary and sufficient, exactly analogous to Mazur's. To state it, we let
\[
\log_\ell:(\Z/\ell\Z)^\times \rsurj \F_p
\]
denote an arbitrary surjective homomorphism, for any prime $\ell$ that is congruent to $1$ modulo $p$ (the statement below will not depend on the choice).
\begin{thm}
\label{thm:main good primes 2}
With the assumptions of Theorem \ref{thm:main 2 primes new}, fix primes $q_0,q_1$ not dividing $N$ (but possibly dividing $p$). Then the elements $T_{q_0}-(q_0+1)$ and $T_{q_1}-(q_1+1)$ together generate $I^\epsilon$ if and only if
\[
(q_0-1)(q_1-1)\det\ttmat{\log_{\ell_0}(q_0)}{\log_{\ell_0}(q_1)}{\log_{\ell_1}(q_0)}{\log_{\ell_1}(q_1)} \in \F_p^\times.
\]
\end{thm}
\begin{rem}
For a single prime $\ell$, Mazur's criterion for $q$ to be a good prime can be written as $(q-1)\log_\ell(q) \in \F_p^\times$, so this is a natural generalization.
\end{rem}

In the situation of Theorem \ref{thm:main r primes}, we only give a sufficient condition, and even this is cumbersome to state. 

\begin{defn}
\label{defn:good primes}
Assume that $\epsilon=(-1,1,\dots,1)$, and order the primes $\ell_i$ so that $\ell_i \equiv -1\pmod{p}$ for $i=1,\dots,s$ and $\ell_i \not \equiv -1 \pmod{p}$ for $s < i \leq r$. We use the number fields $K_i$ set up in \S \ref{subsubsec:defn of K_i}.

Consider an ordered set of primes $\cQ'=\{q_0,q_1,\dots,q_s\}$ disjoint from the primes dividing $N$ and satisfying the following conditions:
\begin{enumerate}
\item $q_0 \not \equiv 1 \pmod{p}$, and 
\item $q_0$ not a $p$-th power modulo $\ell_0$;
\end{enumerate}
and, for $i=1,\dots,s$,
\begin{enumerate}[resume]
\item $q_i \equiv 1 \pmod{p}$, 
\item $\ell_0$ is not a $p$-th power modulo $q_i$,
\item $q_i$ does not split completely in $K_i$, and
\item $q_i$ does split completely in each $K_j$ for $j=1,\dots,s$ with $j \ne i$. 
\end{enumerate}
In the following cases, the described ordered subset $\cQ$ of $\cQ'$ is called a \emph{good set of primes for $(N,p,\epsilon)$}:
\begin{itemize}
\item if $\delta=1$,  $\cQ:=\cQ'$,
\item if $\delta=0$ and $\ell_0 \equiv 1 \pmod{p}$, then $\cQ :=\cQ'\setminus\{q_j\}$ for an index $j>0$ such that $b_0 \cup c_j \ne  0$,
\item if $\ell_0 \not \equiv 1 \pmod{p}$, then $\cQ :=\cQ'\setminus\{q_0\}$.
\end{itemize}
\end{defn}

\begin{rem}Note that, by Chebotarev density, there is an infinite set of primes $q_0$ satisfying (1)-(2), and, for each $i$, there is an infinite set of primes $q_i$ satisfying (3)-(6). Note that when $p \nmid N$ and $\ell_0 \equiv 1 \pmod{p}$, it is possible that $p \in \cQ$. 
\end{rem}
\begin{thm}
\label{thm:main good primes r}
Let $\cQ$ be a good set of primes for $(N,p,\epsilon)$. Then $\{T_q - (q+1) \mid q \in \cQ\} \subset \bT^\epsilon_N$ is a minimal set of generators for $I^\epsilon$.
\end{thm}

\begin{rem}
We can also write down a necessary and sufficient condition, but cannot compute with it, so we doubt its practical use.
\end{rem}

\subsection{Relation to Hida Hecke algebras}
\label{subsec:hida} The reader will note that we have allowed for the possibility that $p\mid N$. When $p\mid N$, in Appendix \ref{appendix:U and w}, we also consider a related Hecke algebra $\bT_{N,H}^\epsilon$ that contains $U_p$ instead of $w_p$ (but still has all other $w_\ell$ for $\ell \mid  \frac{N}{p}$) and show that, in many cases we consider, $\bT_{N,H}^\epsilon= \bT_{N}^\epsilon$.

This is related to Hida theory assuming that (as is well-known for the Hecke algebra $\bT_{N,U}^\epsilon$) there is a Hida-theoretic Hecke algebra $\bT_\Lambda^\epsilon$ that is a free module of finite rank over $\Lambda \simeq \Z_p \lb T \rb$ that satisfies a control theorem with respect to $\bT_{N,H}^\epsilon$: there is an element $\omega_2 \in \Lambda$ such that $\bT_{N,H}^\epsilon = \bT_\Lambda^\epsilon/\omega_2 \bT_\Lambda^\epsilon$.

Then our results about $\bT_N^\epsilon$ (including its Gorensteinness and the number of generators of its Eisenstein ideal) translate directly to $\bT_\Lambda^\epsilon$. Subsequently, these results can be specialized into higher weights, as is usual in Hida theory.

\subsection{Method of pseudodeformation theory}
\label{subsec:psdef method}

Like our previous work \cite{WWE3}, the method of proof of the theorems in \S \ref{subsec:main thms} is to construct a pseudodeformation ring $R$ and prove that $R =\bT$ using the numerical criterion. The ring $R$ is the deformation ring of the residual pseudorepresentation $\Db = \psi(\omega \oplus 1)$ associated to $E^\epsilon_{2,N}$ that is universal subject to certain conditions (here $\psi$ is the functor associating a pseudorepresentation to a representation, and $\omega$ is the mod $p$ cyclotomic character). These conditions include the conditions considered in our previous works \cite{WWE1,WWE3} (having cyclotomic determinant, being flat at $p$, being ordinary at $p$), but they also include new conditions at $\ell$ dividing $N$ that are of a different flavor, as we now explain.

\subsubsection{The Steinberg at $\ell$ condition}
Fix $\ell=\ell_i \mid  N$, assume $\ell \neq p$, and let $G_\ell\subset G_\Q$ be a decomposition group at $\ell$. Let $f$ be a normalized cuspidal eigenform of weight 2 and level $\Gamma_0(N)$. Let $\rho_f:G_\Q \to \GL_2(\sO_f)$ be the associated Galois representation, where $\sO_f$ is a finite extension of $\Z_p$.

If $f$ is old at $\ell$, then $\rho_f|_{G_\ell}$ is unramified. If $f$ is new at $\ell$, we have
\begin{equation}
\label{eq:f on G_ell}
\rho_f|_{G_\ell} \sim \ttmat{ \lambda(a_\ell(f))\kcyc}{*}{0}{\lambda(a_\ell(f))}
\end{equation}
where $\lambda(x)$ is the unramified character of $G_\ell$ sending a Frobenius element $\sigma_\ell$ to $x$, and $a_\ell(f)$ is the coefficient of $q^\ell$ in the $q$-expansion of $f$ (see Lemma \ref{lem:normalization of bT0}). Note that since $\det(\rho_f)=\kcyc$, we have $\lambda(a_\ell(f))^2=1$. In fact, $a_\ell(f)$ is the negative of the $w_\ell$-eigenvalue of $f$. We call such representations \eqref{eq:f on G_ell} ``$\pm1$-Steinberg at $\ell$", where $\pm 1=\mp a_\ell(f)$ is the $w_\ell$-eigenvalue of $f$.

Now assume in addition that $f \in S_2(N)_\mathrm{Eis}^\epsilon$, so that the semi-simplification of the residual representation of $\rho_f$ is $\omega \oplus 1$ and $w_\ell f=\epsilon f$, where $\epsilon=\epsilon_i$. We want to impose a condition on pseudorepresentations that encapsulates the condition that $\rho_f|_{G_\ell}$ is either unramified or $\epsilon$-Steinberg. The main observation is the following, and is inspired by the work of Calegari--Specter \cite{CS2016}.

\begin{obs}
Suppose that $\rho: G_\ell \to  \GL_2(\sO)$ is either unramified or $\epsilon$-Steinberg. Then
\begin{equation}
\label{eq:unram-or-stein}
(\rho(\sigma)-\lambda(-\epsilon)\kcyc(\sigma))(\rho(\tau)-\lambda(-\epsilon)(\tau))=0
\end{equation}
for all $\sigma, \tau \in G_\ell$ with at least one of $\sigma$ or $\tau$ in the inertia group $I_\ell$.
\end{obs}
This is clear if $\rho$ is unramified: the factor involving the one of $\sigma$ or $\tau$ that is in $I_\ell$ will be zero. If $\rho$ is $\epsilon$-Steinberg, then the given product \eqref{eq:unram-or-stein} will have the form
\[
\ttmat{0}{*}{0}{*} \ttmat{*}{*}{0}{0}
\]
and any such product is zero (note that the order is important!).

To impose the unramified-or-$\epsilon$-Steinberg condition on the pseudodeformation ring $R$, we impose the condition \eqref{eq:unram-or-stein} on the universal Cayley--Hamilton algebra, using the theory of \cite{WWE4} (see \S \ref{sec:pseudodef ring}).

\subsubsection{The ordinary at $p$ condition} When $p \mid N$ and $f \in S_2(N)_\mathrm{Eis}^\epsilon$ is a newform, then $\epsilon_p = -1$ and the representation $\rho_f\vert_{G_p}$ is ordinary. In this paper, we define ``ordinary pseudorepresentation" exactly as we define the unramified-or-$\epsilon$-Steinberg, following ideas of Calegari--Specter. In our previous paper \cite{WWE1}, we gave a different definition of ordinary, and we prove in this paper that the two definitions coincide (see Lemma \ref{lem:WWE ord = CS ord}). This gives an answer to a question of Calegari--Specter \cite[pg.\ 2]{CS2016}.

\subsection{Examples} 
\label{subsec:examples}
We conclude this introduction with examples that illustrate the theorems and show that the hypotheses are necessary. For examples where we show that $\bT_N^\epsilon$ is not Gorenstein, it is helpful to note that $\bT_N^\epsilon$ is Gorenstein if and only if $\Soc(\bT_N^\epsilon/p\bT_N^\epsilon)$ is 1-dimensional, where $\Soc(\bT_N^\epsilon/p\bT_N^\epsilon)$ is the annihilator of the maximal ideal (see \S\ref{subsec:Gorenstein defect}).

All computations are using algorithms we have written for the Sage computer algebra software \cite{SAGE}. 

\subsubsection{Examples illustrating Theorem \ref{thm:main r primes}} 
\label{subsubsec:r primes examples}

\begin{eg}
Let $p=5$, $\ell_0=41$, $\ell_1=19$, so $N=19\cdot 41$, and let $\epsilon=(-1,1)$. In this case, we compute that $K_{19}$ is the field cut out by 
\begin{align*}
x^{20} - x^{19} & - 7x^{18} + 21x^{17} + 22x^{16} + 223x^{15} - 226x^{14} - 1587x^{13} + 4621x^{12}  \\
& + 5202x^{11} - 91x^{10} - 3142x^9 - 439x^8 - 2143x^7 - 2156x^6 - 58x^5 \\
&  + 1237x^4 + 414x^3 + 148x^2 + 56x + 16
\end{align*}
and that $41$ splits completely in $K_{19}$. The theorem says that $I^\epsilon$ has 2 generators. Moreover, Theorem \ref{thm:main good primes r} says, in this case, that $I^\epsilon$ is generated by $T_{q_0}-(q_0+1)$ and $T_{q_1}-(q_1+1)$ where $q_0$ is a good prime for $(41,5)$ and where $q_1$ satisfies
\begin{enumerate}[label=(\alph*)]
\item $q_1$ is a prime such that $q_1 \equiv 1 \pmod{5}$, 
\item $41$ is not a $5$-th power modulo $q_1$, and
\item $q_1$ does not split completely in $K_{19}$.
\end{enumerate}
A quick search yields that $q_0=2$ and $q_1=11$ satisfy these criteria. And indeed, we compute that there is an isomorphism 
\[
\frac{\F_5[x,y]}{(y^2-2x^2,xy)} \lrisom \bT_N^\epsilon/5\bT_N^\epsilon, \quad (x,y) \mapsto (T_{2}-3, T_{11}-12).
\]
\end{eg}

\begin{eg}
Let $p=5$, $\ell_0=11$, $\ell_1=19$, $\ell_2=29$, so $N=11\cdot 19 \cdot 29$, and let $\epsilon=(-1,1,1)$. In this case, $11$ does not split completely in either of the fields $K_{19},K_{29}$, and the theorem says that $I^\epsilon$ has 2 generators. Moreover, Theorem \ref{thm:main good primes r} says, in this case, that $I^\epsilon$ is generated by $T_{q_0}-(q_0+1)$ and $T_{q_1}-(q_1+1)$ where $q_0$ is a good prime for $(11,5)$ (for example $q_0=2$) and where the prime $q_1$ satisfies:
\begin{enumerate}[label=(\alph*)]
\item $q_1 \equiv 1 \pmod{5}$,
\item $11$ is not a $5$-th power modulo $q_1$,
\item $q_1$ does not split completely in $K_{19}$, and
\item $q_1$ does split completely in $K_{29}$.
\end{enumerate}
In this case, $K_{19}$ is the field computed in the previous example and $K_{29}$ is the field cut out by
\begin{align*}
x^{20} - x^{19}& - 11x^{18} + 9x^{17} + 124x^{16} - 223x^{15} - 1244x^{14} + 2111x^{13} + 14291x^{12}  \\
&- 19804x^{11} + 7169x^{10} + 7938x^9 - 10937x^8 + 15603x^7 - 9472x^6  \\
&- 2582x^5 + 8233x^4 - 3732x^3 + 1808x^2 - 832x + 256.
\end{align*}
A quick search finds that $q_1=181$ satisfies the conditions (a)-(d). And indeed, we compute that there is an isomorphism
\[
\frac{\F_5[x,y]}{(x^3+2x^2, y^3, xy+y^2)} \lrisom \bT_N^\epsilon/5\bT_N^\epsilon, \quad (x,y) \mapsto (T_2-3,T_{181}-182).
\]

Note that these conditions are far from necessary. For example $T_2-3$ and $T_7-8$ also generate the Eisenstein ideal.
\end{eg}

\subsubsection{Examples related to Theorem \ref{thm:main 2 primes no new}}
\label{subsubsec:2 primes no new examples}
 We give examples illustrating that the assumption is necessary. 
In fact, it seems that the assumption is necessary even for the Gorensteinness of $\bT_N^\epsilon$.

\begin{eg}
Let $p=5$, $\ell_0=11$, $\ell_1=23$, so $N=11\cdot 23$, and let $\epsilon=(-1,-1)$. Then $\ell_1 \equiv 1 \pmod{11}$ is a $5$-th power so the theorem does not apply. We can compute that
\[
\frac{\F_5[x,y]}{(x^2,xy,y^2)} \lrisom \bT_N^\epsilon/5\bT_N^\epsilon, \quad (x,y) \mapsto (T_2 - 3, T_3-4)
\]
has dimension 3. Since $\bT_{11}^0=\Z_5$, we see that the space of oldforms has dimension 2, so there must be a newform at level $N$. Moreover, $\Soc(\bT_N^\epsilon/5\bT_N^\epsilon)=x\F_5\oplus y \F_5$, so $\bT_N^\epsilon$ is not Gorenstein.
\end{eg}

\begin{eg}
Let $p=5$, $\ell_0=31$, $\ell_1=5$, so $N=5 \cdot 31$, and let $\epsilon=(-1,-1)$. Then note that $\ell_1=5 \equiv 7^5 \pmod{31}$, so the theorem does not apply. We can compute that
\[
\frac{\F_5[x,y]}{(x^3,xy,y^2)} \lrisom \bT_N^\epsilon/5\bT_N^\epsilon, \quad (x,y) \mapsto (T_2-3, 2T_2+T_3)
\]
has dimension 4. Since $\mathrm{rank}_{\Z_5}(\bT_{31}^0)=2$, we see that the space of oldforms has dimension 3, and there must be a newform at level $N$. Moreover, $\Soc(\bT_N^\epsilon/5\bT_N^\epsilon)=x^2\F_5\oplus y \F_5$, so $\bT_N^\epsilon$ is not Gorenstein.
\end{eg}

In this last example, the reader may think that $\ell_0=31$ is special because the rank of $\bT_{31}^0$ is 2. However, we can take $p=\ell_1=5$ and $\ell_0=191$ (note that $\bT^0_{191}=\Z_p$). Noting that $5 \equiv 18^5 \pmod{191}$, we again see that the theorem does not apply, and we can compute that $\bT_N^\epsilon$ is also not Gorenstein in this case.

\subsubsection{Examples related to Theorem \ref{thm:main 2 primes new}}
\label{subsubsec:2 primes new examples} First, we give examples illustrating that the assumption is necessary. Again, it seems that the assumption is necessary even for the Gorenstein property of $\bT_N^\epsilon$.
\begin{eg}
Let $p=5$, $\ell_0=11$, $\ell_1=61$, so $N=11 \cdot 61$, and let $\epsilon=(-1,-1)$. Then note that $11 \equiv 8^5 \pmod{61}$ so the theorem does not apply (but note that $61$ is not a $5$-th power modulo $11$). We can compute that
\[
\frac{\F_5[x,y]}{(x^2,xy,y^3)} \isoto \bT_N^\epsilon/5\bT_N^\epsilon, \quad  (x,y) \mapsto (T_3-T_2-1, T_2-3).
\]
We see that $\mathrm{Soc}(\bT_N^\epsilon/5\bT_N^\epsilon)=x\F_5 \oplus y^2\F_5$, so $\bT_N^\epsilon$ is not Gorenstein.
\end{eg}

\begin{eg}
Let $p=5$, $\ell_0=31$, $\ell_1=191$, so $N=31 \cdot 191$, and let $\epsilon=(-1,-1)$. We have $191 \equiv 7^5 \pmod{31}$ and $31 \equiv 61^5 \pmod{191}$, so the assumption of the theorem fails most spectacularly. We can compute that
\begin{align*}
\frac{\F_5[x,y]}{((x,y)^4,2x^3+xy^2+3y^3,x^3-x^2y+2y^3)} &\isoto \bT_N^\epsilon/5\bT_N^\epsilon, \\  (x,y) &\mapsto (T_2-3, T_7-8).
\end{align*}
Letting $\bar{\m}^\epsilon$ denote the maximal ideal of $\bT_N^\epsilon/5\bT_N^\epsilon$, we see that $(\bar{\m}^\epsilon)^4=0$ but that $(\bar{\m}^\epsilon)^3$ is 2-dimensional, so $\dim_{\F_5} \mathrm{Soc}(\bT_N^\epsilon/5\bT_N^\epsilon) >1$ and $\bT_N^\epsilon$ is not Gorenstein.
\end{eg}

Finally, we give an example illustrating Theorem \ref{thm:main good primes 2}.

\begin{eg}
Let $p=5$, $\ell_0=11$, $\ell_1=41$, so $N=11\cdot 41$, and let $\epsilon=(-1,-1)$. We see that neither of $11$ or $41$ is a $5$-th power modulo the other, so Theorem \ref{thm:main good primes 2} applies. We consider the primes $2,3,7$ and $13$, none of which are congruent to 1 modulo $5$. 
\begin{center}
\begin{tabular}{c|c|c}
$q$ & Is $5$-th power modulo $11$? & Is $5$-th power modulo $41$?\\ \hline 
$2$ & No  & No \\ 
$3$ & No & Yes \\
$7$ & No  & No \\ 
$13$ & No  & No \\ 
\end{tabular}
\end{center}
From this we see that
\[
\det\ttmat{\log_{11}(3)}{\log_{11}(q)}{\log_{41}(3)}{\log_{41}(q)}=\log_{11}(3)\cdot \log_{41}(q) \ne 0.
\]
for any $q\in\{2,7,13\}$. By Theorem \ref{thm:main good primes 2}, $\{T_3-4,T_q-(q+1)\}$ generates $I^\epsilon$ for any $q\in\{2,7,13\}$, and we can see by direct computation that this is true.

More subtly, we can compute that
\[
\det\ttmat{\log_{11}(2)}{\log_{11}(7)}{\log_{41}(2)}{\log_{41}(7)} \ne 0, \quad \det\ttmat{\log_{11}(2)}{\log_{11}(13)}{\log_{41}(2)}{\log_{41}(13)} = 0.
\]
By Theorem \ref{thm:main good primes 2}, this implies that $\{T_2-3,T_7-8\}$ generates $I^\epsilon$, but that $\{T_2-3,T_{13}-14\}$ does not, and we again verify this by direct computation.
\end{eg}

\subsection{Acknowledgements}
We thank Akshay Venkatesh for interesting questions that stimulated this work, and Ken Ribet for his inspiring talk \cite{ribet2015}. We thank Hwajong Yoo for bringing our attention to his work on the subject, and Frank Calegari for clarifying the provenance of Ribet's Conjecture. We thank Shekhar Khare for helpful discussions about the Steinberg condition, Matt Emerton for encouragement and for asking us about implications for newforms, and Rob Pollack for asking us about the case $p \mid N$. 

We thank Jo\"el Bella\"iche, Tobias Berger, Frank Calegari, K\c{e}stutis \v{C}esnavi\v{c}ius, Emmanuel Lecouturier, Barry Mazur, Rob Pollack, Ken Ribet, and Hwajong Yoo for comments on and corrections to an early draft. We also thank the referee for helpful comments and suggestions.

 P.W.\ was supported by the National Science Foundation under the Mathematical Sciences Postdoctoral Research Fellowship No.~1606255, as well as the grants DMS-1638352 and DMS-1901867. C.W.E.\ was supported by Engineering and Physical Sciences Research Council grant EP/L025485/1. 

\subsection{Notation and Conventions} 
\label{subsec:notation}
We let $\partial_{ij}$ denote the Kronecker symbol, which is 1 if $i=j$ and $0$ otherwise.

For each prime $\ell\mid Np$, we fix $G_\ell \subset G_\Q$, a decomposition group at $\ell$, and let $I_\ell \subset G_\ell$ denote the inertia subgroup. We fix elements $\sigma_\ell \in G_{\ell}$ whose image in $G_{\ell}/I_{\ell} \cong \Gal(\overline{\F}_{\ell}/\F_{\ell})$ is the Frobenius. For $\ell \ne p$, we fix elements $\gamma_\ell \in I_{\ell}$ such that the image in the maximal pro-$p$-quotient $I_{\ell}^{\mathrm{pro}-p}$ (which is well-known to be pro-cyclic) is a topological generator. Let $\gamma_p \in \Gal(\overline{\Q}_p/\Q_p^\mathrm{nr}(\zeta_p)) \subset I_p$ be an element such that the image of $\gamma_p$ in $\Gal(\Q_p^\mathrm{nr}(\zeta_p,\sqrt[p]{p})/\Q_p^\mathrm{nr}(\zeta_p))$ is non-trivial. When $\ell = \ell_i$ for $i \in \{0,\dotsc, r\}$ (i.e.\ $\ell \mid N$), we also write $\sigma_i := \sigma_{\ell_i}$ and $\gamma_i := \gamma_{\ell_i}$ for these elements. We write $G_{\Q,S}$ for the Galois group of the maximal extension of $\Q$ unramified outside of the set places $S$ of $\Q$ supporting $Np \infty$, and use the induced maps $G_\ell \ra G_{\Q,S}$. For primes $q \nmid Np$, we write $\Fr_q \in G_{\Q,S}$ for a Frobenius element at $q$. 

As in the theory of representations, Cayley--Hamilton representations, actions on modules, pseudorepresentations, and cochains/cocycles/cohomology of profinite groups $G$ discussed in \cite{WWE4}, these objects and categories are implicitly meant to be continuous without further comment. Here all of the targets are finitely generated $A$-modules for some Noetherian local (continuous) $\Z_p$-algebra $A$ with ideal of definition $I$, and the $I$-adic topology is used on the target. Profinite groups used in the sequel satisfy the $\Phi_p$-finiteness condition (i.e.\ the maximal pro-$p$ quotient of every finite-index subgroup is topologically finitely generated), which allows the theory of \cite{WWE4} to be applied.

We write 
\[
H^i(\Z[1/Np], M)=H^i(C^\bullet(\Z[1/Np],M))=\frac{Z^i(\Z[1/Np],M)}{B^i(\Z[1/Np],M)}
\]
for (continuous) cohomology of a $G_{\Q,S}$-module $M$, together with this notation for cochains, cocycles, and coboundaries. We write $x_1 \smile x_2 \in C^*(\Z[1/Np], M_1 \otimes M_2)$ for the cup product of $x_i \in C^*(\Z[1/Np],M_i)$, and $a_1 \cup a_2 \in H^*(\Z[1/Np], M_1 \otimes M_2)$ for cup product of cohomology classes $a_i \in H^*(\Z[1/Np],M_i)$.

\section{Modular forms}
\label{sec:modular}

In this section, we recall some results about modular curves and modular forms.  Our reference is the paper of Ohta \cite{ohta2014}. 

\subsection{Modular curves, modular forms, and Hecke algebras} The statements given here are all well-known. We review them here to fix notation.

\subsubsection{Modular curves} 
Let $Y_0(N)_{/\Z_p}$ be the coarse moduli space of pairs $(E,C)$, where $E$ is an elliptic curve over $S$ and $C \subset E[N]$ is a finite-flat subgroup scheme of rank $N$ and cyclic (in the sense of Katz-Mazur \cite{KM1985}). Let $X_0(N)_{/\Z_p}$ be the usual compactification of $Y_0(N)_{/\Z_p}$, and let $\{\mathrm{cusps}\}$ denote the complement of $Y_0(N)_{/\Z_p}$ in $X_0(N)_{/\Z_p}$, considered as an effective Cartier divisor on $X_0(N)_{/\Z_p}$. Finally, let $X_0(N)=X_0(N)_{/\Z_p} \otimes \Q_p$.

\subsubsection{Modular forms and Hecke algebras} 
\label{subsubsec:mf}
The map $X_0(N)_{/\Z_p} \to \Spec(\Z_p)$ is known to be LCI, and we let $\Omega$ be the sheaf of regular differentials. Let 
\[
S_2(N;\Z_p) = H^0(X_0(N)_{/\Z_p},\Omega), \quad M_2(N;\Z_p) = H^0(X_0(N)_{/\Z_p},\Omega(\{\mathrm{cusps}\}))
\]
Let $\bT_N'$ and $\bT_N'^0$ be the subalgebras of 
\[
\End_{\Z_p}(M_2(N; \Z_p)), \quad \End_{\Z_p}(S_2(N; \Z_p)),
\]
respectively, generated by the standard Hecke operators $T_n$ with $(N,n)=1$, and all Atkin--Lehner operators $w_\ell$ for $\ell\mid N$ (we do not include any $U_\ell$ for $\ell \mid N$). These are semi-simple commutative $\Z_p$-algebras (see e.g.\ \cite{AL1970}).
\subsubsection{Eisenstein series and Eisenstein parts}
\label{subsec:eisenstein series}
For each $\epsilon \in \{\pm 1\}^{r+1} \setminus \{(1,1,\dots,1)\}$, there is a element $E^\epsilon_{2,N} \in M_2(N;\Z_p)$ that is an eigenform for all $T_n$ with $(N,n)=1$, and has $q$-expansion 
\begin{equation}
\label{eq:eisenstein series}
E^\epsilon_{2,N} = -\frac{1}{24} \prod_{i=0}^r (\epsilon_i\ell_i +1)+ \sum_{n=1}^\infty a_n q^n
\end{equation}
where $a_n = \sum_{0<d\mid n} t$ when $\gcd(n,N)=1$ (in particular, $a_1=1$), and $w_{\ell_i}E^\epsilon_{2,N} = \epsilon_i E^\epsilon_{2,N}$ (see \cite[Lem.\ 2.3.4]{ohta2014}).

Let ${I'}^\epsilon=\Ann_{\bT_N'}(E^\epsilon_{2,N})$, and let $\bT_N^\epsilon$ be the completion of ${\bT_N'}$ at the maximal ideal $({I'}^\epsilon,p)$, and let $\bT_N^{0,\epsilon}=\bT_N'^0\otimes_{\bT_N'}\bT_N^\epsilon$. Let $I^\epsilon = {I'}^\epsilon\bT_N^\epsilon$ and let $I^{0,\epsilon}$ be the image of $I^\epsilon$ in $\bT_N^{0,\epsilon}$.  For a $\bT_N'$-module $M$, let $M^\epsilon_\mathrm{Eis}=M \otimes_{\bT_N'} \bT_N^\epsilon$. The map $\bT_N^\epsilon \rsurj \Z_p$ induced by $E^\epsilon_{2,N}$ is a surjective ring homomorphism with kernel $I^\epsilon$. We refer to this as the \emph{augmentation map} for $\bT_N^\epsilon$.

Note that we have $w_{\ell_i} = \epsilon_i$ as elements of $\bT_N^\epsilon$. Indeed, this follows from $w_{\ell_i}^2=1$, $w_{\ell_i} - \epsilon_i \in I^\epsilon$, and $p \ne 2$: consider $(w_{\ell_i}-\epsilon_i)(w_{\ell_i}+\epsilon_i)=0$ and observe that $w_{\ell_i}+\epsilon_i \in (\bT_N^\epsilon)^\times$. Consequently, $\bT_N^\epsilon$ is generated as a $\Z_p$-algebra by $T_q$ for $q \nmid N$.

If $p\nmid N$, let $U_p \in \bT_N^\epsilon$ denote the unit root of the polynomial
\[
X^2-T_pX+p=0,
\]
which exists and is unique by Hensel's lemma. Since $T_p-(p+1) \in I^\epsilon$, we see that $U_p - 1 \in I^\epsilon$. Moreover, we see that $T_p = U_p + p U_p^{-1}$.

\subsubsection{Duality}
As in  \cite[Thm.\ 2.4.6]{ohta2014}, there are perfect pairings of free $\Z_p$-modules
\begin{equation}
\label{eq:M and T duality}
M_2(N;\Z_p)^\epsilon_\mathrm{Eis} \times \bT_N^\epsilon \lra \Z_p, \quad S_2(N;\Z_p)^\epsilon_\mathrm{Eis} \times \bT_N^{0,\epsilon} \lra \Z_p
\end{equation}
given by $(f,t) \mapsto a_1(t\cdot f)$, where $a_1(-)$ refers to the coefficient of $q$ in the $q$-expansion. In particular, $M_2(N;\Z_p)^\epsilon_\mathrm{Eis}$ (resp.\ $S_2(N;\Z_p)_\mathrm{Eis}^\epsilon$) is a dualizing $\bT_N^\epsilon$-module (resp.\ $\bT_N^{0,\epsilon}$-module). 

\subsubsection{Oldforms and stabilizations}
\label{subsub:stabilizations} 
If $\ell \mid N$ is a prime and $f \in S_2(N/\ell; \Z_p)$ is an eigenform for all $T_n$ with $(n,N/\ell)=1$, then the subspace
\[
\{g \in S_2(N;\Z_p) \ : \ a_n(g)=a_n(f) \text{ for all } (n,N/\ell)=1\}
\]
has rank two, with basis $f(z),f(\ell z)$. If we let $f_\pm(z) = f(z) \pm \ell f(\ell z)$, then $w_\ell f_\pm(z)=\pm f_\pm (z)$. Note that, since $p \ne 2$, we have $f_+ \not \equiv f_- \pmod{p}$. In particular, if $\epsilon'\in \{\pm 1\}^r$  is the tuple obtained from $\epsilon$ by deleting the entry corresponding to $\ell$, then there are injective homomorphisms given by $f \mapsto f_{\epsilon_\ell}$, 
\[
M_2(N/\ell;\Z_p)^{\epsilon'}_\mathrm{Eis} \rinj M_2(N;\Z_p)^{\epsilon}_\mathrm{Eis} \quad \text{and} \quad S_2(N/\ell;\Z_p)^{\epsilon'}_\mathrm{Eis} \rinj S_2(N;\Z_p)^{\epsilon}_\mathrm{Eis}.
\]

\subsection{Congruence number} We recall this theorem of Ohta, and related results.

\begin{thm}[Ohta]
\label{thm:congruence number}
There is an isomorphism $\bT_N^{\epsilon,0}/I^{\epsilon,0} \cong \Z_p/a_0(E^\epsilon_{2,N})\Z_p$.
\end{thm}

This is \cite[Thm.\ 3.1.3]{ohta2014}. His method of proof actually can be used to give the  following stronger result, exactly as in \cite[Lem.\ 3.2.2]{WWE3}. See Lemma \ref{lem:fiber prods} for a discussion of fiber products of rings.

\begin{lem}
\label{lem:bT is a pull-back}
The composition of the augmentation map $\bT_N^\epsilon \to \Z_p$ with the quotient map $\Z_p \to \Z_p/a_0(E^\epsilon_{2,N})\Z_p$ factors through $\bT_N^{0,\epsilon}$ and induces an isomorphism
\[
\bT_N^\epsilon \lrisom  \bT_N^{0,\epsilon} \times_{\Z_p/a_0(E^\epsilon_{2,N})\Z_p} \Z_p.
\]
In particular, $\ker(\bT_N^{\epsilon} \to \bT_N^{0,\epsilon})=\Ann_{\bT_N^\epsilon}(I^\epsilon)$.
\end{lem}

\subsection{Eigenforms and associated Galois representations}

Let $\nu: \bT_N^{0,\epsilon} \rinj \tilde{\bT}_N^{0,\epsilon}$ denote the normalization of $\bT^{0, \epsilon}_N$. 
\begin{lem}
\label{lem:normalization of bT0}
We record facts about $\tilde{\bT}_N^{0,\epsilon}$ and associated Galois representations. 
\begin{enumerate}[leftmargin=2em]
\item Letting $q$ vary over primes $q \nmid Np$, there is an isomorphism 
\[
h: \tilde{\bT}_N^{0,\epsilon} \lrisom \bigoplus_{f\in \Sigma} \sO_f, \quad \nu(T_q) \mapsto (a_q(f))_{f \in \Sigma}, 
\]
where $\Sigma \subset S_2(N;\overline{\Q}_p)^\epsilon_\mathrm{Eis}$ is the set of normalized eigenforms, and $\sO_f$ is the valuation ring of the finite extension $\Q_p(a_q(f)_{q \nmid Np})/\Q_p$. 

\item For each $f \in \Sigma$, there is an absolutely irreducible representation $\rho_f:G_{\Q,S} \to \GL_2(\sO_f[1/p])$ such that the characteristic polynomial of $\rho_f(\Fr_q)$ is $X^2-a_q(f) X +q$ for any $q \nmid Np$.

\item Assume $\ell_i \neq p$. The representation $\rho_f|_{G_{\ell_i}}$ is unramified if $f$ is old at $\ell_i$. Otherwise, $f$ is new at $\ell_i$ and there is an isomorphism
\begin{equation}
\label{eq:Steinberg G_ell rep}
\rho_f|_{G_{\ell_i}} \simeq \ttmat{\lambda(a_{\ell_i}(f))\kcyc}{*}{0}{\lambda(a_{\ell_i}(f))},
\end{equation}
where $a_{\ell_i}(f)=-\epsilon_i$.
\item There is an isomorphism 
\begin{equation}
\label{eq:Steinberg G_p rep}
\rho_f|_{G_{p}} \simeq \ttmat{\lambda(a_p(f)^{-1})\kcyc}{*}{0}{\lambda(a_p(f))}.
\end{equation}
Moreover, 
\begin{enumerate}
\item $\rho_f\vert_{G_p}$ is finite-flat if and only if either 
\begin{enumerate}
	\item $p \nmid N$, in which case $h: \nu(U_p) \mapsto (a_p(f))_{f \in \Sigma}$, or
	\item $p \mid N$ and $f$ is old at $p$.
\end{enumerate}
\item If $p \mid N$ and $f$ is new at $p$, then $a_p(f) = -\epsilon_p=+1$, i.e.\ $\epsilon_p = -1$.
\end{enumerate}

\end{enumerate}
\end{lem}
\begin{proof}
For (1)-(3) and (4a) see, for example, \cite[Thm.\ 3.1]{DDT1994}. In (4b), the fact that $a_p(f) = -\epsilon_p$ is \cite[Thm.~3]{AL1970}. To see that $\epsilon_p=-1$, note that the semi-simple residual representation $\bar{\rho}^\mathrm{ss}_f$ is $\omega \oplus 1$, but \eqref{eq:Steinberg G_p rep} implies $\bar{\rho}^\mathrm{ss}_f|_{G_p} = \lambda(-\epsilon_p)\omega \oplus \lambda(-\epsilon_p)$. Since $\omega|_{G_p}$ is ramified, this implies that $\lambda(-\epsilon_p)=1$, so $\epsilon_p=-1$.
\end{proof}

Combining Lemmas \ref{lem:bT is a pull-back} and \ref{lem:normalization of bT0}, we obtain an injective homomorphism
\begin{equation}
\label{eq:bT in product}
\bT_N^\epsilon \to \Z_p \oplus \bT_N^{0,\epsilon} \to \Z_p \oplus \bigoplus_{f\in \Sigma} \sO_f
\end{equation}
determined by sending $T_q$ to $(q+1,a_q(f)_{f \in \Sigma})$ for $q \nmid Np$ and, if $p\nmid N$, sending $U_p$ to $(1,a_p(f)_{f \in \Sigma})$.

\subsection{The kernel of $\m^\epsilon$ on the modular Jacobian and the Gorenstein condition}
\label{subsec:gorenstein and jacobian} 

In this section, we use some results of Ohta (following ideas of Mazur) to relate the structure of the rings $\bT_N^{\epsilon}$ and $\bT_N^{0,\epsilon}$ to the geometry of the N\'{e}ron model $J_0(N)_{/\Z_p}$ of the Jacobian of $X_0(N)$. Let $J_0(N) = J_0(N)_{/\Z_p} \otimes \Q_p$.

For a $\Z_p$-module $M$, let $\mathrm{Ta}_p(M)=\Hom(\Q_p/\Z_p,M)$ be the Tate module of $M$, let $M^*=\Hom_{\Z_p}(M,\Q_p/\Z_p)$ be the Pontrjagin dual, and let $M^\vee=\Hom_{\Z_p}(M,\Z_p)$ be the $\Z_p$-dual. If $M$ is $p$-divisible, then there is an identification $M^*\cong \mathrm{Ta}_p(M)^\vee$.

Let $\cT=H^1_{{\mathrm{\acute{e}t}}}(X_0(N)_{\overline{\Q}},\Z_p(1)) \cong \mathrm{Ta}_p(J_0(N)(\overline{\Q}_p))$. 
\begin{lem}
\label{lem:failure of mult one = gorenstein defect}
There is an exact sequence of $\bT_N^{0,\epsilon}[I_p]$-modules
\[
0 \lra \bT_N^{0,\epsilon}(1) \lra \cT_\mathrm{Eis}^\epsilon \lra (\bT_N^{0,\epsilon})^\vee \lra 0.
\]
The sequence splits as $\bT_N^{0,\epsilon}$-modules. In particular, we have
\[
\dim_{\F_p} J_0(N)[\m^\epsilon](\overline{\Q}_p) = \dim_{\F_p}( \cT/{\m^\epsilon}\cT) = 2+ \delta(\bT_N^{0,\epsilon})
\]
where $\delta(\bT_N^{0,\epsilon})$ is the Gorenstein defect of $\bT_N^{0,\epsilon}$. (See \S\ref{subsec:Gorenstein defect} for a discussion of Gorenstein defect.)
\end{lem}
\begin{proof}
Ohta has shown in \cite[Prop.\ 3.5.4 and Prop.\ 3.5.9]{ohta2014} that
\[
\dim_{\F_p} J_0(N)_{/\Z_p}(\overline{\F}_p)[\m^\epsilon] \le 1.
\]
This implies the result, following \cite[\S\S II.7-II.8]{mazur1978} (see also \cite{mazur1997}). 
\end{proof}

\begin{lem}
\label{lem:goren and I principal}
Suppose that $\bT_N^{\epsilon}$ is Gorenstein. Then there is an isomorphism of $\bT_N^\epsilon$-modules
\[
I^\epsilon \lrisom (\bT_N^{0,\epsilon})^\vee.
\]
In particular, the minimal number of generators of $I^\epsilon$ is $\delta(\bT_N^{0,\epsilon})+1$, and $I^\epsilon$ is principal if and only if $\bT_N^{0,\epsilon}$ is Gorenstein.
\end{lem}
\begin{proof}
Like the proof of \cite[Lem.\ 3.2.5]{ohta2014}, there is an exact sequence of $\bT_N^\epsilon$-modules 
\[
0 \lra S_2(N;\Z_p)^\epsilon_\mathrm{Eis} \lra M_2(N;\Z_p)^\epsilon_\mathrm{Eis} \xrightarrow{\mathrm{Res}} \Z_p \lra 0
\]
where $\bT_N^\epsilon$ acts on $\Z_p$ via the augmentation map $\bT_N^\epsilon \to \bT_N^\epsilon/I^\epsilon=\Z_p$. Since we assume that $\bT_N^{\epsilon}$ is Gorenstein, we see by the duality \eqref{eq:M and T duality} that $M_2(N;\Z_p)^\epsilon_\mathrm{Eis}$ is a free $\bT_N^\epsilon$-module of rank $1$. We may choose a generator $f$ of $M_2(N;\Z_p)^\epsilon_\mathrm{Eis}$ such that $\mathrm{Res}(f)=1$. Then we obtain a surjective $\bT_N^\epsilon$-module homomorphism
\[
\bT_N^\epsilon \rsurj \Z_p, \quad T \mapsto \mathrm{Res}(Tf)
\]
whose kernel is isomorphic to $S_2(N;\Z_p)^\epsilon_\mathrm{Eis}$. This is a $\bT_N^\epsilon$-module homomorphism that sends $1$ to $1$, so it is the augmentation map $\bT_N^\epsilon \rsurj \Z_p$. Thus $I^\epsilon \cong S_2(N;\Z_p)^\epsilon_\mathrm{Eis}$, so duality \eqref{eq:M and T duality} yields the isomorphism of the lemma. The remaining parts follow from \S \ref{subsec:Gorenstein defect}.
\end{proof}

Combining the preceding two lemmas, we obtain the following
\begin{lem}
\label{lem:I and Jacobian kernel}
Assume that $\bT_N^\epsilon$ is Gorenstein. Then
\[
\dim_{\F_p}J_0(N)[\m^\epsilon](\overline{\Q}_p) = 1 + \dim_{\F_p}(I^\epsilon/\m^\epsilon I^\epsilon).
\]
\end{lem}

\section{The pseudodeformation ring}
\label{sec:pseudodef ring}

In this section, we set up the deformation theory of Galois pseudorepresentations modeling those that arise from Hecke eigenforms of weight 2 and level $N$ that are congruent to the Eisenstein series $E^\epsilon_{2,N}$. These are the Galois representations of Lemma \ref{lem:normalization of bT0}. See \S\ref{subsec:psdef method} for further introduction. 

\subsection{Theory of Cayley--Hamilton representations} 
\label{subsec:CH setup} 

This section is a summary of \cite{WWE4}. Only for this section, we work with a general profinite group $G$ satisfying condition $\Phi_p$ (of \S\ref{subsec:notation}). All pseudorepresentations are assumed to have dimension $2$, for simplicity. 

\subsubsection{Pseudorepresentations}
A pseudorepresentation $D:E \to A$ is the data of an associative $A$-algebra $E$ along with a homogeneous multiplicative polynomial law $D$ from $E$ to $A$. This definition is due to Chenevier \cite{chen2014}; see \cite{WWE4} and the references therein. Despite the notation, the pseudorepresentation $D$ includes the data of a multiplicative function $D:E\to A$, but is not characterized by this function alone. It is characterized by the pair of functions $\Tr_D,D:E \to A$, where $\Tr_D$ is defined by the \emph{characteristic polynomial}:
\begin{equation}
\label{eq:x+1 identity}
D(x-t)=t^2 - \Tr_D(x)t + D(x) \in A[t].
\end{equation}
A pseudorepresentation $D:E \to A$ is said to be \emph{Cayley--Hamilton} if, for every commutative $A$-algebra $B$, every element $x\in E \otimes_A B$ satisfies its characteristic polynomial. We also denote by $D : G \ra A$ a pseudorepresentation $D : A[G] \ra A$. 

\subsubsection{Cayley--Hamilton representations}
In the category of \emph{Cayley--Hamilton representations of a profinite group $G$}, an object is a triple
\[
(\rho: G \ra E^\times, E, D : E\ra A),
\]
and sometimes referred to more briefly as ``$\rho$.'' Here $\rho$ is a homomorphism (continuous, as always), $E$ is an associative $A$-algebra that is finitely generated as an $A$-module, $(A,\m_A)$ is a Noetherian local $\Z_p$-algebra, and $D$ is a Cayley--Hamilton pseudorepresentation. We call $A$ the \emph{scalar ring} of $E$. The \emph{induced pseudorepresentation of $\rho$} is $D \circ \rho : G \ra A$, also denoted $\psi(\rho)$. The functor $\psi$ is essentially surjective. The Cayley--Hamilton representation $\rho$ is said to be \emph{over} $\psi(\rho)\otimes_A A/\m_A$, and $\psi(\rho)$ is said to be a \emph{pseudodeformation} of $\psi(\rho)\otimes_A A/\m_A$. If $(\rho,E,D)$ is a Cayley--Hamilton representation of $G$ and $x \in A[G]$, then we abuse notation and write $D(x)$ for $D(\rho(x))$ (where we also abuse notation and write $\rho:A[G] \to E$ for the linearization of $\rho$).

Given a pseudorepresentation $\Db:G \to \F$ for a field $\F$, there is a universal object in the category of Cayley--Hamilton representations over $\Db$. This is denoted by 
\[
(\rho^u_\Db : G \lra (E^u_\Db)^\times, E^u_\Db, D_{E^u_\Db} : E^u_\Db \ra R^u_\Db), 
\]
and the induced pseudorepresentation $D^u_\Db := \psi(\rho^u_\Db)$ is the universal pseudodeformation of $\Db$.

\subsubsection{Generalized matrix algebras (GMA)} An important example of a Cayley--Hamilton algebra is a \emph{generalized matrix algebra (GMA)}. An $A$-GMA $E$ is given by the data $(B,C,m)$ where $B$ and $C$ are finitely-generated $R$-modules, $m:B \otimes_R C \to R$ is an $R$-module homomorphism satisfying certain conditions, and $E =\sm{R}{B}{C}{R}$ (see \cite[Example 3.1.7]{WWE4}). There is a Cayley--Hamilton pseudorepresentation $D:E \to A$ given by the usual formula for characteristic polynomial. We write a homomorphism $\rho: G \to E^\times$ as $\rho = \sm{\rho_{1,1}}{\rho_{1,2}}{\rho_{2,1}}{\rho_{2,2}}$.

If $\Db$ is multiplicity-free (see \cite[Defn.\ 3.2.1]{WWE4}), then $E^u_{\Db}$ has a GMA structure whose associated pseudorepresentation is $D_{E^u_\Db}$ \cite[Thm.\ 3.2.2]{WWE4}.

\subsubsection{Reducibility}
We will refer to the condition that a Cayley--Hamilton representation or a pseudorepresentation is \emph{reducible}. We also refer to the \emph{reducibility ideal} in rings receiving a pseudorepresentations. For these definitions, see \cite[\S4.2]{WWE4} or \cite[\S5.7]{WWE1}. The important case for this paper is that, if $(\rho, E, D:E\to A)$ is a Cayley--Hamilton representation where $E$ is the GMA associated to $(B,C,m)$, then the reducibility ideal of $D$ is the image of $m$. There are also universal objects, denoted $\rho^\red$, etc. 

\subsubsection{Conditions on Cayley--Hamilton representations}
\label{sssec:cond on CH}
We consider two flavors of conditions $\cP$ imposed on Cayley--Hamilton representations of $G$:
\begin{enumerate}
	\item $\cP$ is a condition that certain elements vanish, e.g.\ Definition \ref{defn:US-CH ell}. 
	\item $\cP$ is a property applying to finite-length $\Z_p[G]$-modules and satisfying a basic stability condition, e.g.\ \S\ref{subsec:flat case}.
\end{enumerate}

In case (1), one produces a universal Cayley--Hamilton $\rho_\Db^\cP$ representation of $G$ satisfying $\cP$ by taking the quotient by the two-sided ideal of $E_\Db$ generated by the relevant elements, and then taking a further quotient so that a pseudorepresentation exists. This final quotient is known as the \emph{Cayley--Hamilton quotient of $\rho^u_\Db$ for $\cP$}. See \cite[Defn.\ 2.4.7]{WWE4} for details; cf.\ also \cite[Defn.\ 5.9.5]{WWE1}. 

In case (2), we consider $E^u_\Db$ as a $G$-module using its left action on itself by multiplication, and find in \cite[\S2.4]{WWE4} that the maximal left quotient module satisfying $\cP$ can be defined and is an algebra quotient. The subsequent Cayley--Hamilton quotient is then shown to satisfy the desired properties of $\rho_\Db^\cP$.

\subsubsection{Conditions on pseudorepresentations}

As discussed in \cite[\S2.5]{WWE4}, one says that a pseudorepresentation $D$ of $G$ satisfies $\cP$ if there exists a Cayley--Hamilton representation $\rho$ of $G$ such that $\psi(\rho) = D$ and $\rho$ satisfies $\cP$. Then the universal pseudodeformation of $\Db$ with property $\cP$ turns out to be $\psi(\rho^\cP_\Db)$.

\subsection{Universal Cayley--Hamilton representations of Galois groups}
\label{subsec:CH Gal setup}
Let $\ell \mid Np$ be a prime. Recall from \S\ref{subsec:notation} the decomposition group $G_\ell \ra G_{\Q,S}$. Let $\Db: G_{\Q,S} \to \F_p$ denote the pseudorepresentation $\psi(\F_p(1)\oplus \F_p)$. 

We denote by 
\[
(\rho_\Db : G_{\Q,S} \lra E_\Db^\times, E_\Db, D_{E_\Db} : E_\Db \ra R_\Db)
\]
the universal Cayley--Hamilton representation of $G_{\Q,S}$ over $\Db$. The scalar ring $R_\Db$ is the universal pseudodeformation ring of $\Db$, with universal pseudorepresentation $D_\Db := \psi(\rho_\Db)$. Similarly, we let the triple
\[
(\rho_\ell : G_\ell \ra E_\ell^\times, E_\ell, D_{E_\ell} : E_\ell \ra R_\ell)
\]
denote the universal Cayley--Hamilton representation of $G_\ell$ over $\Db\vert_{G_\ell}$, so that $D_\ell := \psi(\rho_\ell) : G_\ell \ra R_\ell$ is the universal pseudodeformation of $\Db\vert_{G_\ell}$. 

\begin{defn}
\label{def:GMA structure}
Note that $\Db$ is multiplicity-free, and that, if $\ell \not \equiv 1 \pmod{p}$, then $\Db|_{G_\ell}$ is multiplicity-free. In this case, $E_\ell$ and $E_\Db$ have the structure of a GMA. In this paper, whenever we fix such a structure, we assume that $(\rho_\ell)_{1,1} \otimes_{R_\ell} \F_p \cong \omega|_{G_\ell}$ (resp.\ $(\rho_\Db)_{1,1}\otimes_{R_\Db} \F_p \cong \omega$).
\end{defn}

\subsection{Case $\ell \nmid Np$: unramified} 
\label{subsec:unram}
For $\ell \nmid Np$, we want Galois representations to be unramified at $\ell$. We impose this by considering representations of $G_{\Q,S}$, as opposed to $\Gal(\overline{\Q}/\Q)$. 

\subsection{Case $\ell \neq p$ and $\ell \mid N$: the unramified-or-Steinberg condition}
\label{subsec:US ell}

In this subsection, we write $\ell$ for one of the factors of $N$ referred to elsewhere in this manuscript as $\ell_i$. Likewise, we write $\epsilon_\ell$ for $\epsilon_i$. 

\begin{defn}
\label{defn:US-CH ell}
Let $(\rho : G_\ell \ra E, E, D_E: E \ra A)$ be a Cayley--Hamilton representation of $G_\ell$ over $\Db\vert_{G_\ell}$. We call $\rho$ \emph{unramified-or-$\epsilon_\ell$-Steinberg} (or $\mathrm{US}^{\epsilon_\ell}_\ell$) if 
\begin{equation}
\label{eq:US test elts}
V^{\epsilon_\ell}_\rho(\sigma,\tau) := (\rho(\sigma)-\lambda(-\epsilon_\ell)(\sigma)\kcyc(\sigma))(\rho(\tau)-\lambda(-\epsilon_\ell)(\tau)) \in E
\end{equation}
is equal to $0$ for all $(\sigma, \tau)$ ranging over the set 
\[
I_{\ell} \times G_{\ell}  \cup G_{\ell} \times I_{\ell} \ \subset \ G_{\ell} \times G_{\ell}.
\]
Write $V^{\epsilon_\ell}_\rho$ for the set of all elements $V^{\epsilon_\ell}_\rho(\sigma,\tau)$ over this range.

A pseudodeformation $D : G_\ell \ra A$ of $\Db\vert_{G_\ell}$ is called $\mathrm{US}^\epsilon_\ell$ if there exists a $\mathrm{US}^\epsilon_\ell$ Cayley--Hamilton representation $\rho$ of $G_\ell$ such that $\psi(\rho) = D$.
\end{defn}

\begin{defn}
\label{defn:US-CH ell UO}
Let $(E^{\epsilon_\ell}_\ell, D_{E^{\epsilon_\ell}_\ell}: E^{\epsilon_\ell}_\ell \ra R^{\epsilon_\ell}_\ell)$ be the Cayley--Hamilton quotient of $(E_\ell, D_\ell)$ by $V^{\epsilon_\ell}_{\rho_\ell}$. Let
\[
(\rho^{\epsilon_\ell}_\ell : G_\ell \ra (E^{\epsilon_\ell}_\ell)^\times, E^{\epsilon_\ell}_\ell, D_{E^{\epsilon_\ell}_\ell} : E^{\epsilon_\ell}_\ell \ra R^{\epsilon_\ell}_\ell),
\]
be the corresponding Cayley--Hamilton representation, with induced pseudorepresentation of $G_\ell$ denoted $D^{\epsilon_\ell}_\ell := \psi(\rho^{\epsilon_\ell}_\ell) : G_\ell \ra R^{\epsilon_\ell}_\ell$. 
\end{defn}

By the theory of \S\ref{sssec:cond on CH}, $\rho^{\epsilon_\ell}_\ell$ is the universal $\mathrm{US}_\ell^{\epsilon_\ell}$ Cayley--Hamilton representation over $\Db\vert_{G_\ell}$, and $D^{\epsilon_\ell}_\ell$ is the universal $\mathrm{US}_\ell^{\epsilon_\ell}$ pseudodeformation of $\Db\vert_{G_\ell}$. 

\begin{lem}
	\label{lem:USell implies pseudo-unram}
If $\ell \neq p$, then, for any $\epsilon_\ell$, we have $D_\ell^{\epsilon_\ell}(\tau)=1$ and $\Tr_{D_\ell^{\epsilon_\ell}}(\tau)=2$ for all $\tau \in I_\ell$. That is, $(D^{\epsilon_\ell}_\ell)\vert_{I_\ell} = \psi(1 \oplus 1)$. 
\end{lem}
\begin{proof}
	Let $\tau \in I_\ell$. We see in \eqref{eq:US test elts} that $V^{\epsilon_\ell}_{\rho^{\epsilon_\ell}_\ell}(\tau,\tau)=(\rho^{\epsilon_\ell}_\ell(\tau)-1)^2 = 0$. Thus by \cite[Lem.\ 2.7(iv)]{chen2014}, we see $\Tr_{D^{\epsilon_\ell}_\ell}(\tau-1) = D^{\epsilon_\ell}_\ell(\tau-1) = 0$. As traces are additive, we have $\Tr_{D^{\epsilon_\ell}_\ell}(\tau) = \Tr_{D^{\epsilon_\ell}_\ell}(1)=2$. Applying \eqref{eq:x+1 identity} with $x = \tau$ and using the naturality of $D^{\epsilon_\ell}_\ell$ with respect to the morphism $R_\ell^{\epsilon_\ell}[t] \to R_\ell^{\epsilon_\ell}$ given by $t \mapsto 1$, we find that $D^{\epsilon_\ell}_\ell(\tau) = 1$. 
\end{proof}

\begin{lem}
\label{lem:epsilon=1 and unram}
Suppose that $\epsilon_\ell = +1$ and $\ell \not \equiv -1,0 \pmod{p}$. Then $\rho_\ell^{\epsilon_\ell}$ is unramified (i.e. $\rho_\ell^{\epsilon_\ell}|_{I_{\ell}} =1$).
\end{lem}

\begin{proof}
Let $\sigma \in G_\ell$ be the element $\sigma_\ell$ defined in \S\ref{subsec:notation}. By definition of $E_\ell^{\epsilon_\ell}$, 
\[
V^{\epsilon_\ell}_{\rho^{\epsilon_\ell}_\ell}(\tau,\sigma) = (\rho_\ell^{\epsilon_\ell}(\tau)-1)(\rho_\ell^{\epsilon_\ell}(\sigma)+1)= 0,
\]
for any $\tau \in I_\ell$. To prove the lemma, it suffices to show that $(\rho_\ell^{\epsilon_\ell}(\sigma)+1) \in (E_\ell^{\epsilon_\ell})^\times$. 

By the Cayley--Hamilton property, we know that any element $x \in E_\ell^{\epsilon_\ell}$ satisfies  $x^2-\Tr_{D^{\epsilon_\ell}_\ell}(x)x+D^{\epsilon_\ell}_\ell(x)=0$. In particular, we see that $x \in (E_\ell^{\epsilon_\ell})^\times$ if $D^{\epsilon_\ell}_\ell(x) \in (R_\ell^{\epsilon_\ell})^\times$. Hence it will suffice to show that $D^{\epsilon_\ell}_\ell(\sigma + 1) \in (R_\ell^{\epsilon_\ell})^\times$. 

Writing $\m \subset R^{\epsilon_\ell}_\ell$ for the maximal ideal, we know that $D^{\epsilon_\ell}_\ell \equiv \Db \pmod{\m}$, so it will suffice to show that $\Db(\sigma+1) \in \F_p^\times$. 
Because $\ell \neq p$ and $\Db = \psi(\omega \oplus 1)$, we apply \eqref{eq:x+1 identity} with $x = \sigma$ and $t=-1$, calculating that $\Db(\sigma + 1)= 2(\ell+1) \in \F_p$. This is a unit because $p$ is odd and $\ell \not\equiv -1 \pmod{p}$. 
\end{proof}

\subsection{The finite-flat case: $\ell = p$ and $p \nmid N$}
\label{subsec:flat case}

A finite-length $\Z_p[G_p]$-module $V$ is said to be \emph{finite-flat} when it arises as $\cG(\overline{\Q}_p)$, where $\cG$ is a finite flat group scheme over $\Z_p$. In \cite[\S5.2]{WWE4} we check that the theory of \S\ref{sssec:cond on CH} can be applied to the finite-flat condition. This theory gives us 
\[
(\rho_p^\fl: G_p \to (E_p^\fl)^\times, E_p^\fl, D_{E^\fl_p} : E_p^\fl \ra R_p^\fl),
\]
the universal finite-flat Cayley--Hamilton representation of $G_p$ over $\Db \vert_{G_p}$. The pseudorepresentation $D^\fl_p := \psi(\rho^\fl_p) : G_p \ra R^\fl_p$ is the universal finite-flat pseudodeformation of $\Db \vert_{G_p}$.

Consider a GMA structure on $E_p^\fl$ as in Definition \ref{def:GMA structure}, which we write as
\[
\rho_p^\fl = \ttmat{\rho^\fl_{p,1,1}}{\rho^\fl_{p,1,2}}{\rho^\fl_{p,2,1}}{\rho^\fl_{p,2,2}} : G_p \lra 
\ttmat{R_p^\fl}{B_p^\fl}{C_p^\fl}{R_p^\fl}^\times.
\]

\begin{lem}
\label{lem:flat implies upper tri}
For any such GMA structure on $E_p$, $C_p^\fl = 0$. 
\end{lem}

\begin{proof}
The proof is implicit in \cite{WWE3} but not stated in this form there. One simply combines the following facts. See \cite[\S B.4]{WWE3} for the notation. 
\begin{itemize}
\item As the maximal ideal of $R_p^\fl$ contains the reducibility ideal, we have $\Hom_{R_p^\fl}(C_p^\fl, \F_p) =\Ext^1_{\mathrm{ffgs}/\Z_p}(\mu_p,\Z/p\Z)$, where $\mathrm{ffgs}/\Z_p$ is the category of finite flat groups schemes over $\Z_p$, by \cite[Thm.\ 4.3.5]{WWE4}. 
\item We see in \cite[Lem.\ 6.2.1(1)]{WWE3}  that $\Ext^1_{\mathrm{ffgs}/\Z_p}(\mu_p,\Z/p\Z) = 0$. 
\end{itemize}
As $C_p^\fl$ is a finitely-generated $R_p^\fl$-module, this implies that $C_p^\fl = 0$. 
\end{proof}

Now that we know that $C_p^\fl = 0$, $\rho_{p,i,i}^\fl$ are $R_p^\fl$-valued characters of $G_p$, for $i = 1,2$. Similarly to \cite[\S5.1]{WWE3}, using the fact that $\omega\vert_{G_p} \neq 1$, we see the following 
\begin{lem}
\label{lem:flat psrep form}
A pseudodeformation $D$ of $\Db|_{G_p}$ is finite-flat if and only if $D=\psi(\kcyc \chi_1 \oplus \chi_2)$ where $\chi_1,\chi_2$ are unramified deformations of the trivial character. 
\end{lem}

\subsection{The finite-flat case: $\ell = p$, $p \mid N$, and $\epsilon_p=+1$}
\label{subsec:epsilon_p=1}
By Lemma \ref{lem:normalization of bT0}(4), we see that, if $\epsilon_p=+1$, then the residually Eisenstein cusp forms are old at $p$ with associated $G_{\Q,S}$-representation being finite-flat at $p$. We impose this condition exactly as in \S\ref{subsec:flat case}. Namely, we say that a Cayley--Hamilton representation of $G_p$ is \emph{unramified-or-$(+1)$-Steinberg} (or $\mathrm{US}_p^{+1}$) if it is finite-flat.

\subsection{The ordinary case: $\ell = p$, $p \mid N$, and $\epsilon_p=-1$} 
\label{subsec:ord case}

Based on the form of Galois representations arising from $p$-ordinary eigenforms given in Lemma \ref{lem:normalization of bT0}(4), we proceed exactly as in the case $\ell \neq p$ given in \S\ref{subsec:US ell}. 
\begin{defn}
	\label{defn:US-CH p}
We say that a Cayley--Hamilton representation or a pseudodeformation over $\Db\vert_{G_p}$ is \emph{ordinary} (or $\mathrm{US}_p^{-1}$) when it satisfies Definition \ref{defn:US-CH ell}, simply letting $\ell = p$. 
\end{defn}

Similarly to Definition \ref{defn:US-CH ell UO}, let $(E^\ord_p, D_{E^\ord_p})$ be the Cayley--Hamilton quotient of $(E_p,D_{E_p})$ by $V_{\rho_p}^{-1}$, and let $(\rho^\ord_p, E^\ord_p, D_{E^\ord_p} : E^\ord_p \ra R^\ord_p)$ be the corresponding Cayley--Hamilton representation. As per \S\ref{sssec:cond on CH}, $\rho^\ord_p$ is the universal ordinary Cayley--Hamilton representation over $\Db\vert_{G_p}$, and $D^\ord_p := \psi(\rho^\ord_p) : G_p \ra R^\ord_p$ is the universal ordinary pseudodeformation of $\Db\vert_{G_p}$.

\begin{rem}
If one applies $V^{+1}_{\rho_p} = 0$ in the case $\epsilon_p = +1$, one does not get the the desired finite-flat condition of \S \ref{subsec:epsilon_p=1} that agrees with Lemma \ref{lem:normalization of bT0}(4b). Instead, one finds that $E_p^{+1} = 0$ (i.e.\ no deformations of $\Db$ satisfy this condition). 
\end{rem}

We set up the following notation, which includes all cases: $\epsilon_p=\pm 1$ or $p \nmid N$. 
\begin{defn}
\label{defn:E_p^epsilon}
For any $N$ and $\epsilon$, we establish notation 
\[
	(\rho^{\epsilon_p}_p, E^{\epsilon_p}_p, D_{E^{\epsilon_p}_p}, R^{\epsilon_p}_p, D^{\epsilon_p}_p) := \left\{\begin{array}{ll}
	(\rho^\ord_p, E^\ord_p, D_{E^\ord_p}, R^\ord_p, D^\ord_p) & \text{if } p \mid N, \epsilon_p=-1, \\
	(\rho^\fl_p, E^\fl_p, D_{E^\fl_p}, R^\fl_p, D^\fl_p) & \text{otherwise}.
	\end{array}\right.
\]
\end{defn}

In \cite[\S5]{WWE1}, we developed an alternative definition of ordinary Cayley--Hamilton algebra. (This definition applies to general weight, which we specialize to weight 2 here.) Choose a GMA structure on $E_p$, as in Definition \ref{def:GMA structure}. Let $J_p^\ord \subset E_p$ be the two-sided ideal generated by the subset 
\[
\rho_{p,2,1}(G_p) \bigcup (\rho_{p,1,1}-\kcyc)(I_p) \bigcup (\rho_{p,2,2} - 1)(I_p).
\]
As in \cite[Lem.\ 5.9.3]{WWE1}, $J^{\ord}_p$ is independent of the choice of GMA-structure.

\begin{lem}
\label{lem:WWE ord = CS ord}
The Cayley--Hamilton quotient of $E_p$ by $J^{\ord}_p$ is equal to $E^\ord_p$. 
\end{lem}

\begin{proof} 
Let $(V^\ord_{\rho_p})$ denote the kernel of $E_p \rsurj E_p^\ord$, which contains (but may not be generated by) $V^\ord_p$ (see \S\ref{sssec:cond on CH}). It will suffice to show that $(V^\ord_{\rho_p}) = J^\ord_p$. The inclusion  $(V^\ord_{\rho_p})\subset J^\ord_p$ is straightforward: see the calculations in \cite[\S5.9]{WWE1}, from which it is evident that the Cayley--Hamilton quotient of $\rho_p$ by $J^\ord_p$ is a Cayley--Hamilton representation that is ordinary (in the sense of Definition \ref{defn:US-CH ell}). It remains to show that $J^{\ord}_p \subset (V^\ord_{\rho_p})$.

First we will show that $D^\ord_p\vert_{I_p} = \psi(\kcyc \oplus 1)\vert_{I_p} \otimes_{\Z_p} R^\ord_p$. For any $\tau \in I_p$, $\rho^\ord_p(\tau)$ satisfies both polynomials 
\[
T^2 - \Tr_{D^\ord_p}(\tau)T - D^\ord_p(\tau) \quad \text{ and } \quad 
(T -\kcyc(\tau))(T-1),
\]
the first by the Cayley--Hamilton condition and the second by Definition \ref{defn:US-CH p}. If $\omega(\tau) \neq 1$, Hensel's lemma implies that these two polynomials are identical. For such $\tau$, we have $D^\ord_p(\tau)=\kcyc(\tau)$ and $\Tr_{D^\ord_p}(\tau)=\kcyc(\tau)+1$. Now choose an arbitrary element of $I_p$ and write it as $\sigma\tau$ with $\omega(\sigma), \omega(\tau) \neq 1$. We immediately see that $D^\ord_p(\sigma\tau) = \kcyc(\sigma\tau)$, since both sides are multiplicative. Let $r_\sigma =\rho_p^\ord(\sigma)$ and $r_\tau=\rho_p^\ord(\tau)$. Since $E_p^\ord$ is Cayley--Hamilton, we have
\[
(t_\sigma r_\sigma + t_\tau r_\tau)^2-\Tr_{D_p^\ord}(t_\sigma r_\sigma + t_\tau r_\tau)(t_\sigma r_\sigma + t_\tau r_\tau)+D_p^\ord(t_\sigma r_\sigma + t_\tau r_\tau)=0
\]
in the polynomial ring $E_p^\ord[t_\sigma, t_\tau]$. We can expand $D_p^\ord(t_\sigma r_\sigma + t_\tau r_\tau)$ using \cite[Example 1.8]{chen2014}. Taking the coefficient of $t_\sigma t_\tau$ and writing $\Tr = \Tr_{D^\ord_p}$ for brevity,
\[
r_\sigma r_\tau + r_\tau r_\sigma - \Tr(\sigma)r_\tau - \Tr(\tau)r_\sigma - \Tr(\sigma\tau) + \Tr(\sigma)\Tr(\tau) = 0.
\]
Substituting for $r_\sigma r_\tau$ using $V^\ord_{\rho_p}(\sigma, \tau) = 0$ and for $r_\tau r_\sigma$ using $V^\ord_{\rho_p}(\tau, \sigma) = 0$, one obtains the desired conclusion $\Tr(\sigma\tau) = \kcyc(\sigma\tau) + 1$. 

Let $\sigma \in I_p$, and let $\tau \in I_p$ be such that $\omega(\tau) \ne 1$. Using the fact that $\rho^\ord_p\vert_{I_p}$ is reducible, we see that the $(1,1)$-coordinate of $V^\ord_{\rho^\ord_p}(\sigma, \tau)$ is
\[
(\rho_{p,1,1}^\ord(\sigma)-\kcyc(\sigma))(\rho^\ord_{p,1,1}(\tau)-1)=0
\]
Since $\rho_{p,1,1}^{\ord}$ is a deformation of $\omega$, we have $\rho_{p,1,1}^\ord(\tau)-1 \in (R_p^\ord)^\times$, so this implies $\rho_{p,1,1}^\ord(\sigma)-\kcyc(\sigma) =0$. This shows that $(\rho_{p,1,1}-\kcyc)(I_p) \subset (V^\ord_{\rho_p})$, and a similar argument gives $(\rho_{p,2,2}^\ord-1)(I_p) \subset (V^\ord_{\rho_p})$. 

It remains to show that $\rho^\ord_{p,2,1}(G_p) = 0$. Let $\m \subset R^\ord_p$ be the maximal ideal. In fact, we will show that $C^\ord_p/\m C^\ord_p = 0$, which is equivalent because $\rho^\ord_{p,2,1}(G_p)$ generates the finitely generated $R^\ord_p$-module $C^\ord_p$. We work with $\bar{\rho}^\ord := \rho^\ord_p \pmod{\m}$. Since $\bar{\rho}^\ord$ is reducible, we can consider $\bar\rho^{\ord}_{2,1} \in Z^1(G_p, C_p^\ord/\m C_p^\ord \otimes_{\F_p} \F_p(-1))$, and \cite[Thm.\ 1.5.5]{BC2009} implies that there is an injection
\[
\Hom_{\F_p}(C_p^\ord/\m C_p^\ord, \F_p) \rinj H^1(G_p, \F_p(-1))
\]
sending $\phi$ to the class of the cocycle $\phi \circ \bar\rho^{\ord}_{2,1}$. So to show that $C^\ord_p/\m C^\ord_p$ is zero, it is enough to show that $\bar\rho^{\ord}_{2,1}$ is a coboundary, or, equivalently, that $\bar\rho^{\ord}_{2,1}(\sigma)=0$ for all $\sigma \in \ker(\omega) \subset G_p$. However, we compute that the $(2,1)$-entry of $V^\ord_{\rho_p}(\sigma, \tau)$ is 
\[
\rho_{p,2,1}^\ord(\sigma)(\rho_{p,1,1}^\ord(\tau)-1)+(\rho_{p,2,2}^\ord(\sigma)-\kcyc(\sigma))\rho_{p,2,1}^\ord(\tau).
\]
Taking $\sigma \in \ker(\omega)$ and $\tau \in I_p$ such that $\omega(\tau)\ne 1$, we see that $\rho_{p,1,1}^\ord(\tau)-1 \equiv \omega(\tau) -1 \not \equiv 0 \pmod{\m}$ and $\rho_{p,2,2}^\ord(\sigma)-\kcyc(\sigma) \in \m$, so this implies $\bar\rho^{\ord}_{2,1}(\sigma)=0$.
\end{proof}

We have the following consequence, following \cite[\S5.9]{WWE1}.

\begin{prop}
\label{prop:ord C-H form}
A Cayley--Hamilton representation $(\rho: G_p \ra E^\times, E, D: E \ra A)$ over $\Db\vert_{G_p}$ is ordinary if and only if it admits a GMA structure such that 
\begin{enumerate}
\item it is upper triangular, i.e.\ $\rho_{2,1} = 0$, and 
\item the diagonal character $\rho_{1,1}$ (resp.\ $\rho_{2,2}$) is the product of $\kcyc \otimes_{\Z_p} A$ (resp.\ the constant character $A$) and an unramified $A$-valued character. 
\end{enumerate}
\end{prop}

\begin{cor}
\label{cor:ord to flat compare}
Any finite-flat Cayley--Hamilton representation of $G_p$ over $\Db\vert_{G_p}$ is ordinary. The resulting morphism of universal Cayley--Hamilton representations of $G_p$, $(\rho^\ord_p, E^\ord_p, D_{E^\ord_p}) \ra (\rho^\fl_p, E^\fl_p, D_{E^\fl_p})$, induces an isomorphism on universal pseudodeformation rings $R^\ord_p \risom R^\fl_p$. The universal pseudodeformations $D^\ord_p \cong D^\fl_p$ of $\Db\vert_{G_p}$ have the form $\psi(\kcyc \chi_1 \oplus \chi_2)$, where $\chi_1, \chi_2$ are unramified deformations of the trivial character $1 : G_p \ra \F_p^\times$. 
\end{cor}

\begin{proof}
The Cayley--Hamilton representation $\rho_p^\fl$ satisfies conditions (1) and (2) of Proposition \ref{prop:ord C-H form} by Lemmas \ref{lem:flat implies upper tri} and \ref{lem:flat psrep form}, respectively. The isomorphism of universal pseudorepresentations becomes evident by comparing Lemma \ref{lem:flat psrep form} and Proposition \ref{prop:ord C-H form}(2). 
\end{proof}

\subsection{Global formulation}
We now combine the local constructions to define what it means for a global Cayley--Hamilton representation or pseudorepresentation to be unramified-or-Steinberg of level $N$ and type $\epsilon$. 

\begin{defn} 
\label{defn:US global}
Let $(\rho : G_{\Q,S} \ra E^\times, E, D_E : E \ra A)$ be a Cayley--Hamilton representation over $\Db$. We say that $\rho$ is \emph{unramified-or-Steinberg of level $N$ and type $\epsilon$} (or $\mathrm{US}_N^\epsilon$) when $\rho\vert_{G_\ell}$ is $\mathrm{US}_\ell^{\epsilon_\ell}$ for all primes $\ell \mid N$, and, if $p \nmid N$, $\rho|_{G_p}$ is finite-flat. 

Let $D : G_{\Q,S} \ra A$ be a pseudodeformation of $\Db$. We say that $D$ is \emph{unramified-or-Steinberg of level $N$ and type $\epsilon$} (or $\mathrm{US}_N^\epsilon$) when there exists a Cayley--Hamilton representation $(\rho: G_{\Q,S} \ra E^\times, E, D_E : E \ra A)$ such that $D = \psi(\rho)$ and $\rho$ is $\mathrm{US}^\epsilon_N$. 
\end{defn}

Recall the Cayley--Hamilton representation $\rho_\Db$ set up in \S\ref{subsec:CH Gal setup}. There are maps of Cayley--Hamilton algebras $\iota_\ell : (E_\ell, D_{E_\ell}) \ra (E_\Db, D_{E_\Db})$ arising from the fact that $\rho_\Db\vert_{G_\ell}$ is a Cayley--Hamilton representation of $G_\ell$ over $\Db\vert_{G_\ell}$. For any $\ell \mid Np$, write $J^\epsilon_\ell$ for the kernel of $E_\ell \ra E_\ell^{\epsilon_\ell}$ (refer to Definition \ref{defn:E_p^epsilon} for $E_p^{\epsilon_p}$).

\begin{defn}
\label{defn:global US univ objects}
Let $(E^\epsilon_N, D_{E^\epsilon_N})$ denote the Cayley--Hamilton algebra quotient of $E_\Db$ by the union of $\iota_\ell(J^\epsilon_\ell)$ over all primes $\ell \mid Np$. We denote the quotient Cayley--Hamilton representation of $G_{\Q,S}$ by
\[
(\rho^\epsilon_N : G_{\Q,S} \lra (E^\epsilon_N)^\times, E^\epsilon_N, D_{E^\epsilon_N} : E^\epsilon_N \lra R^\epsilon_N)
\]
and its induced pseudorepresentation by $D^\epsilon_N = \psi(\rho^\epsilon_N) : G_{\Q,S} \ra R^\epsilon_N$. 
\end{defn} 

Using \S\ref{sssec:cond on CH}, we see that $\rho_N^\epsilon$ (resp.\ $D^\epsilon_N$) is the universal $\mathrm{US}^\epsilon_N$ Cayley--Hamilton representation (resp.\ pseudodeformation) over $\Db$. In particular, a homomorphism $R_\Db \to A$ factors through $R_N^\epsilon$ if and only if the corresponding pseudodeformation $D:G_{\Q,S} \to A$ of $\Db$ satisfies $\mathrm{US}^\epsilon_N$.

\begin{prop}
	\label{prop:global US det is kcyc}
	Let $D : G_{\Q,S} \ra A$ be a pseudodeformation of $\Db$ satisfying $\mathrm{US}^\epsilon_N$. Then $D(\tau)=\kcyc(\tau)$ for all $\tau \in G_{\Q,S}$. 
\end{prop}
\begin{proof}
It suffices to show that $D(\tau)=\kcyc(\tau)$ for all $\tau \in I_\ell$ and all $\ell \mid Np$, since this will show that $G_{\Q,S} \ni \sigma \mapsto D(\sigma) \kcyc^{-1}(\sigma) \in A^\times$ is a character of $G_{\Q,S}$ that is unramified everywhere and hence trivial. For $\ell \ne p$, this follows from Lemma \ref{lem:USell implies pseudo-unram}, and for $\ell=p$ this follows from Corollary \ref{cor:ord to flat compare}.
\end{proof}

\subsection{Information about $B_N^\epsilon$ and $C_N^\epsilon$} 
Recall that we fixed a GMA structure on $E_p$ in \S \ref{subsec:ord case}. This defines a GMA structure on $E_p^{\epsilon_p}$ and $E^\epsilon_N$ via the Cayley--Hamilton algebra morphisms $E_p \to E_p^{\epsilon_p}$ and $E_p^{\epsilon_p} \ra E^\epsilon_N$ (see \cite[Theorem 3.2.2]{WWE4}). We write this GMA structure as
\begin{equation}
\label{eq:GMA}
E^\epsilon_N= \ttmat{R^\epsilon_N}{B^\epsilon_N}{C^\epsilon_N}{R^\epsilon_N}, \quad \rho_N^\epsilon(\tau) = \ttmat{a_{\tau}}{b_{\tau}}{c_{\tau}}{d_{\tau}}.
\end{equation}

\subsubsection{Computation of $B_\fl^{\min}$ and $C_\fl^{\min}$} 
\label{sssec:comp of Bfl and Cfl}

First we work in the case that either $p \nmid N$ or $\epsilon_p=+1$, so $E^{\epsilon_p}_p = E^\fl_p$, with a GMA structure chosen. Let $(E_\fl, D_{E_\fl})$ represent the Cayley--Hamilton quotient of $E_\Db$ by $\iota_p(J^\epsilon_p)$, with a GMA structure coming from $E^\fl_p \ra E_\fl$. Write this GMA structure as
\begin{equation}
\label{eq:GMA structure}
E_\fl \cong \ttmat{R_\fl}{B_\fl}{C_\fl}{R_\fl}, \quad \rho_\fl(\tau) = \ttmat{a_{\fl,\tau}}{b_{\fl,\tau}}{c_{\fl,\tau}}{d_{\fl,\tau}}.
\end{equation}

Let $J^{\min}_\fl =\ker(R_\fl \to \Z_p)$, where $R_\fl \to \Z_p$ corresponds to $\psi(\Z_p(1) \oplus \Z_p)$, which is obviously finite-flat. Let 
\[
B_\fl ^{\min} =B_\fl/J^{\min}_\fl B_\fl, \quad   C_\fl ^{\min} =C_\fl/J^{\min}_\fl C_\fl .
\]
By \cite[Prop.\ 2.5.1]{WWE3}, we have, for any finitely-generated $\Z_p$-module $M$, isomorphisms
\begin{align}
\label{eq:BC and exts}
\begin{split}
&\Hom_{\Z_p}(B_\fl ^{\min},M) \cong H^1_\fl(\Z[1/Np],M(1)) \\
 &\Hom_{\Z_p}(C_\fl ^{\min},M) \cong H^1_{(p)}(\Z[1/Np],M(-1)) .
 \end{split}
\end{align}
where $H^1_\fl(\Z[1/Np],M(1))$ equals 
\[
 \ker\left(H^1(\Z[1/Np],M(1)) \to \frac{H^1(\Q_p,M(1))}{\Ext_{\mathrm{ffgs}/\Z_p}(\Z_p,M \otimes_{\Z_p} \mathrm{Ta}_p(\mu_{p^\infty}))}\right)
\]
and
\[
H^1_{(p)}(\Z[1/Np],M(-1))=\ker(H^1(\Z[1/Np], M(-1)) \to H^1(\Q_p,M(-1))).
\]
Here $\mathrm{ffgs}/\Z_p$ is the category of locally-free group schemes of finite rank over $\Z_p$, which maps to the category of $G_p$-modules by taking generic fiber. In other words, a class in $H^1_\fl(\Z[1/Np],M(1))$ (resp.~$H^1_{(p)}(\Z[1/Np],M(-1))$) is represented by a Galois representation $\rho$ that is an extension of $\Z_p$ by $M(1)$ (resp.\ $M(-1)$), such that $\rho|_{G_p}$ is isomorphic to the generic fiber of a locally-free group scheme of finite rank over $\Z_p$ (resp.~$\rho|_p$ is a trivial extension).
The Galois cohomology computations of \cite[\S6.3]{WWE3} allow us to compute these.  

\begin{lem}
\label{lem:B_fl and C_fl}
Recall that $N=\ell_0\ell_1 \cdots \ell_r$, and recall the elements $\gamma_i \in I_{\ell_i}$ for $i=0,\dots,r$ defined in \S \ref{subsec:notation}. There are isomorphisms
\[
\Z_p^{\oplus r+1} \isoto B_\fl ^{\min}, \quad  \bigoplus_{i=0}^r \Z_p/(\ell_i^2-1)\Z_p \isoto C_\fl ^{\min}
\]
given by $e_i \mapsto b_{\fl,\gamma_i}$ and $e_i \mapsto c_{\fl,\gamma_i}$, where $e_i \in \Z_p^{\oplus r+1}$ is the $i$-th standard basis vector.
\end{lem}

\subsubsection{Computation of $B_\ord^{\min}$ and $C_\ord^{\min}$}
\label{sssec:comp of Bord and Cord}
Next we compute in the case $p \mid N$ and $\epsilon_p = -1$, so $E^{\epsilon_p}_p = E^\ord_p$. Let $(E_\ord, D_{E_\ord})$ be the Cayley--Hamilton quotient of $E_\Db$ by $\iota_p(J^\epsilon_p)$, receiving a GMA structure via $E^\ord_p \ra E_\ord$.  Write this GMA structure as  
\begin{equation}
\label{eq:GMA structure ord}
E_\ord \cong \ttmat{R_\ord}{B_\ord}{C_\ord}{R_\ord}, \quad \rho_\ord(\tau) = \ttmat{a_{\ord,\tau}}{b_{\ord,\tau}}{c_{\ord,\tau}}{d_{\ord,\tau}}.
\end{equation}

Let $J^{\min}_\ord =\ker(R_\ord \to \Z_p)$, where $R_\ord \to \Z_p$ corresponds to $\psi(\Z_p(1) \oplus \Z_p)$, which is obviously ordinary. Let 
\[
B_\ord^{\min} =B_\ord/J^{\min}_\ord B_\fl, \quad   C_\ord ^{\min} =C_\ord/J^{\min}_\ord C_\ord .
\]
Just as in \cite[Lem.\ 4.1.5]{WWE2}, we have, for any finitely-generated $\Z_p$-module $M$, isomorphisms
\begin{align}
\label{eq:BC and exts ord}
\begin{split}
&\Hom_{\Z_p}(B_\ord ^{\min},M) \cong H^1(\Z[1/Np],M(1)),\\
& \Hom_{\Z_p}(C_\ord^{\min},M) \cong H^1_{(p)}(\Z[1/Np], M(-1)).
\end{split}
\end{align}
The Galois cohomology computations of \cite[\S6.3]{WWE3} allow us to compute these. Recall that $\gamma_i$ is defined in \S\ref{subsec:notation}, even when $\ell_i = p$. 

\begin{lem}
\label{lem:B_ord and C_ord}
There are isomorphisms
\[
\Z_p^{\oplus r+1} \isoto B_\ord ^{\min}, \quad  \bigoplus_{i=0}^r \Z_p/(\ell_i^2-1)\Z_p \isoto C_\ord ^{\min}
\]
given by $e_i \mapsto b_{\ord,\gamma_i}$ and $e_i \mapsto c_{\ord,\gamma_i}$, where $e_i \in \Z_p^{\oplus r+1}$ is the $i$-th standard basis vector.
\end{lem}

\subsubsection{Information about $B_N^{\epsilon,{\min}}$ and $C_N^{\epsilon,{\min}}$} 

Let $J^{\min} := \ker(R^\epsilon_N \ra \Z_p)$, where this homomorphism is induced by the $\mathrm{US}^\epsilon_N$ pseudodeformation $\psi(\Z_p(1) \oplus \Z_p)$ of $\Db$. 

\begin{lem}
\label{lem:info about B and C}
We consider $B_N^{\epsilon,{\min}}=B_N^\epsilon/\Jm B_N^\epsilon$ and $C_N^{\epsilon,{\min}}=C_N^\epsilon/\Jm C_N^\epsilon$.
\begin{enumerate}
\item If $\epsilon_i=1$ and $\ell_i \ne p$, then the image of $b_{\gamma_i}$ in $B_N^{\epsilon,{\min}}$ is $0$.
\item If $\epsilon_i +\ell_i \not \equiv 0\pmod{p}$, then the image of $c_{\gamma_i}$ in $C_N^{\epsilon,{\min}}$ is $0$.
\end{enumerate}
Moreover, there are surjections
\[
\bigoplus_{i=0}^r \Z_p/(\epsilon_i+1)\Z_p \onto B_N^{\epsilon,{\min}}, \qquad \bigoplus_{i=0}^r \Z_p/(\ell_i+\epsilon_i)\Z_p \onto C_N^{\epsilon,{\min}}.
\]
given by $e_i \mapsto b_{\gamma_i}$ and $e_i \mapsto c_{\gamma_i}$, respectively.
\end{lem}
\begin{proof}
Note that for $\rho_N^{\epsilon,{\min}} = \rho_N^\epsilon \otimes_{R_N^\epsilon} R_N^{\epsilon}/\Jm$, in the GMA structure, we have
\[
\rho_N^{\epsilon,{\min}} = \ttmat{\kcyc}{b}{c}{1}.
\]
Note that we have 
\[
V^{\epsilon_i}_{\rho_N^{\epsilon, {\min}}}(\gamma_i, \sigma_i) = (\rho_N^{\epsilon,{\min}}(\gamma_i)-1)(\rho_N^{\epsilon,{\min}}(\sigma_i)+\epsilon_i)= 0.
\]
In GMA notation, this is
\[
0=\ttmat{0}{b_{\gamma_i}}{c_{\gamma_i}}{0} \ttmat{\ell_i+\epsilon_i}{b_{\sigma_i}}{c_{\sigma_i}}{1+\epsilon_i} = \ttmat{0}{(1+\epsilon_i)b_{\gamma_i}}{(\ell_i+\epsilon_i)c_{\gamma_i}}{0}.
\]
In case (1), $(1+\epsilon_i)$ is invertible, so $b_{\gamma_i}=0$. In case (2), $(\ell_i+\epsilon_i)$ is invertible, so $c_{\gamma_i}=0$.

The final statement follows from (1) and (2) and Lemma \ref{lem:B_ord and C_ord} if $p \mid N$ and $\epsilon_p=-1$; otherwise, it follows from Lemma \ref{lem:B_fl and C_fl}.
\end{proof}

\subsection{Labeling some cohomology classes}
\label{subsec:label}

Later, in \S\ref{sec:GP}, it will be convenient to have notation for the extension classes, taken as Galois cohomology classes, arising from homomorphisms $B_N^{\epsilon,{\min}} \ra \F_p$ and $C_N^{\epsilon,{\min}} \ra \F_p$. 

\begin{defn}
\label{defn:bc in H1}
We call a cohomology class $x \in H^1(\Z[1/Np],M)$ \emph{ramified at a prime $\ell$} when its image in $H^1(I_\ell, M)$ is non-zero. For certain $i$ with $0 \leq i \leq r$, we designate $b_i$ and $c_i$ as follows. 

\begin{itemize}[leftmargin=2em]
\item For $i=0,\dots, r$, let $\tilde{b}_i$ denote the $\F_p^\times$-scaling of the Kummer cocycle of $\ell_i$ such that $\tilde b_i(\gamma_i) = 1$, and let $b_i \in H^1(\Z[1/Np],\F_p(1))$ be the class of $\tilde{b}_i$. 
\item Let $T=\{0\le j \le r \ : \ell_i \equiv \pm 1\pmod{p}\}$. For $i \in T$, let $c_i \in H^1_{(p)}(\Z[1/Np],\F_p(-1))$ be an element that is ramified exactly at $\ell_i$ and such that $\tilde c_i(\gamma_i) = 1$ for any cocycle $\tilde c_i$ representing $c_i$. 
\end{itemize}
\end{defn}

\begin{lem}
The sets $\{b_i\}_{i=0}^r$ and $\{c_i\}_{i \in T}$ are well-defined and satisfy the following properties:
\begin{enumerate}[label=(\roman*)]
\item $b_i$ is characterized up to $\F_p^\times$-scaling by being ramified at $\ell_i$, unramified outside $\{\ell_i, p\}$, and finite-flat at $p$ if $\ell_i \ne p$. 
\item If $p \mid N$, the set $\{b_i\}_{i=0}^r$ is a basis of $H^1(\Z[1/Np],\F_p(1))$. 
\item The subset $\{b_i \ : \ \ell_i \ne p\}$ is a basis of $H^1_\fl(\Z[1/Np],\F_p(1))$.
\item The set $\{c_i\}_{i \in T}$ is a basis of $H^1_{(p)}(\Z[1/Np],\F_p(-1))$. 
\end{enumerate}
\end{lem}

\begin{proof}
The value of $\tilde b_i(\gamma_i)$ is well-defined for the same reason when $\ell_i \neq p$, and $b_p(\gamma_p)$ is well-defined by the choice of $\gamma_p$ (in \S\ref{subsec:notation}). Parts (i), (ii), and (iii) follow from Kummer theory (note that the Kummer class of $p$ is not finite-flat at $p$). 

For part (iv), note that the module $C_N^{\epsilon,{\min}}$ is computed in \cite[Prop.\ 6.3.3]{WWE3}. Together with \eqref{eq:BC and exts}, this computation implies the existence of $c_i \in H^1_{(p)}(\F_p(-1))$ characterized up to $\F_p^\times$-scaling by being ramified exactly at $\ell_i$. These statements also imply part (iv). Because $\omega\vert_{I_{\ell_i}} = 1$, $\tilde c_i\vert_{I_{\ell_i}} : I_{\ell_i} \rsurj \F_p$ is a homomorphism not dependent on the choice of $\tilde c_i$.
\end{proof}

The stated bases are almost dual bases, with the exception arising from the possibility that $b_i$ is ramified at $p$ even when $\ell_i \neq p$. 
\begin{lem}
\label{lem:dual B C and H1}
Under the perfect pairings
\begin{enumerate}
\item $B_\fl \otimes_{R_\fl} \F_p \times H^1_\fl(\Z[1/Np],\F_p(1)) \lra \F_p$,
\item $C_\ord \otimes_{R_\ord} \F_p \times H^1_{(p)}(\Z[1/Np],\F_p(-1)) \lra \F_p$
\item $C_\fl \otimes_{R_\fl} \F_p \times H^1_{(p)}(\Z[1/Np],\F_p(-1)) \lra \F_p$,
\end{enumerate}
defined by \eqref{eq:BC and exts} and \eqref{eq:BC and exts ord}, the following are respective dual basis pairs
\begin{enumerate}
\item $\{b_{\fl,\gamma_i} \ : \ i=0,\dots, r \text{ if } \ell_i \ne p\}$ and $\{b_{i} \ : \ i=0,\dots, r  \text{ if } \ell_i \ne p\}$
\item $\{c_{\ord,\gamma_i} \ : \ i\in T\}$ and $\{c_{i} \ : \ i\in T\}$
\item $\{c_{\fl,\gamma_i} \ : \ i\in T\}$ and $\{c_{i} \ : \ i\in T\}$
\end{enumerate}
Also, for $0 \leq i, j \leq r$ such that $\ell_i = p$ or $\ell_j \neq p$, we have $b_i(b_{\ord, \gamma_j}) = \partial_{ij}$. 
\end{lem}

\begin{proof}
We give the proof for (1), the other parts being similar. The pairing \eqref{eq:BC and exts} sends a class $x \in H^1_\fl(\Z[1/Np],\F_p(1))$ to a homomorphism $B_\fl \to \F_p$ that sends $b_\tau$ to $\tilde{x}(\tau)$, where $\tilde{x}$ is a particular cocycle representing $x$ (the choice is determined by the choice of GMA structure on $E_\fl$). However, if $\omega(\tau)=1$, the value of $\tilde{x}(\tau)$ is independent of the choice of cocycle, and we may write this value as $x(\tau)$. Hence we see that $b_i(b_{\fl,\gamma_j})=b_i(\gamma_j)=\partial_{ij}$.
\end{proof}
\begin{defn}
\label{defn:K_i}
For each $i \in T$, let $K_i$ be the fixed field of $\ker(\tilde c_i \vert_{G_{\Q(\zeta_p),S}})$, where $\tilde c_i$ is any cocycle $\tilde c_i : G_{\bQ,S} \ra \F_p(-1)$ representing $c_i$. 
\end{defn}
One readily verifies that $K_i$ is the unique extension of $\Q(\zeta_p)$ satisfying the properties of \S\ref{subsubsec:defn of K_i}.

\section{Toward $R=\bT$}

\subsection{The map $R_N^\epsilon \to \bT_N^\epsilon$} We prove the following proposition, following the construction technique of Calegari--Emerton \cite[Prop.\ 3.12]{CE2005}.
\begin{prop}
\label{prop:R to T}
There is a surjective homomorphism $R_N^\epsilon \rsurj \bT_N^\epsilon$ of augmented $\Z_p$-algebras. Moreover, $\bT_N^\epsilon$ is generated as a $\Z_p$-algebra by $T_q$ for any cofinite set of primes $q$ not dividing $Np$. 
\end{prop}
\begin{proof}
For this proof, it is important to note that the elements $\Tr_{D^\epsilon_N}(\Fr_q)$ for any such set of primes $q$ generate $R_N^\epsilon$ as a $\Z_p$-algebra. This follows the fact that $R^\epsilon_N$ is a quotient of the (unrestricted) universal pseudodeformation ring $R_\Db$, that traces $\{\Tr_{D_\Db}(\sigma) : \sigma \in G_{\Q,S}\}$ of the universal pseudodeformation generate $R_\Db$ (because the residue characteristic is not 2, see \cite[Prop.\ 1.29]{chen2014}), and Chebotarev density. 

In the rest of the proof, we use the notation $\Sigma$, $\rho_f$ and $\sO_f$ established in Lemma \ref{lem:normalization of bT0}. We proceed in three steps:
\begin{itemize}[leftmargin=4em]
\item[{\bf Step 1.}] Construct a homomorphism $R_N^\epsilon \to \sO_f$ for each $f \in \Sigma$ that sends $\Tr_{D^\epsilon_N}(\Fr_q)$ to $a_q(f)$ for each prime $q \nmid Np$.
\item[{\bf Step 2.}] Show that the resulting map $R_N^\epsilon \to \Z_p \oplus \bigoplus_f \sO_f$ sends $\Tr_{D^\epsilon_N}(\Fr_q)$ to the image of $T_q$ under the map $\bT_N^\epsilon \to \Z_p \oplus \bigoplus_f\sO_f$ of \eqref{eq:bT in product}, for each $q \nmid Np$. This gives a homomorphism $R_N^\epsilon \to \bT_N^\epsilon$ whose image is the $\Z_p$-subalgebra generated by the $T_q$ for all $q \nmid Np$. This completes the proof if $p \mid N$.
\item[{\bf Step 3.}] In the case that $p \nmid N$, show that the image of $R_N^\epsilon \to \bT_N^\epsilon$ contains $U_p$ and $U_p^{-1}$. This shows both that $R_N^\epsilon \to \bT_N^\epsilon$ is surjective and that $\bT_N^\epsilon$ is generated as a $\Z_p$-algebra by $T_q$ for $q \nmid Np$.
\end{itemize}

\begin{proof}[Proof of Step 1] 
Let $f \in \Sigma$. Then $\psi(\bar{\rho}_f)=\Db$, so $\psi(\rho_f)$ induces a map $R_\Db \to \sO_f$. For each prime $q \nmid Np$, we have $\Tr(\rho_f(\Fr_q))=a_q(f)$, so $R_\Db \to \sO_f$ sends $\Tr_{D_\Db}(\Fr_q)$ to $a_q(f)$. 

In order to show that $R_\Db \to \sO_f$ factors through $R_N^\epsilon$, we prove that $\psi(\rho_f)$ and $\rho_f$ are $\mathrm{US}^\epsilon_N$ by verifying local conditions, as per Definition \ref{defn:US global}. 
\begin{itemize}[leftmargin=2em] 
\item For $\ell \mid N$ with $\ell \ne p$, $\rho_f\vert_{G_\ell}$ is $\mathrm{US}_\ell^{\epsilon_\ell}$ by Lemma \ref{lem:normalization of bT0}(3).
\item If $p \nmid N$, or if $p \mid N$ and $f$ is old at $p$, then $\rho_f\vert_{G_p}$ is finite-flat by Lemma \ref{lem:normalization of bT0}(4a). Also, when $p\mid N$, this implies that $\rho_f\vert_{G_p}$ is $\mathrm{US}_p^{\epsilon_p}$ by definition if $\epsilon_p=+1$ and by Corollary \ref{cor:ord to flat compare} if $\epsilon_p=-1$. 
\item If $f$ is new at $p$, then $\epsilon_p=-1$ and $\rho_f\vert_{G_p}$ is $\mathrm{US}_p^{-1}$ by Lemma \ref{lem:normalization of bT0}(4b). \qedhere
\end{itemize}
\end{proof}

\begin{proof}[Proof of Step 2]
By construction, the map $R_N^\epsilon \to \Z_p \oplus \bigoplus_f\sO_f$ sends $\Tr_{D^\epsilon_N}(\Fr_q)$ to $(1+q,\bigoplus_f a_q(f))$, which, by \eqref{eq:bT in product}, is the image of $T_q$.
\end{proof}

\begin{proof}[Proof of Step 3]
Let $\tau \in I_p$ be an element such that $\omega(\tau) \ne 1$. Let $x=\kcyc(\tau) \in \Z_p$, so that $1-x \in \Z_p^\times$. Let $\sigma_p \in G_p$ be the element defined in \S\ref{subsec:notation} and let $z=\kcyc(\sigma_p)$. By Lemma \ref{lem:normalization of bT0}(4), we see that $\Tr(\rho_f(\sigma_p))=za_p(f)^{-1}+a_p(f)$ and $\Tr(\rho_f(\tau\sigma_p))=xza_p(f)^{-1}+a_p(f)$. Hence we have
\begin{align*}
&a_p(f)=\frac{1}{x-1}\big( x\Tr(\rho_f(\sigma_p))-\Tr(\rho_f(\tau\sigma_p))\big) \quad \text{and}  \\
&a_p(f)^{-1} = \frac{1}{z-xz} \big(\Tr(\rho_f(\sigma_p)) -\Tr(\rho_f(\tau\sigma_p))\big).
\end{align*}

We see that $U_p$ is the image of $\frac{1}{x-1}(x\Tr_{D^\epsilon_N}(\sigma_p)-\Tr_{D^\epsilon_N}(\tau\sigma_p))$ and $U_p^{-1}$ is the image of $\frac{1}{z-xz}(\Tr_{D^\epsilon_N}(\sigma_p)-\Tr_{D^\epsilon_N}(\tau\sigma_p))$. Since $\bT_N^\epsilon$ is generated by $T_q$ for $q \nmid Np$ along with $T_p = U_p + pU_p^{-1}$, we see that $R_N^\epsilon \to \bT_N^\epsilon$ is surjective. 
\end{proof}
\let\qed\relax %keep this command only if it continues to be sandwiched between two "\end{proof}"
\end{proof}

\subsection{Computation of $(R_N^\epsilon)^\red$} 
\label{subsec:compute Rred}
In this section, we will frequently make use of the elements $\sigma_i$ and $\gamma_i$ defined in \S \ref{subsec:notation}. We denote by $M^{p\text{-part}}$ the maximal $p$-primary quotient of a finite abelian group $M$.

Consider the group $\Gal(\Q(\zeta_N)/\Q)^{p\text{-part}}$. We have isomorphisms 
\[
\Gal(\Q(\zeta_N)/\Q)^{p\text{-part}} \lrisom \prod_{i=0}^r \Gal(\Q(\zeta_{\ell_i})/\Q)^{p\text{-part}} \lrisom \prod_{i=0}^r \Z_p/(\ell_i-1)\Z_p.
\]
Since $\Q(\zeta_{\ell_i})/\Q$ is totally ramified at $\ell_i$, we can and do choose the second isomorphism so that the image of $\gamma_i$ is $(0,\dots,0,1,0,\dots,0)$ (with $1$ in the $i$-th factor). We define $\alpha_j^i$ to be the $j$-th factor of the image of $\sigma_i$, so that $\sigma_i \mapsto (\alpha_0^i,\alpha_1^i,\dots,\alpha_r^i)$ (we can and do assume that $\alpha_i^i=0$).

\begin{rem}
\label{rem:log tame}
Note that if $\ell_j \equiv 1 \pmod{p}$, we may choose a surjective homomorphism $\log_{\ell_j}:(\Z/\ell_j\Z)^\times \rsurj \F_p$ such that $\log_{\ell_j}(\ell_i) \equiv \alpha_j^i \pmod{p}$. By abuse of notation, we denote by $\log_j = \log_{\ell_j}$ a $\F_p$-valued character of $G_{\Q,S}$ produced by composition with the canonical surjection $G_{\Q,S} \rsurj \Gal(\Q(\zeta_{\ell_j})/\Q) \risom (\Z/\ell_j)^\times$. 
\end{rem}

This isomorphism determines an isomorphism of group rings
\[
\Z_p[\Gal(\Q(\zeta_N)/\Q)^{p\text{-part}}] \isoto \Z_p\left[\prod_{i=0}^r \Z_p/(\ell_i-1)\Z_p\right] \cong \Z_p[y_0,\dots,y_r]/(y_i^{p^{v_i}}-1)
\]
where $v_i=v_p(\ell_i-1)$, and where the second isomorphism sends $y_i$ to the group-like element $(0,\dots,0,1,0,\dots,0)$ (with $1$ in the $i$-th factor). Let
\[
\dia{-}:G_{\Q,S} \to (\Z_p[y_0,\dots,y_r]/(y_i^{p^{v_i}}-1))^\times
\]
be the character obtained by the quotient $G_{\Q,S} \onto \Gal(\Q(\zeta_N)/\Q)^{p\text{-part}}$ followed by this isomorphism. Note that 
\[
\dia{\gamma_i}=y_i, \qquad \dia{\sigma_i}=\prod_{j=0}^r y_j^{\alpha_j^i}.
\]

Let $R^\red_\fl(\kcyc)$ (resp.\ $R^\red_\ord(\kcyc)$) be the quotient of the finite-flat global deformation ring $R_\fl$ (resp.\ ordinary global deformation ring $R_\ord$) defined in \S\ref{sssec:comp of Bfl and Cfl} (resp.\ \S\ref{sssec:comp of Bord and Cord}) by the ideal generated by the reducibility ideal along with $\{D_\Db(\gamma) - \kcyc(\gamma) : \gamma \in G_{\Q,S}\}$. That is, we are insisting that the determinant is $\kcyc$. 
\begin{lem}
	\label{lem:calculate R-red-flat}
	The surjection $R_\ord \rsurj R_\fl$ induces an isomorphism $R_{\ord}^\red(\kcyc) \risom R_\fl^\red(\kcyc)$. Moreover, they are both isomorphic as rings to 
	\[
	\Z_p[y_0,\dots,y_r]/(y_i^{p^{v_i}}-1)
	\]
	and the universal reducible pseudorepresentation pulls back to $D^\red=\psi(\kcyc \dia{-}^{-1} \oplus \dia{-})$ via these isomorphisms.
\end{lem}
\begin{proof} 
	The quotient map $R_\ord \rsurj R_\fl$ comes from the first part of Corollary \ref{cor:ord to flat compare}, and the two rings differ only in the local condition at $p$. After imposing the reducibility and determinant conditions, the universal pseudodeformations both have the form $\psi(\kcyc \chi^{-1} \oplus \chi)$ for a character $\chi$ that deforms the trivial character. By the latter parts of the corollary, the finite-flat and ordinary conditions on such pseudodeformations are identical. The last statement is proven just as in \cite[Lem.\ 5.1.1]{WWE3}.
\end{proof}

Let $Y_i=y_i-1$.
\begin{lem}
\label{lem:computation of Rred}
There is an isomorphism
\[
(R_N^\epsilon)^\red \cong \Z_p[Y_0,\dots,Y_r]/\mathfrak{a}
\]
where $\mathfrak{a}$ is the ideal generated by the elements 
\[
Y_i^2,(\ell_i-1)Y_i,(\epsilon_i+1)Y_i, Y_i\left(\prod_{j=0}^r(1-\tilde{\alpha}_i^jY_j)-1\right),
\]
for $i=0,\dots,r$, where $\tilde{\alpha}_i^j \in \Z_p$ is any lift of ${\alpha_i^j}\in \Z_p/(\ell_j-1)\Z_p$ (note that $\mathfrak{a}$ is independent of the choice of this lift).
\end{lem}
\begin{proof}
We consider $(E_N^\epsilon)^\red =E_N^\epsilon \otimes_{R_N^\epsilon} (R_N^\epsilon)^\red$. We write the base-change of $\rho_N^\epsilon$ to this algebra as $\rho^\red$, for simplicity. Write $\odia{-}: G_{\Q,S} \to ((R_N^\epsilon)^\red)^\times$ for the composite of $\dia{-}$ with the quotient $R_\fl^\red(\kcyc) \to (R_N^\epsilon)^\red$, which exists by Proposition \ref{prop:global US det is kcyc}. (We use $R_\fl^\red(\kcyc)$ even in the ordinary case, in light of Lemma \ref{lem:calculate R-red-flat}.) 

First we show that the map $R_\fl^\red(\kcyc) \to (R_N^\epsilon)^\red$ factors through $\Z_p[Y_0,\dots,Y_r]/\mathfrak{a}$. We can write $\rho^\red$ in GMA notation as
\[
\rho^\red = \ttmat{\kcyc\odia{-}^{-1}}{*}{*}{\odia{-}}.
\]
Since $V^{\epsilon_i}_{\rho^\red}(\gamma_i, \gamma_i) = (\rho^\red(\gamma_i)-1)^2=0$ in $(E_N^\epsilon)^\red$, we see that $Y_i^2=0$ in $(R_N^\epsilon)^\red$. Since $(1+Y_i)^{p^{v_i}}-1=0$, this implies that $p^{v_i}Y_i=0$ in $(R_N^\epsilon)^\red$. Moreover, by Lemma \ref{lem:epsilon=1 and unram}, if $\epsilon_i=+1$ and $v_i>0$, then $\rho^\red(\gamma_i)=1$; for such $i$, this implies that $Y_i=0$ in $(R_N^\epsilon)^\red$. We can rephrase this as $(\epsilon_i+1)Y_i=0$ for all $i$. 

From now on, consider $i$ such that $\epsilon_i =-1$. Already, we see that
\[
\odia{\sigma_i} = \prod_{j=0}^r y_j^{\alpha_j^i} = \prod_{j=0}^r (1+\tilde{\alpha}_j^iY_j)
\]
Since $V^{\epsilon_i}_{\rho^\red}(\gamma_i, \sigma_i) = (\rho^\red(\gamma_i)-1)(\rho^\red(\sigma_i)-1)=0$ in $(E_N^\epsilon)^\red$, we obtain 
\[
(\odia{\gamma_i}^{-1}-1)(\ell_i\odia{\sigma_i}^{-1}-1)=0, \quad (\odia{\gamma_i}-1)(\odia{\sigma_i}-1)=0.
\]
These imply
\[
0=Y_i\left(\prod_{j=0}^r(1-\tilde{\alpha}_i^jY_j)-1\right)=Y_i\left(\prod_{j=0}^r(1+\tilde{\alpha}_i^jY_j)-1\right).
\]
However, this last equation is redundant because
\[
\left(-\prod_{j=0}^r(1+\tilde{\alpha}_i^jY_j)\right) \left(\prod_{j=0}^r(1-\tilde{\alpha}_i^jY_j)-1\right) \equiv \left(\prod_{j=0}^r(1+\tilde{\alpha}_i^jY_j)-1\right) \pmod{Y_0^2,\dots,Y_r^2}.
\]

This shows that $R_\fl^\red(\kcyc) \to (R_N^\epsilon)^\red$ factors through $\Z_p[Y_0,\dots,Y_r]/\mathfrak{a}$. It remains to verify that the pseudorepresentation $D:G_{\Q,S} \to \Z_p[Y_0,\dots,Y_r]/\mathfrak{a}$ defined by $\psi(\kcyc\odia{-}^{-1}\oplus \odia{-})$ is $\mathrm{US}_N^\epsilon$. This is checked easily.
\end{proof}

\section{The case $\epsilon=(-1,1,1,\dots,1)$}
\label{sec:case1}

In this section, we consider the case where $\epsilon_0=-1$ and $\epsilon_i=1$ for $0 <  i \leq r$. Without loss of generality, we can and do, for this section, assume that the primes $\{\ell_i\}_{i=0}^r$ are ordered so that $\ell_i \equiv -1 \pmod{p}$ for $i=1,\dots,s$ and $\ell_i \not \equiv -1 \pmod{p}$ for $s < i \leq r$. Here $s$ is an integer, $0 \leq s \leq r$. The most interesting case is $s=r$, and, in fact, we immediately reduce to this case.

\subsection{Reduction to the case $s=r$} 
Let $N(s)=\prod_{i=0}^s \ell_i$ and $\epsilon(s) \in \{ \pm 1\}^{s+1}$ be defined by $\epsilon(s)_0=-1$ and $\epsilon(s)_i=1$ for $0 < i \leq s$. There is a natural map $\bT_N^\epsilon \onto \bT_{N(s)}^{\epsilon(s)}$ by restricting to the space of forms that are old at $\ell_i$ for $s < i \leq r$. There is also a natural surjection $R_N^\epsilon \onto R_{N(s)}^{\epsilon(s)}$, since $\rho_{N(s)}^{\epsilon(s)}$ is unramified (resp.\ finite-flat) at $\ell_i$ when $\ell_i \neq p$ (resp.\ $\ell_i = p$) and $s < i \leq r$. 

\begin{lem}
\label{lem:reduction to lower level}
The natural map $R_N^\epsilon \onto R_{N(s)}^{\epsilon(s)}$ is an isomorphism. Moreover, if the map $R_{N(s)}^{\epsilon(s)} \rsurj \bT_{N(s)}^{\epsilon(s)}$ is an isomorphism, then the surjections $R_{N}^{\epsilon} \rsurj \bT_{N}^{\epsilon}$ and $\bT_N^\epsilon \onto \bT_{N(s)}^{\epsilon(s)}$ of Proposition \ref{prop:R to T} are isomorphisms. 
\end{lem}
\begin{proof}
The isomorphy of $R_N^\epsilon \onto R_{N(s)}^{\epsilon(s)}$ can be rephrased as saying that, for all $s < i \leq r$, $\rho_N^\epsilon$ is unramified (resp.\ finite-flat) at $\ell_i$ if $\ell_i \neq p$ (resp.\ if $\ell_i = p$). This follows from Lemma \ref{lem:epsilon=1 and unram} and \S\ref{subsec:epsilon_p=1}. For the second statement, consider the commutative diagram of surjective ring homomorphisms
\begin{align*}
\begin{split}
\xymatrix{
R_N^\epsilon \ar[r] \ar[d]^(.4)\wr & \bT_{N}^{\epsilon} \ar[d] \\
R_{N(s)}^{\epsilon(s)} \ar[r] & \bT_{N(s)}^{\epsilon(s)}. 
}
\end{split}
\qedhere
\end{align*} 
\end{proof}

\subsection{The case $s=r$} 
\label{subsec:case1 r=s}

Now we assume that $s=r$ (i.e.\ that $\ell_i \equiv -1 \pmod{p}$ for $i=1,\dots, r$). We write $\Jm \subset R_{N}^{\epsilon}$ for the augmentation ideal, and $J^\red = \ker(R_{N}^{\epsilon} \onto (R_{N}^{\epsilon})^\red)$. We have the following consequence of Wiles's numerical criterion \cite[Appendix]{wiles1995}.

\begin{prop}
\label{prop:numerical crit}
The surjection $R_{N}^{\epsilon} \rsurj \bT_{N}^{\epsilon}$ is a isomorphism of complete intersection rings if and only if
\[
\# \Jm/\Jm^2  \le p^{v_p(\ell_0-1)} \cdot \prod_{i=1}^r p^{v_p(\ell_i+1)}.
\]
If this is the case, then equality holds.
\end{prop}
\begin{proof}
The surjection comes from Proposition \ref{prop:R to T}.  Note that 
\[
p^{v_p(\ell_0-1)} \cdot \prod_{i=1}^r p^{v_p(\ell_i+1)} = \# \Z_p/a_0(E^\epsilon)\Z_p.
\]
The proposition follows from Theorem \ref{thm:congruence number} and the numerical criterion, as in \cite[Thm.\ 7.1.1]{WWE3}.
\end{proof}

\begin{lem}
\label{lem:Jm/Jred}
There is an isomorphism 
\[
\Jm/J^\red \cong \Z_p/(\ell_0-1)\Z_p
\]
sending $d_{\gamma_0}-1$ to $1$, and $\Jm^2 \subset J^\red$.
\end{lem}
\begin{proof}
By Lemma \ref{lem:computation of Rred}, we have 
\[
(R_N^\epsilon)^\red = \Z_p[Y_0]/((\ell_0-1)Y_0,Y_0^2),
\]
and we can easily see that $d_{\gamma_0}-1$ maps to $Y_0$ and generates the image of $\Jm$. Since $Y_0^2=0$, we have the second statement. 
\end{proof}

\begin{lem}
\label{lem:Jred/JmJred}
There is a surjection
\[
\Z_p/(\ell_0-1)\Z_p \oplus \left(\bigoplus_{i=1}^r \Z_p/(\ell_i+1)\Z_p \right) \onto J^\red/\Jm J^\red
\]
given by $e_i \mapsto b_{\gamma_0}c_{\gamma_i}$.
\end{lem}
\begin{proof}
By Lemma \ref{lem:info about B and C}, we have surjections
\[
\Z_p \onto B_N^{\epsilon,{\min}}, \quad 1 \mapsto b_{\gamma_0}
\]
and 
\[
\Z_p/(\ell_0-1)\Z_p \oplus \left(\bigoplus_{i=1}^r \Z_p/(\ell_i+1)\Z_p \right) \onto C_N^{\epsilon,{\min}}, \quad e_i \mapsto c_{\gamma_i}.
\]
By \cite[Prop.\ 1.5.1]{BC2009}, in any $A$-GMA $E=\sm{A}{B}{C}{A}$, the structure map $B \otimes_A C \to A$ has image equal to the reducibility ideal of $E$.  Applying this to the $R^\epsilon_N$-GMA $E^\epsilon_N$ of \eqref{eq:GMA} we have an $R^\epsilon_N$-module surjection $B_N^{\epsilon} \otimes_{R^\epsilon_N}  C_N^{\epsilon} \onto J^\red$. Tensoring this by $R^\epsilon_N/\Jm=\Z_p$, we have a surjection
\begin{equation}
\label{eq:GMA cotangent control}
B_N^{\epsilon,{\min}} \otimes_{\Z_p}  C_N^{\epsilon,{\min}} \onto J^\red/\Jm J^\red, \quad b \otimes c \mapsto bc.
\end{equation}
Combining these, we have the lemma.
\end{proof}

\begin{lem}
\label{lem:bc in Jm^2}
The element $b_{\gamma_0}c_{\gamma_0} \in R_N^\epsilon$ is in $\Jm^2$. 
\end{lem}
\begin{proof}
Since $V^{\epsilon_0}_{\rho^\epsilon_N}(\gamma_0, \gamma_0) = (\rho_N^\epsilon(\gamma_0)-1)^2=0$, we see that $(a_{\gamma_0}-1)^2+b_{\gamma_0}c_{\gamma_0} =0$. Since $a_{\gamma_0}-1 \in \Jm$, we have the lemma.
\end{proof}

We have arrived at the main theorem. 
\begin{thm}
\label{thm:star main}
Let $N=\ell_0\ell_1 \dotsm \ell_r$ and $\epsilon=(-1,1,\dots,1)$. Then the map $R_{N}^{\epsilon} \rsurj \bT_{N}^{\epsilon}$ is a isomorphism of augmented $\Z_p$-algebras, and both rings are complete intersection. The ideal $J^{\min}$ is generated by the elements $b_{\gamma_0}c_{\gamma_i}$ for $i=1,\dots,r$ together with $d_{\gamma_0}-1$. There is an exact sequence
\begin{equation}
\label{eq:star ses}
0 \to \bigoplus_{i=1}^r \Z_p/(\ell_i+1)\Z_p  \to I^\epsilon/{I^\epsilon}^2 \to \Z_p/(\ell_0-1)\Z_p \to 0.
\end{equation}
\end{thm} 
\begin{proof}
By Lemma \ref{lem:Jm/Jred}, there is an exact sequence
\begin{equation}
\label{eq:proof ses}
0 \to J^\red/\Jm^2 \to \Jm/\Jm^2 \to \Z_p/(\ell_0-1)\Z_p \to 0
\end{equation}
Combining Lemmas \ref{lem:Jred/JmJred} and \ref{lem:bc in Jm^2}, we see that there is a surjection
\begin{equation}
\label{eq:onto Jred/Jm^2}
\bigoplus_{i=1}^r \Z_p/(\ell_i+1)\Z_p \onto J^\red/\Jm^2
\end{equation}
given by $e_i \mapsto b_{\gamma_0}c_{\gamma_i}$. This shows that
\[
\# \Jm/\Jm^2  \le p^{v_p(\ell_0-1)} \cdot \prod_{i=1}^r p^{v_p(\ell_i+1)}.
\]
By Proposition \ref{prop:numerical crit}, this shows that $R_{N}^{\epsilon} \rsurj \bT_{N}^{\epsilon}$ is an isomorphism of complete intersection rings, and that this inequality is actually equality. This implies that \eqref{eq:onto Jred/Jm^2} is an isomorphism. Using Lemma \ref{lem:Jm/Jred} and Nakayama's lemma, this shows that $J^{\min}$ is generated by the stated elements. Since $\Jm$ maps isomorphically onto $I^\epsilon$, the desired sequence follows from \eqref{eq:proof ses}.
\end{proof}

\section{The case $\epsilon=(-1,-1)$}
\label{sec:ep=(-1,-1)}

In this section, we assume that $r=1$ and also that $\epsilon=(-1,-1)$.

\subsection{No interesting primes} If $\ell_i \not \equiv 1 \pmod{p}$ for $i=0,1$, then there are no cusp forms congruent to the Eisenstein series.

\begin{thm}
If $\ell_i \not \equiv 1 \pmod{p}$ for $i=0,1$, then $\bT_N^{\epsilon}= \Z_p$ and $\bT_N^{\epsilon,0}=0$.
\end{thm}
\begin{proof}
It is enough to show that $R_N^\epsilon=\Z_p$. By Lemma \ref{lem:computation of Rred}, we have $R_N^{\epsilon,\red} = \Z_p$ and by Lemma \ref{lem:info about B and C} we have $C_N^{\epsilon}=0$, so $J^\red=0$. This implies $R_N^\epsilon=\Z_p$.
\end{proof}

\subsection{Generators of $B_N^\epsilon$} Since nothing interesting happens if there are no interesting primes, we now assume that $\ell_0 \equiv 1 \pmod{p}$. We emphasize that, in this section, we do not assume that $\ell_1 \ne p$. Recall the notation $a_\tau,b_\tau,c_\tau, d_\tau$ for $\tau \in G_{\Q,S}$ from \eqref{eq:GMA} and the elements $\gamma_i, \sigma_i \in G_{\Q,S}$ from \S \ref{subsec:notation}.

\begin{lem}
\label{lem:b_sig and b_gam generate}
Assume that $\ell_1$ is not a $p$-th power modulo $\ell_0$. Then the subset $\{b_{\gamma_0},b_{\sigma_0}\} \subset B_N^\epsilon$ generates $B_N^\epsilon$ as a $R_N^\epsilon$-module.
\end{lem}
\begin{proof}
We give the proof in the case $\ell_1=p$; the case $\ell_1 \ne p$ is exactly analogous, changing `ordinary' to `finite-flat' everywhere. Because $B^{\min}_{\ord}$ surjects onto $B^\epsilon_N$ and by Nakayama's lemma, it is enough to show that the images $\bar b_{\ord,\gamma_0},\bar b_{\ord,\sigma_0}$ of $b_{\ord,\gamma_0}, b_{\ord, \sigma_0}$ in $B^{\min}_\ord/pB^{\min}_\ord$ generate $B^{\min}_\ord/pB^{\min}_\ord$.

Using $b_i$, $\tilde b_i$ defined in \S\ref{subsec:label} and the lemmas there, we know that $\{\bar b_{\ord,\gamma_0}, \bar b_{\ord,\gamma_1}\}$ is a basis for $B^{\min}_\ord/pB^{\min}_\ord$ and $b_1(\bar b_{\ord,\gamma_j}) = \partial_{1j}$ for $j = 0,1$. Hence it is enough to show that $b_1(\bar b_{\ord,\sigma_0}) \neq 0$. As in the proof of Lemma \ref{lem:dual B C and H1}, the fact that $\omega(\sigma_0)=1$ implies that $b_1(\bar b_{\ord,\sigma_0})=\tilde{b}_1(\sigma_0)$. Because $\ell_1$ is not a $p$-th power modulo $\ell_0$, class field theory implies that $\tilde{b}_1(\sigma_0) \neq 0$. 
\end{proof}

\begin{prop}
\label{prop:bcs in Jm^2}
Assume that $\ell_1$ is not a $p$-th power modulo $\ell_0$. Then 
\[
b_{\gamma_0}c_{\gamma_0},\ b_{\gamma_1}c_{\gamma_1}, \ b_{\gamma_1}c_{\gamma_0} \in \Jm^2.
\] 
If, in addition, $\ell_1 \equiv 1 \pmod{p}$ and $\ell_0$ is not a $p$-th power modulo $\ell_1$, then $b_{\gamma_0}c_{\gamma_1} \in \Jm^2$ as well.
\end{prop}
\begin{proof}
The proof for $b_{\gamma_0}c_{\gamma_0},b_{\gamma_1}c_{\gamma_1}$ is just as in Lemma \ref{lem:bc in Jm^2}. If we prove that $b_{\gamma_1}c_{\gamma_0} \in \Jm^2$, then we get $b_{\gamma_0}c_{\gamma_1} \in \Jm^2$ in the second statement by symmetry. So it suffices to prove $b_{\gamma_1}c_{\gamma_0} \in \Jm^2$.

Let $X=a_{\sigma_0}-\ell_0$ and $W = a_{\gamma_0}-1$, and note that $X,W \in \Jm$. From the $(1,1)$-coordinate of the equation $V^{\epsilon_{\ell_0}}_{\rho^\epsilon_N}(\sigma_0, \gamma_0) = 0$ defined in \eqref{eq:US test elts}, we see that $XW+b_{\sigma_0}c_{\gamma_0}=0$. In particular, $b_{\sigma_0}c_{\gamma_0} \in \Jm^2$.

By Lemma \ref{lem:b_sig and b_gam generate}, we know that $b_{\gamma_1}$ is in the $R^\epsilon_N$-linear span of $b_{\sigma_0}$ and $b_{\gamma_0}$. Because both $b_{\sigma_0}c_{\gamma_0}$ and $b_{\gamma_0}c_{\gamma_0}$ lie in $\Jm^2$, so does $b_{\gamma_1}c_{\gamma_0}$. 
\end{proof}

\subsection{One interesting prime} We assume that $\ell_0 \equiv 1 \pmod{p}$ and $\ell_1 \not \equiv 1 \pmod{p}$ (including the possibility that $\ell_1=p$). There is a natural surjective homomorphism $\bT_N^\epsilon \onto \bT_{\ell_0}$ by restricting to forms that are old at $\ell_1$.

\begin{thm}
\label{thm:one interesting prime}
Assume that $\ell_0 \equiv 1 \pmod{p}$, that $\ell_1 \not \equiv 1 \pmod{p}$, and that $\ell_1$ is not a $p$-th power modulo $\ell_0$. Then the natural map $\bT_N^\epsilon \rsurj \bT_{\ell_0}$ is an isomorphism. In particular, $I^\epsilon$ is principal, $\bT_N^\epsilon$ and $\bT_N^{\epsilon,0}$ are complete intersections, and there are no newforms in $S_2(N)_{\rm Eis}^\epsilon$.
\end{thm}
\begin{proof}
Just as in the proof of Lemma \ref{lem:reduction to lower level}, it suffices to show that the map $R_N^\epsilon \onto \bT_{\ell_0}$ is an isomorphism. By Lemma \ref{lem:computation of Rred}, there is an isomorphism
\[
R_N^{\epsilon,\red} \cong \Z_p[Y_0]/(Y_0^2,(\ell_0-1)Y_0),
\]
where the image of $\Jm$ is the ideal generated by $Y_0$. This implies that $\Jm^2 \subset J^\red$ and that there is an isomorphism 
\[
\Z_p/(\ell_0-1)\Z_p \lrisom \Jm/J^\red, \quad 1 \mapsto Y_0.
\]
On the other hand, we know that $J^\red$ is generated by the set $\{b_{\gamma_0}c_{\gamma_0},b_{\gamma_1}c_{\gamma_0}\}$ by Lemma \ref{lem:info about B and C} and the surjection \eqref{eq:GMA cotangent control}. By Proposition \ref{prop:bcs in Jm^2}, we see that this set is contained in $\Jm^2$. Hence $J^\red \subset \Jm^2$, and so $J^\red= \Jm^2$. 

Now we have $\#\Jm/\Jm^2 = p^{v_p(\ell_0-1)}$ and, by the numerical criterion (Proposition \ref{prop:numerical crit}), $R_N^\epsilon \onto \bT_{\ell_0}$ is an isomorphism.
\end{proof}

\begin{rem}
The assumption that $\ell_1$ is not a $p$-th power modulo $\ell_0$ is necessary: see the examples in \S\ref{subsubsec:2 primes no new examples}.
\end{rem}

\subsection{Two interesting primes} 
\label{subsec:two int primes}

We consider the case $\ell_i \equiv 1 \pmod{p}$ for $i=0,1$.

\begin{thm}
\label{thm:thm2}
Let $N=\ell_0\ell_1$ and $\epsilon=(-1,-1)$. Assume that $\ell_i \equiv 1 \pmod{p}$ for $i=0,1$ and assume that neither prime is a $p$-th power modulo the other. Then the map $R_{N}^{\epsilon} \rsurj \bT_{N}^{\epsilon}$ is an isomorphism of complete intersection rings augmented over $\Z_p$, and there is an isomorphism 
\[
I^\epsilon/{I^\epsilon}^2 \cong \Z_p/(\ell_0-1)\Z_p \oplus \Z_p/(\ell_1-1)\Z_p.
\]
\end{thm}
\begin{proof}
By Lemma \ref{lem:computation of Rred}, we see that there is an isomorphism 
\[
R_N^{\epsilon,\red} \cong \Z_p[Y_0,Y_1]/(Y_0^2,Y_0Y_1,Y_1^2,(\ell_0-1)Y_0,(\ell_1-1)Y_1)
\]
and that the image of $\Jm$ is the ideal generated by $(Y_0,Y_1)$. In particular $\Jm^2 \subset J^\red$ and
\[
\Jm/J^\red \cong  \Z_p/(\ell_0-1)\Z_p \oplus \Z_p/(\ell_1-1)\Z_p.
\]
Moreover, by Proposition \ref{prop:bcs in Jm^2} and Lemma \ref{lem:info about B and C}, we see that $J^\red \subset \Jm^2$ so we have
\[
\Jm/\Jm^2 = \Jm/J^\red  \cong \Z_p/(\ell_0-1)\Z_p \oplus \Z_p/(\ell_1-1)\Z_p.
\]
In particular, $\#\Jm/\Jm^2 = p^{v_p(\ell_0-1)+v_p(\ell_1-1)}$.

Now the numerical criterion of Proposition \ref{prop:numerical crit} implies that $R_{N}^{\epsilon} \rsurj \bT_{N}^{\epsilon}$ is a isomorphism of complete intersection augmented $\Z_p$-algebras. It follows that $I^\epsilon = \Jm$, and so the description of $I^\epsilon/{I^\epsilon}^2$ also follows.
\end{proof}

\begin{rem}
Again, the assumptions are necessary. See the examples in \S\ref{subsubsec:2 primes new examples}. 
\end{rem}

\begin{defn}
\label{defn:no newforms} We say there are \emph{no newforms in $M_2(N;\Z_p)_{\mathrm{Eis}}^\epsilon$} if 
\[
M_2(N;\Z_p)_{\mathrm{Eis}}^\epsilon = M_2(\ell_0;\Z_p)_{\mathrm{Eis}} + M_2(\ell_1;\Z_p)_{\mathrm{Eis}},
\]
where the later are considered submodules of the former via the stabilizations in \S \ref{subsub:stabilizations}. Otherwise, we say there   are \emph{newforms in $M_2(N;\Z_p)_{\mathrm{Eis}}^\epsilon$}.
\end{defn}

\begin{thm}
\label{thm:newforms in case 2}
Let $N=\ell_0\ell_1$ and $\epsilon=(-1,-1)$ and assume that $\ell_i \equiv 1 \pmod{p}$ for $i=0,1$. If there are no newforms in $M_2(N;\Z_p)_{\mathrm{Eis}}^\epsilon$, then $\bT_N^\epsilon$ is not Gorenstein. In particular, if neither prime $\ell_i$ is a $p$-th power modulo the other, then there are newforms in $M_2(N;\Z_p)_{\mathrm{Eis}}^\epsilon$.
\end{thm}
\begin{proof}
The second statement follows from the first statement by Theorem \ref{thm:thm2}. Now assume that there are no newforms in $M_2(N;\Z_p)_{\mathrm{Eis}}^\epsilon$. We count that 
\[\mathrm{rank}_{\Z_p}( M_2(N; \Z_p)_{\rm Eis}^{\epsilon}) = \mathrm{rank}_{\Z_p}(M_2(\ell_0;\Z_p)_{\rm Eis})+ \mathrm{rank}_{\Z_p}(M_2(\ell_1;\Z_p)_{\rm Eis})-1
\]
(by Lemma \ref{lem:normalization of bT0}, for example).

We claim that, under this assumption, we have an isomorphism $\bT_N^{\epsilon} \isoto \bT_{\ell_0} \times_{\Z_p} \bT_{\ell_1}$. To see this, consider the commutative diagram of free $\Z_p$-modules, where the right square consists of canonical surjective homomorphisms of commutative $\Z_p$-algebras and the rows are exact: 
\[\xymatrix{
0 \ar[r] & \mathfrak{a}_1 \ar[r] \ar[d] & \bT_N^{\epsilon}  \ar@{->>}[r] \ar@{->>}[d] & \bT_{\ell_1} \ar[r] \ar@{->>}[d] & 0 \\
0 \ar[r] & I_0 \ar[r] & \bT_{\ell_0}  \ar@{->>}[r]  & \Z_p \ar[r]  & 0.
}\]
By Lemma \ref{lem:fiber prods}, it is enough to show that $\mathfrak{a}_1 \to I_0$ is an isomorphism. From this diagram and the above rank count, we see that $\mathrm{rank}_{\Z_p}(\mathfrak{a}_1) = \mathrm{rank}_{\Z_p}( I_0)$. Thus it suffices to show that the $\Z_p$-dual map is surjective. By duality \eqref{eq:M and T duality}, the dual map is identified with the map
\[
M_2(\ell_0;\Z_p)_{\rm Eis}/\Z_p E_{2,\ell_0} \to M_2(N;\Z_p)_{\rm Eis}^{\epsilon}/M_2(\ell_1;\Z_p)_{\rm Eis}
\]
induced by stabilization, which is surjective by our assumption $M_2(N;\Z_p)_{\mathrm{Eis}}^\epsilon = M_2(\ell_0;\Z_p)_{\mathrm{Eis}} + M_2(\ell_1;\Z_p)_{\mathrm{Eis}}$. This proves that $\mathfrak{a}_1 \to I_0$ is an isomorphism.

Using this isomorphism $\bT_N^{\epsilon} \isoto \bT_{\ell_0} \times_{\Z_p} \bT_{\ell_1}$ and Mazur's results (\S\ref{subsec:Mazur results}) on the structure of $\bT_{\ell_i}$, it is then a simple computation to see that
\[
\bT_N^{\epsilon}/p\bT_N^{\epsilon} \cong \F_p[y_0,y_1]/(y_0^{e_0+1},y_1^{e_1+1},y_0y_1), \quad \text{for some } e_0,e_1 >0.
\]
Thus $\Soc(\bT_N^{\epsilon}/p\bT_N^{\epsilon})=\F_p y_0^{e_0} \oplus \F_p y_1^{e_1}$. By Lemma \ref{lem:Gorenstein defects equality}, $\bT_N^{\epsilon}$ is not Gorenstein.
\end{proof}

\section{Generators of the Eisenstein ideal}
\label{sec:GP}

In this section, we prove Part (4) of Theorem \ref{thm:main r primes} about the number of generators of the Eisenstein ideal, as well as Theorems \ref{thm:main good primes r} and \ref{thm:main good primes 2}, about specific generators.

\subsection{Determining the number of generators of $I^\epsilon$ when $\varep = (-1,1,\dotsc, 1)$}
\label{subsec:splitting}

In this subsection, we prove Part (4) of Theorem \ref{thm:main r primes}. Assume we are in the setting of that theorem, so $\epsilon=(-1,1,\dots,1)$. Recall the fields $K_i$ of Definition \ref{defn:K_i}.

\begin{thm}
	\label{thm:star split}
Assume that $\ell_i \equiv -1 \pmod{p}$ for $i=1,\dots, r$. The minimal number of generators of $I^\epsilon$ is $r+\delta$ where
\begin{equation}
\label{eq:delta}
\delta =\left\{
\begin{array}{lc}
1 & \text{ if } \ell_0 \text{ splits completely in } K_i \text{ for } i=1,\dots,r \\
0 & \text{ otherwise.}
\end{array}\right.
\end{equation}
\end{thm}
This immediately implies Part (4) of Theorem \ref{thm:main r primes} by Lemma \ref{lem:reduction to lower level}. For the rest of \S \ref{subsec:splitting}, we assume that $r>0$ and $\ell_i \equiv -1\pmod{p}$ for $i=1,\dots, r$, and we use $\delta$ to refer to the integer \eqref{eq:delta}.

\subsubsection{Outline of the proof} By Theorem \ref{thm:star main}, we see that $R_N^\epsilon \to \bT_N^\epsilon$ is an isomorphism and that it induces an isomorphism $\Jm \isoto I^\epsilon$. Hence we are reduced to computing the number of generators of $\Jm$. Moreover, \eqref{eq:star ses} implies that this number of generators is either $r$ or $r+1$, and we are reduced to showing that it is $r+1$ if and only if the splitting condition in \eqref{eq:delta} holds. 

We do some initial reductions in \S\ref{subsubsect:first reductions}. We use class field theory to show that the splitting condition in \eqref{eq:delta} is equivalent to the vanishing of certain cup products in Galois cohomology. The number of generators of $\Jm$ is the same as the dimension of the tangent space of $R_N^\epsilon/pR_N^\epsilon$, and this is related to cup products. Explicitly, Bella\"iche proved in \cite[Thm.\ A]{bellaiche2012} that the the tangent space $\frt_\Db$ of $R_\Db/pR_\Db$ (where $R_\Db$ is the unrestricted pseudodeformation ring of $\Db$ -- without any local conditions) fits in an exact sequence
\[\xymatrix{
 \frt_\Db \ar[r]^-{\iota} & H^1(\F_p(1)) \otimes_{\F_p} H^1(\F_p(-1)) \ar[rrr]^-{b' \otimes c' \mapsto (b' \cup c', c' \cup b')} & & & H^2(\F_p) \oplus H^2(\F_p).
}\]
In other words, the tangent space is larger as more cup products vanish. We show that this immediately implies one implication of Theorem \ref{thm:star split}: if the number of generators of $\Jm$ is $r+1$, this forces the tangent space of $\frt_\Db$ to be large, which can only happen if the cup products vanish.

The other implication is more delicate. If we assume that the cup products vanish, then Bella\"{i}che's theory only tells us that the tangent space of the unrestricted deformation ring is large. We have to show that these first-order deformations can be made to satisfy the right local conditions at primes dividing $Np$. To do this, we construct in \S \ref{sssec:constr rho_M} a particular GMA representation $\rho_M$ that realizes Bella\"{i}che's tangent space computation, and show that $\rho_M$ satisfies the $\mathrm{US}^\epsilon_N$ local conditions.  In \S\ref{sssect:end of proof}, we show that $\rho_M$ indeed realizes the tangent space of $R_N^\epsilon/pR_N^\epsilon$, and this completes the proof.

\subsubsection{First reductions}
\label{subsubsect:first reductions} Note that because $\m=\Jm+pR_N^\epsilon \subset R_N^\epsilon$ is the maximal ideal, we have
\[
\Jm/\m\Jm \cong \m/(p,\m^2).
\]
By Nakayama's lemma, the minimal cardinality of a generating subset of $\Jm$ is $\dim_{\F_p}\m/(p,\m^2)$. By Theorem \ref{thm:star main} we have $I^\epsilon \cong \Jm$, so, to prove Theorem \ref{thm:star split}, it suffices to show that $\dim_{\F_p}\m/(p,\m^2) =r+\delta$, and this is what we will prove.

Recall the notation of \S\ref{subsec:label}, in particular, the class $b_0 \in H^1(\Z[1/Np],\F_p(1))$ and the representing cocycle $\tilde{b}_0$, as well as the classes $c_0,\dots,c_r \in  H^1_{(p)}(\Z[1/Np],\F_p(-1))$ (note that $c_0$ is only defined if $\ell_0 \equiv \pm 1 \pmod{p}$). The starting point is the following proposition, which is proven in Appendix \ref{appendix:cohomology}. 
\begin{prop}
\label{prop:bc cup}
Let $i\in\{1,\dots,r\}$. Then $\ell_0$ splits completely in $K_i$ if and only if $\ell_0 \equiv 1 \pmod{p}$ and $b_0 \cup c_i$ vanishes in $H^2(\Z[1/Np],\F_p)$. 
\end{prop}

We can now prove one implication of Theorem \ref{thm:star split}.
\begin{prop}
\label{prop:r+1 generators implies delta=1}
Suppose that the minimal number of generators of $I^\epsilon$ is $r+1$. Then $\delta=1$.
\end{prop}
\begin{proof}
By Theorem \ref{thm:star main}, we see that minimal number of generators of $I^\epsilon$ is $r+1$ if and only if $\ell_0 \equiv 1 \pmod{p}$ and the images of the elements $b_{\gamma_0}c_{\gamma_i}$ for $i=1,\dots,r$ in $\m/(p,\m^2)$ are linearly independent. In particular, for each $i$, the image of $b_{\gamma_0}c_{\gamma_i}$ in $\m/(p,\m^2)$ is non-zero. Fix such an $i$, and let (writing $\F_p[\varep]$ for $\F_p[\varep]/(\varep^2)$)
\[
\alpha: R_N^\epsilon/(p,\m^2) \to \F_p[\varep]
\]
be a ring homomorphism sending $b_{\gamma_0}c_{\gamma_i}$ to $\varep$.

Let $E=\sm{\F_p[\varep]}{\F_p}{\F_p}{\F_p[\varep]}$ be the $\F_p[\varep]$-GMA with data $(\F_p,\F_p,m)$ where $m:\F_p \times \F_p \to \F_p[\varep]$ is the map $(x,y) \mapsto xy\varep$. By Lemma \ref{lem:dual B C and H1}, we have a homomorphism of GMAs
$A:E_N^\epsilon \to E$ given by 
\[
A=\ttmat{\alpha}{\tilde b_0}{\tilde c_i}{\alpha}.
\]
Let $D_A=\psi(A \circ \rho_N^\epsilon):G_{\Q,S} \to \F_p[\varep]$ be the corresponding deformation of $\Db$. Then $D_A$ contributes a non-zero element to the tangent space $\mathfrak{t}_\Db$ of $R_\Db/pR_\Db$. Examining \cite{bellaiche2012}, the image of $D_A$ under $\iota$ in the exact sequence of \cite[Thm.\ A]{bellaiche2012}
\[\xymatrix{
 \frt_\Db \ar[r]^-{\iota} & H^1(\F_p(1)) \otimes_{\F_p} H^1(\F_p(-1)) \ar[rrr]^-{b' \otimes c' \mapsto (b' \cup c', c' \cup b')} & & & H^2(\F_p) \oplus H^2(\F_p)
}\]
is $b_0 \otimes c_i$, and hence $b_0 \cup c_i=0$. Since this is true for all $i$, Proposition \ref{prop:bc cup} implies that $\delta=1$.
\end{proof}

The remainder of the proof of Theorem \ref{thm:star split} relies on the following construction. 

\subsubsection{Construction of a maximal first-order pseudodeformation}
\label{sssec:constr rho_M}

Let $H$ be the kernel of the map
\begin{align*}
\begin{split}
H^1_{(p)}(\Z[1/Np],\F_p(-1)) &\lra H^2(\Z[1/Np],\F_p) \oplus H^1(I_{\ell_0}, \F_p(-1)), \\
 x &\mapsto (b_0 \cup x, \ x\vert_{I_{\ell_0}}).
\end{split}
\end{align*}

\begin{lem}
\label{lem:basis of H}
If $\ell_0 \equiv 1 \pmod{p}$ and $\delta=0$, then $b_0 \cup c_i  \ne 0$ for some $i$. In that case, there are elements $\alpha_j \in \F_p$ such that the set $\{c_j-\alpha_jc_i\}$ for $j \in \{1,\dots,r\}\setminus\{i\}$ is a basis for $H$. Otherwise, the set $\{c_1,\dots,c_r\}$ is a basis for $H$.
\end{lem}
\begin{proof}
The first statement follows from Proposition \ref{prop:bc cup}. Recall that $c_i$ is ramified at $\ell_0$ if and only if $i = 0$, so $H$ is contained in the span of the linearly independent set $\{c_1,\dots,c_r\}$. Since 
\[
\dim_{\F_p} H^2(\Z[1/Np],\F_p) = \left\{
\begin{array}{ll}1 &\text{ when } \ell_0 \equiv 1 \pmod{p} \\
0 &\text{ when } \ell_0 \not\equiv 1 \pmod{p},
\end{array}\right.
\]
the lemma follows.
\end{proof}

\begin{lem}
\label{lem:H cocycles trivial on ell0}
If $\ell_0 \ne p$ and $h \in H$, the image $h\vert_{G_{\ell_0}} \in H^1(\Q_{\ell_0}, \F_p(-1))$ is zero.
\end{lem}
\begin{proof}
If $\ell_0 \not \equiv \pm 1 \pmod{p}$, then $H^1(\Q_{\ell_0}, \F_p(-1))=0$. If $\ell_0 \equiv -1 \pmod{p}$, then $H^1(\Q_{\ell_0}, \F_p(-1))=H^1(\Q_{\ell_0}, \F_p(1))$, and so this follows from Lemma \ref{lem:Kummer injective}. Now assume $\ell_0 \equiv 1 \pmod{p}$. Then, since $h \cup b_0=0$ in $H^2(\Z[1/Np],\F_p)$, $b_0$ is ramified at $\ell_0$, and $h$ is unramified at $\ell_0$, Lemma \ref{lem:local tate desc} implies that $h\vert_{G_{\ell_0}} = 0$. \end{proof}

\begin{constr}
\label{constr:C and F}
We construct a cocycle $C: G_{\Q,S} \to H^*(-1)$, where $H^*=\Hom_{\F_p}(H,\F_p)$ with trivial $G_{\Q,S}$-action, and a cochain $F:G_{\Q,S} \to H^*$ such that:
\begin{enumerate}
\item $C|_{G_p}=0$,
\item if $\ell_0 \ne p$, then $C|_{G_{\ell_0}}$ is a coboundary,
\item $dF = \tilde{b}_0 \smile C$,
\item $F|_{I_p}= 0$,
\item For any cocycle $\tilde{h}$ whose cohomology class $h$ is in $H$, and any $\tau \in G_{\Q,S}$ with $\omega(\tau)=1$, we have $C(\tau)(h)=\tilde{h}(\tau)$.
\end{enumerate}
\end{constr}
\begin{proof}
For any $G_{\Q,S}$-module $X$, let 
\[
Z^1_{(p)}(\Z[1/Np],X) =\{(a,x) \in Z^1(\Z[1/Np],X) \times X  \ | \ a(\tau)=(\tau-1)x, \ \forall\,  \tau \in G_p\}.
\]
There is a surjection $Z^1_{(p)}(\Z[1/Np],\F_p(-1)) \onto H^1_{(p)}(\Z[1/Np],\F_p(-1))$ sending $(a,x)$ to the class of $a$. Choose a linear section $s: H \rinj  Z^1_{(p)}(\Z[1/Np],\F_p(-1))$, and write $s(h)=(s(h)_1,s(h)_2) \in Z^1(\Z[1/Np],\F_p(-1)) \times \F_p(-1)$. 

Define an element $(C',x) \in C^1(\Z[1/Np],H^*(-1)) \times H^*(-1)$ by $C'(\tau)(h)=s(h)_1(\tau)$ and $x(h)=s(h)_2$ for $h \in H$. One observes $(C',x)\in Z^1_{(p)}(\Z[1/Np],H^*(-1))$. Then let $C=C'-dx$, so that $C|_{G_p}=0$ and (1) holds. We also see that (5) holds, since the value $\tilde{h}(\tau)$ is independent of the choice of cocycle. Computing with dual vector spaces, it is easy to see that $b_0 \cup C = 0$ in $H^2(\Z[1/Np],H^*)$ and that Lemma \ref{lem:H cocycles trivial on ell0} implies (2).

Finally, to see (3) and (4), let $y$ be any cochain such that $dy=\tilde{b}_0 \smile C$. Note that the restriction map 
\[
H^1(\Z[1/Np],H^*)\to H^1(I_p,H^*)
\]
is surjective, and that, since $H^*$ has trivial action, we may and do identify a cohomology class with its representing cocycle. Since $C|_{I_p}=0$ and $dy=\tilde{b}_0 \smile C$, we see that $y|_{I_p} \in H^1(I_p,H^*)$. Hence there is a cocycle $y' \in H^1(\Z[1/Np],H^*)$ with $y'|_{I_p}=y|_{I_p}$. Letting $F=y-y'$, we have $dF=dy=\tilde{b}_0 \smile C$ and $F|_{I_p}=0$. 
\end{proof}

Let $M = H^* \oplus \Z/(p, \ell_0-1)$, and let $\F_p[M]$ be the vector space $\F_p \oplus M$ thought of as a local $\F_p$-algebra with square-zero maximal ideal $M$. We write elements of $\F_p[M]$ as triples $(x,y,z)$ with $x \in \F_p$, $y \in H^*$ and $z \in  \Z/(p, \ell_0-1)\Z$.

Let $E_M$ be the $\F_p[M]$-GMA given by the data $(\F_p,H^*,m)$ where $m$ is the homomorphism
\[
m: \F_p \otimes_{\F_p} H^* \cong H^* \risom H^* \oplus \{0\} \subset M \rinj \F_p[M].
\]
Let $\rho_M: G_{\Q,S} \lra E_M^\times$ be the function
\begin{equation}
\label{eq:rho_M}
\rho_M(\tau) = \ttmat{\omega(\tau)(1,F(\tau),\log_{\ell_0}(\tau))}{\tilde b_0(\tau)}{\omega(\tau) C(\tau)} {(1,\tilde{b}_0(\tau)C(\tau)-F(\tau),-\log_{\ell_0}(\tau))}.
\end{equation}
Then $\rho_M$ is a homomorphism by Construction \ref{constr:C and F}. Let $D_M : G_{\Q,S} \ra \F_p[M]$ denote the pseudorepresentation $D_M := \psi(\rho_M)$. 

\begin{lem}
\label{lem:rho_M US}
$\rho_M$ satisfies $\mathrm{US}^\epsilon_N$. 
\end{lem}

\begin{proof} 
As per Definition \ref{defn:US global}, we verify $\mathrm{US}^\epsilon_N$ by proving that $\rho_M\vert_{G_p}$ is finite-flat if $\ell_0 \ne p$, and that $\rho_M\vert_{G_\ell}$ satisfies condition $\mathrm{US}^{\epsilon_\ell}_\ell$ for all $\ell \mid N$. 

\noindent
\textbf{If $\ell_0 \ne p$, $\rho_M\vert_{G_p}$ is finite-flat:} For this, we will make frequent use of the 
notion of a Cayley--Hamilton module, developed in \cite[\S2.6]{WWE4}.

Let $E_M'$ be the $\F_p[M]$-sub-GMA of $E_M$ given by $E'_M=\sm{\F_p[M]}{\F_p}{0}{\F_p[M]}$. Since $C|_{G_p}=0$, we see that the action of $G_p$ on $E_M$ via $\rho_M$ factors through $E_M'$. Hence $(\rho_M|_{G_p} : G_p \ra E_M^{\prime\times}, E'_M, D_{E'_M} : E'_M \ra \F_p[M])$, which we denote by $\rho'_{M,p}$ for convenience, is a Cayley--Hamilton representation of $G_p$. Then $E_M$ is a faithful Cayley--Hamilton module of $\rho'_{M,p}$; by \cite[Thm.\ 2.6.3]{WWE4}, it is enough to show that $\rho'_{M,p}$ is finite-flat.

Consider the extension $\cE_{\tilde{b}_0}$ defined by $\tilde{b}_0$:
\[
0 \lra \F_p(1) \lra \cE_{\tilde{b}_0} \lra \F_p \lra 0,
\]
which is finite-flat by Kummer theory. Let $W_\omega=\F_p[M]$ and $W_1=\F_p[M]$ with $G_p$ acting by the characters $\omega(1,F,\log_{\ell_0})$ and $(1,-F,-\log_{\ell_0})$, respectively. Since $F|_{I_p}$ and $\log_{\ell_0}|_{I_p}$ are zero, $W_\omega$ and $W_1$ are finite-flat. We have exact sequences of $\F_p[M][G_p]$-modules
\[
0 \to M(1) \to W_\omega \to \F_p(1) \to 0, \qquad 0 \to M \to W_1 \to \F_p \to 0.
\]
Let $l: \F_p \rinj M$ be an injective linear map. This induces a injection $\F_p(1) \rinj W_\omega$ of $\F_p[M][G_p]$-modules. Taking the pushout of $\cE_{\tilde{b}_0}$ by this injection, we obtain an exact sequence
\[
0 \lra W_\omega \lra \cE_{\tilde{b}_0,\omega} \lra \F_p \lra 0.
\]
Pulling back this sequence by $W_1 \rsurj \F_p$, we obtain an exact sequence
\[
0 \lra W_\omega \lra \cE_{\tilde{b}_0,\omega,1} \lra W_1 \lra 0.
\]
Following \cite[App.\ C]{WWE3}, we see that $\cE_{\tilde{b}_0,\omega,1}$ is finite-flat and that there is an isomorphism $\cE_{\tilde{b}_0,\omega,1} \cong \F_p[M]^{\oplus 2}$ under which the action of $G_p$ is given by
\begin{equation}
\label{eq:action on V}
\left.\ttmat{\omega(1,F,\log_{\ell_0})}{(0,\tilde{b}_0 \cdot l(1)) }{0} {(1,-F,-\log_{\ell_0})}\right\vert_{G_p} : G_p \ra \GL_2(\F_p[M]).
\end{equation}
We now use this isomorphism $\cE_{\tilde{b}_0,\omega,1} \cong \F_p[M]^{\oplus 2}$ as an identification.

We have an injective $\F_p[M]$-GMA homomorphism $l': E_M' \to \End_{\F_p[M]}(\cE_{\tilde{b}_0,\omega,1})=M_{2\times 2}(\F_p[M]))$ given by
\[
l'=\ttmat{\mathrm{id}_{\F_p[M]}}{l}{0}{\mathrm{id}_{\F_p[M]}} .
\]
By \eqref{eq:action on V}, we see that action of $G_p$-action on $\cE_{\tilde{b}_0,\omega,1}$ factors through $l'$. In other words, $\cE_{\tilde{b}_0,\omega,1}$ is a faithful Cayley--Hamilton module of $\rho'_{M,p}$. Since $\cE_{\tilde{b}_0,\omega,1}$ is finite-flat, $\rho'_{M,p}$ is finite-flat by \cite[Thm.\ 2.6.3]{WWE4}.

\noindent
\textbf{If $\ell_0 = p$, then $\rho_M|_{G_p}$ is ordinary:} This follows from Proposition \ref{prop:ord C-H form} and Construction \ref{constr:C and F}.

\noindent
\textbf{If $\ell_0 \equiv 1 \pmod{p}$, then $\rho_M|_{G_{\ell_0}}$ is $\mathrm{US}_{\ell_0}^{-1}$:} Since $\ell_0 \equiv 1 \pmod{p}$, $\omega|_{G_{\ell_0}}=1$. By Construction \ref{constr:C and F}, we have $C|_{G_{\ell_0}}=0$. Then, for any $\sigma, \tau \in G_{\ell_0}$, we have
\[
V^{-1}_{\rho_M}(\sigma, \tau) := (\rho_M(\sigma) - \omega(\sigma))(\rho(\tau) - 1) = \ttmat{\varep_1}{\tilde b_0(\sigma)}{0}{\varep_2}\ttmat{\varep_3}{\tilde b_0(\tau)}{0}{\varep_4} = 0,
\]
where $\varep_i \in M \subset \F_p[M]$.

\noindent
\textbf{If $\ell_0 \not \equiv 0,1 \pmod{p}$, then $\rho_M|_{G_{\ell_0}}$ is $\mathrm{US}_{\ell_0}^{-1}$:} 
By assumption, we have $M=H^*$ and $\log_{\ell_0}=0$, so  we write elements of $\F_p[M]$ as pairs $(x,y)$ with $x\in \F_p$ and $y \in H^*$. Since $C|_{G_{\ell_0}}$ is a coboundary, there exists $z \in H^*$ such that $C(\tau)=(\omega^{-1}(\tau)-1)z$ for all $\tau \in G_{\ell_0}$. 

Let $\rho_M':G_{\Q,S} \to E_M^\times$ be the composite of $\rho_M$ with the automorphism $E_M \risom E_M$ given by conjugation by $\sm{1}{0}{z}{1} \in E_M^\times$. By explicit computation, we see that 
\[
\rho_M' = \ttmat{\omega(1,F_a)}{\tilde b_0}{\omega(C-(\omega^{-1}-1)z)}{(1,F_d)},
\]
where $F_a=F-\omega^{-1}\tilde{b}_0z$ and $F_d=\tilde{b}_0C-F+\omega\tilde{b}_0 z$; in particular, the $(2,1)$-coordinate of $\rho_M'|_{G_{\ell_0}}$ is zero. This implies that $F_a\vert_{G_{\ell_0}}, F_d\vert_{G_{\ell_0}} : G_{\ell_0} \ra H^*$ are homomorphisms. Because $\ell_0 \not\equiv 0,1 \pmod{p}$ and $H^*$ has exponent $p$, they are unramified. 

For any $(\sigma,\tau) \in G_{\ell_0}\times I_{\ell_0}$, we compute that
\[
V^{-1}_{\rho'_M}(\sigma,\tau) = 
\ttmat{\varep}{*}{0}{*}
\ttmat{0}{*}{0}{0} = 0
\]
where $\varep \in M$. Equivalently, $V^{-1}_{\rho_M} = 0$. A similar computation shows that $V^{-1}_{\rho_M}(\sigma,\tau)=0$ for $(\sigma,\tau) \in I_{\ell_0}\times G_{\ell_0}$.

\noindent
\textbf{If $\ell \mid N$ and $\ell \ne \ell_0$, then $\rho_M|_{G_\ell}$ is $\mathrm{US}_{\ell}^{+1}$:}
In this case we have $\ell \equiv -1 \pmod{p}$, and hence $\omega|_{G_\ell}=\lambda(-1)$. Since $\ell \ne \ell_0$, we have $b_0|_{I_\ell}=0$, so $b_0|_{G_\ell}=0$ by Lemma \ref{lem:Kummer injective}. Hence there exists $z \in \F_p$ such that $\tilde{b}_0(\tau)=(\omega(\tau)-1)z$ for all $\tau \in G_\ell$. Exactly as in the previous case, we can show that $V^{+1}_{\rho_M}(\sigma,\tau)=0$ for all $(\sigma,\tau) \in G_{\ell}\times I_\ell \cup I_{\ell}\times G_\ell$ by conjugating $\rho_M$ by $\sm{1}{z}{0}{1} \in E_M^\times$. 
\end{proof}

\subsubsection{End of the proof} 
\label{sssect:end of proof}We will show that $D_M$ is, in a sense, the universal $\mathrm{US}_N^\epsilon$ first-order deformation of $\Db$.

\begin{prop}
\label{prop:first-order ps}
The pseudodeformation $D_M$ of $\Db$ induces an isomorphism $R^\epsilon_N/(p, \m^2) \risom \F_p[M]$.
\end{prop}
\begin{proof}
By Lemma \ref{lem:rho_M US}, $\rho_M$ is $\mathrm{US}_N^\epsilon$, so $D_M$ is also $\mathrm{US}_N^\epsilon$ by Definition \ref{defn:US global}, and there is an induced map $E_N^\epsilon \to E_M$. This gives us a local homomorphism $R^\epsilon_N \ra \F_p[M]$, and any such map factors through $R^\epsilon_N/(p, \m^2) \to \F_p[M]$. Let $f$ denote the restriction $\m/(p, \m^2) \to M$. It suffices to show that $f$ is an isomorphism.

Assume that the GMA structure on $E_N^\epsilon$ is chosen so that $E_N^\epsilon \to E_M$ is a morphism of GMAs (such a GMA structure is known to exist by \cite[Thm.\ 3.2.2]{WWE4}). By Theorem \ref{thm:star main}, we see that the elements $b_{\gamma_0}c_{\gamma_i}$ for $i=1,\dots,r$ together with the element $d_{\gamma_0}-1$ generate $\m/(p,\m^2)$, and, moreover, if $\ell_0\not \equiv 1 \pmod{p}$, the elements  $b_{\gamma_0}c_{\gamma_i}$ for $i=1,\dots,r$ are a basis.

By construction, we see that $f(b_{\gamma_0}c_{\gamma_i})=(0,\tilde{b}_0(\gamma_0)C(\gamma_i),0)=(0,C(\gamma_i),0)$, and that $f(d_{\gamma_0}-1)=(0,0,-\log_{\ell_0}(\gamma_0))$ (which is non-zero if $\ell_0\equiv 1 \pmod{p}$). By Lemma \ref{lem:C spans H*} below, $f$ is surjective.

Now we count dimensions. By Theorem \ref{thm:star main} and Proposition \ref{prop:r+1 generators implies delta=1}, we have
\[
\dim_{\F_p}(\m/(p, \m^2)) = \left\{ \begin{array}{ll}
r & \text{ if } \delta=0 \\
r \text{ or } r+1 & \text{ if } \delta=1.
\end{array}\right.
\]
By Lemma \ref{lem:basis of H}, we have
\begin{equation}
\label{eq:dim of M}
\dim_{\F_p}(M) = \left\{ \begin{array}{ll}
r & \text{ if }  \delta=0 \\
r+1 & \text{ if } \delta=1.
\end{array}\right.
\end{equation}
Since $f$ is surjective, this implies that $f$ is an isomorphism in all cases.
\end{proof}

\begin{lem}
\label{lem:C spans H*}
Let $\tau_1,\dots,\tau_r \in G_{\Q,S}$ be any elements such that:
\begin{itemize}
\item $\omega(\tau_i)=1$ for $i=1,\dots, r$, and
\item $\tilde{c}_j(\tau_i)=\partial_{ij}$ for all $1 \le i,j \le r$.
\end{itemize} 
If $\delta=1$ or $\ell_0 \not \equiv 1 \pmod{p}$, then the set $\{C(\tau_i) \ : \ i=1,\dots,r\}$ is a basis for $H^*$. Otherwise $b_0 \cup c_j \neq 0$ for some $j$ and the set $\{C(\tau_i) \ : \ i=1,\dots,r, i \ne j\}$ is a basis for $H^*$.
\end{lem}
\begin{proof}
Indeed, if $c_j-\alpha c_k \in H$ for some $\alpha \in \F_p$ and $j,k\in \{1,\dots,r\}$, then by Construction \ref{constr:C and F}(5) we have
\[
C(\tau_i)(c_j-\alpha c_k) = \tilde{c}_j(\tau_i)-\alpha \tilde{c}_k(\tau_i) = \partial_{ij}-\alpha\partial_{ik}.
\]
Using the explicit basis of $H$ constructed in Lemma \ref{lem:basis of H}, the lemma follows.
\end{proof}

\begin{proof}[Proof of Theorem \ref{thm:star split}]
By Proposition \ref{prop:first-order ps}, we have $\m/(p, \m^2) \risom M$, and the dimension of $M$ is given by \eqref{eq:dim of M}. This completes the proof.
\end{proof}

\subsection{Good sets of primes in the case $\epsilon = (-1, 1, \dotsc, 1)$}
\label{subsec:good primes for r primes}

In this section, we prove Theorem \ref{thm:main good primes r}. Recall Definition \ref{defn:good primes} for the meaning of the set of good primes $\cQ$.

\begin{proof}[Proof of Theorem \ref{thm:main good primes r}]
We freely refer to $\rho_M$ and related objects in this proof (see \eqref{eq:rho_M}). Let $J$ be the index set of $\cQ$ (i.e.~ $J=\{0,\dots,s\}$, $J=\{0,\dots,s\}\setminus\{j\}$ or $J=\{1,\dots,s\}$ in the three cases of Definition \ref{defn:good primes}, respectively).

By Theorem \ref{thm:star main}, Proposition \ref{prop:first-order ps}, and Nakayama's lemma, it suffices to show that the projection $\Upsilon(q)$ of $T_q - (q+1)$ under $\bT^\epsilon_N \risom R^\epsilon_N \rsurj \F_p[M]$ comprise a basis $\{\Upsilon(q)\}_{q \in \cQ}$ of $M$. The conditions (1)-(6) on $\cQ$ have been chosen so that:
\begin{enumerate}[label=(\roman*), leftmargin=2em]
\item If $0 \in J$ and $q_0 \neq p$, then $\omega(\Fr_{q_0}) \neq 1$ and $\log_{\ell_0}(\Fr_{q_0}) \neq 0$. This follows from (1) and (2).

\item $\omega(\Fr_{q_i}) = 1$ for $i \in J$ with $i>0$. This follows from condition (3).

\item $\tilde b_0(\Fr_{q_i}) \neq 0$ for $i \in J$ with $i>0$. This follows from (4) by class field theory.

\item $\{C(\Fr_{q_i}) : i\in J, i>0\}$ is a basis for $H^*$. This follows from Lemma \ref{lem:C spans H*} by (ii), (5), and (6).
\end{enumerate}

When $q_i \neq p$, it is clear that $\Upsilon(q_i) = \Tr\rho_M(\Fr_{q_i}) - (q_i + 1)$, and we calculate: 
\begin{enumerate}[label=(\alph*), leftmargin=2em]
\item By (ii), $\Upsilon(q_i) = (0,\tilde b_0(\Fr_{q_i})\cdot C(\Fr_{q_i}),0) \in \F_p[M]$ for $i \in J$ with $i>0$. By (iii) and (iv), these elements form a basis of $H^*$.
\item If $0 \in J$ and $q_0 \neq p$, then $\Upsilon(q_0) \in \F_p[M]$ lies in $M$ and projects via $M \rsurj \Z/(p,\ell_0-1)$ to $(\omega(\Fr_{q_0}) - 1)\log_{\ell_0}(\Fr_{q_0})$. This is non-zero, by (i).
\item If $0 \in J$ and $q_0 = p$, we claim that $\Upsilon(p) \in \F_p[M]$ lies in $M$ and  maps to $\log_{\ell_0} p \neq 0$ under the summand projection $M \rsurj \Z/(p, \ell_0-1)$. This follows from the same argument as in Case $q_0 = p$ of the proof in \S\ref{subsec:prove thm good primes 2}, but is simpler. \qedhere
\end{enumerate}
\end{proof}

\begin{rem}
The reader will note that, in this proof, our conditions are used to ensure that a certain matrix is diagonal with non-zero diagonal entries. Of course, the necessary and sufficient condition is simply that this same matrix is invertible.
\end{rem}

\subsection{Good pairs of primes in the case $\epsilon = (-1,-1)$} 
\label{subsec:prove thm good primes 2}

In this section, we prove Theorem \ref{thm:main good primes 2}. We assume we are in the setting of Theorem \ref{thm:thm2}.

\begin{proof}[Proof of Theorem \ref{thm:main good primes 2}]
By Theorem \ref{thm:thm2} and Nakayama's lemma, $\bT^\epsilon_N$ is generated by $\{T_{q_i} - (q_i+1)\}_{i=0,1}$ if and only if their images $\{\Upsilon(q_i)\}_{i=0,1}$ via $\bT^\epsilon_N \risom R^\epsilon_N \rsurj R^\epsilon_N/(p, \m^2)$ are a basis of $\m/(p,\m^2)$. We see in the proof of Theorem \ref{thm:thm2} that $J^\red =\Jm^2$. In particular, as $\m = \Jm + (p) \subset R^\epsilon_N$,  there are isomorphisms
\[
R_N^\epsilon/(p,\m^2) \risom R_N^{\epsilon,\red}/(p) \cong \F_p[Y_0,Y_1]/(Y_0^2,Y_0Y_1,Y_1^2), \quad \m/(p,\m^2) \risom (Y_0, Y_1), 
\]
which we use as identifications. Then $D_N^\epsilon$ pulls back to the pseudorepresentation $D=\psi(\omega\dia{-}^{-1}\oplus \dia{-}) : G_{\Q,S} \ra R^{\epsilon,\red}_N/(p)$, where, for particular choices of $\log_{\ell_i}$, 
\[
G_{\Q,S} \ni \tau \mapsto \dia{\tau} := 1+\log_{\ell_0}(\tau)Y_0+\log_{\ell_1}(\tau)Y_1 \in (R_N^{\epsilon,\red}/(p))^\times.
\]
We see that if $q_i \neq p$, then $\Upsilon(q_i) = \Tr_D(\Fr_{q_i}) - (q_i+1)$. 

\noindent 
\textbf{Case $q_0,q_1 \neq p$.} One computes that the matrix expressing $\{\Upsilon(q_0), \Upsilon(q_1)\}$ in the basis $\{Y_0,Y_1\}$ of $\m/(p,\m^2) \cong (Y_0, Y_1)$ is
\[
\ttmat{(q_0-1)\log_{\ell_0} q_0}{(q_1-1)\log_{\ell_0} q_1}{(q_0-1)\log_{\ell_1} q_0}{(q_1-1)\log_{\ell_1} q_1} \in M_2(\F_p),
\]
which completes the proof. 

\noindent 
\textbf{Case $q_0 = p$.} We note that the images of $T_p-(p+1)$ and $U_p-1$ in $I^\epsilon/\m I^\epsilon$ are equal, so we may replace $T_p-(p+1)$ by $U_p-1$ in the statement. We recall from Step 3 of the proof of Proposition \ref{prop:R to T} that $U_p$ is the image under $R_N^\epsilon \risom \bT^\epsilon_N$ of $\frac{1}{x-1}(x\Tr(\rho_N^\epsilon)(\sigma_p)-\Tr(\rho_N^\epsilon)(\tau\sigma_p))$, where $\tau \in I_p$ is such that $\omega(\tau) \ne 1$ and $x=\kcyc(\tau)$. We compute that
\[
\Upsilon(p) = \frac{1}{x-1}\big(x\Tr_D(\sigma_p)-\Tr_D(\tau\sigma_p)\big) - 1 = \log_{\ell_0}(p)Y_0+\log_{\ell_1}(p)Y_1.
\]
Thus, the matrix expressing $\{\Upsilon(p), \Upsilon(q_1)\}$ in the basis $\{Y_0,Y_1\}$ of $\m/(p,\m^2)$ is
\[
\ttmat{\log_{\ell_0} p}{(q_1-1)\log_{\ell_0} q_1}{\log_{\ell_1} p}{(q_1-1)\log_{\ell_1} q_1} \in M_2(\F_p). \qedhere
\]
\end{proof}

\appendix

\section{Comparison with the Hecke algebra containing $U_\ell$}
\label{appendix:U and w}

In order to compare our results with existing results and conjectures, in this appendix we consider a Hecke algebra that contains the $U_\ell$ operators rather than the $w_\ell$ operators. We prove comparison results between Eisenstein completions of this algebra and the  Eisenstein completions $\bT_N^\epsilon$ studied in this paper. Throughout this appendix, we drop the subscripts `$N$' on all Hecke algebras to avoid cumbersome notation.

Recall that we have the normalization map of Lemma \ref{lem:normalization of bT0} 
\[
\bT^\epsilon \rinj \Z_p \oplus \left(\bigoplus_{f \in \Sigma} \sO_f\right),
\]
where $\Sigma, \sO_f$ were defined there. For each $f \in \Sigma$, there is a unique pair $(N_f,\tf)$ of a divisor $N_f$ of $N$ and a newform $\tf$ of level $N_f$ such that $a_q(f)=a_q(\tf)$ for all primes $q$ not dividing $N_f$ and $a_\ell(\tilde{f})=-\epsilon_\ell$ for primes $\ell$ dividing $N_f$. For this $\tf$, we have $a_q(\tf) \equiv 1+q \pmod{\m_f}$ for all $q \nmid N_f$.

\subsection{Oldforms and stabilizations}
\label{subsec:U stabilizations} Just as in \S \ref{subsub:stabilizations}, if $\ell \mid N$ and $f \in S_2(N/\ell; \Z_p)$ is an eigenform for all $T_n$ with $(n,N/\ell)=1$, then there are two ways to stabilize $f$ to be a $U_\ell$-eigenform in $S_2(N; \Z_p)$. Let $\alpha_\ell(f), \beta_\ell(f)$ denote the roots of $x^2-a_\ell(f)x+\ell$. Then $f_{\alpha_\ell}(z) = f(z)-\beta_\ell(f)f(\ell z)$ and $f_{\beta_\ell}(z) = f(z)-\alpha_\ell(f)f(\ell z)$
 satisfy $U_\ell f_{\alpha_\ell} = \alpha_\ell(f) f_{\alpha_\ell}$ and $U_\ell f_{\beta_\ell} = \beta_\ell(f) f_{\beta_\ell}$. Note that, unlike in \S \ref{subsub:stabilizations}, it may happen that $\alpha_\ell(f) \equiv \beta_\ell(f) \pmod{p}$.

\subsection{The case $p\nmid N$} 

For this section, assume $p \nmid N$. Let $\bT_U'$ and $\bT'^0_U$ be the $\Z_p$-subalgebras of
\[
\End_{\Z_p}(M_2(N;\Z_p)) \quad \text{and} \quad \End_{\Z_p}(S_2(N;\Z_p)),
\]
respectively, generated by the Hecke operators $T_\ell$ for $\ell \nmid N$ and $U_\ell$ for $\ell \mid  N$. These are semi-simple commutative algebras (see \cite{CE1998} for the semi-simplicity).

For each $\epsilon \in \cE$ as in \S\ref{sssec:eisen series and hecke alg defn}, we let $I'^\epsilon_U \subset \bT_U'$ be the ideal generated by the set
\[
\{T_q-(q+1),\ U_\ell-\ell^\frac{\epsilon_\ell+1}{2}  : q \nmid N, \ \ell \mid N \text{ primes} \}.
\]
Note that $U_\ell -1 \in I'^\epsilon_U$ if $\epsilon_\ell =-1$ and $U_\ell - \ell \in I'^\epsilon_U$ if $\epsilon_\ell =1$, so the ideal $I'^\epsilon_U$ is the annihilator of a certain stabilization of the Eisenstein series $E_{2,1}$ (but generally not $E_{2,N}^\epsilon$). Let $\bT_U^\epsilon$ and $\bT_U^{0,\epsilon}$ denote the completions of $\bT_U'$ and $\bT'^0_U$ respectively, at the maximal ideal $(p, I'^\epsilon_U) \subset \bT_U'$. Let $\m_U^\epsilon \subset \bT_U^\epsilon$ and $\m_U^{0,\epsilon} \subset \bT_U^{0,\epsilon}$ be the maximal ideals.

\subsubsection{The normalization of $\bT_U^\epsilon$} Since $\bT_U^\epsilon$ and $\bT_U^{0,\epsilon}$ are semi-simple, the standard description of prime ideals in terms of eigenforms allows us to describe their normalizations, just as for $\bT^\epsilon$. For newforms $f$, we know that $U_\ell f = -w_\ell f$ for all $\ell \mid N$. For oldforms, we can use the stabilization formulas from \S \ref{subsub:stabilizations} and \S\ref{subsec:U stabilizations} to describe the eigenforms for $\bT_U^\epsilon$ in terms of the set $\Sigma$. We write down the result of this description explicitly in Lemma \ref{lem:normalization of bT_U}.

We require the following notation. Let $L_N = \{\ell \mid N \ : \ \ell \equiv 1 \pmod{p}\}$. For each $f \in \Sigma$ and each $\ell \mid \frac{N}{N_f}$,  let $\alpha_\ell(f)$ and $\beta_\ell(f)$ be the roots of $x^2-a_\ell(\tf)x+\ell$. Assume that $\alpha_\ell(f) \equiv \ell^{\frac{\epsilon_\ell+1}{2}} \pmod{p}$ and let $L_f =\{\ell \mid  \frac{N}{N_f} : \ell \equiv 1 \pmod{p}\}$. Let $\tilde{\sO}_f$ be the extension of $\cO_f$ generated by $\alpha_\ell(f)$ and $\beta_\ell(f)$. If $\ell \not\equiv 1 \pmod{p}$, then the congruence condition determines $\alpha_\ell(f)$ (and $\beta_\ell(f)$) uniquely, and $\tilde{\sO}_f= \sO_f$; in this case, only stabilizations of $\tilde{f}_{\alpha_\ell}$ can appear in the completion $S_2(N;\Z_p) \otimes_{\bT'^0_U} \bT_U^\epsilon$. If $\ell \equiv 1 \pmod{p}$, then we label the two roots arbitrarily (in this situation, below, we will use the two roots symmetrically), and $\tilde{\sO}_f$ may be a quadratic extension of $\sO_f$; in this case the stabilizations of both $\tilde{f}_{\alpha_\ell}$ and $\tilde{f}_{\beta_\ell}$ can appear in the completion $S_2(N;\Z_p) \otimes_{\bT'^0_U} \bT_U^\epsilon$.

\begin{lem}
\label{lem:normalization of bT_U}
The normalization of $\bT_U^\epsilon$ is the injection 
\[
\bT_U^\epsilon \rinj \left( {\bigoplus'_{L \subset L_N}} \Z_p\right) \oplus \left(\bigoplus_{f \in \Sigma} \left(\bigoplus_{L \subset L_f} \tilde{\sO}_f\right)\right),
\]
where the primed summation in the first factor indicates that we omit the subset $L=L_N$ if $\epsilon_\ell = 1$ for all $\ell \not\in L_N$. The map is given by 
\[
T_q \mapsto ((1+q)_{L \subset L_N}, a_q(f)_{f \in \Sigma,L \subset L_f}) \text{ for all } q \nmid N,
\]
and sending $U_\ell$ for $\ell \mid N$ as follows. In the factor $\bT_U^\epsilon \to \Z_p$ for $L \subset L_N$, we have
\[
U_\ell \mapsto \begin{array}{lcl}
1 & \text{if} & \left\{\begin{array}{l}
\ell \not\in L_N \text{ and }  \epsilon_\ell =-1, \text{ or}\\
\ell \in L_N-L \text{ and } \epsilon_\ell=-1, \text{ or}\\
\ell \in L \text{ and } \epsilon_\ell =1
\end{array}\right.\\
\end{array}
\]
and
\[
U_\ell \mapsto \begin{array}{lcl}
\ell & \text{if} & \left\{\begin{array}{l}
\ell \not\in L_N \text{ and }  \epsilon_\ell =1, \text{ or}\\
\ell \in L_N-L \text{ and } \epsilon_\ell=1, \text{ or}\\
\ell \in L \text{ and } \epsilon_\ell =-1
\end{array}\right.
\end{array}
\]
for all $\ell \mid N$. In the factor $\bT_U \to \tilde{\sO}_f$ corresponding to $f \in \Sigma$ and $L \subset L_f$, we have
\[
U_\ell \mapsto u_\ell(f) := 
\left\{
\begin{array}{lcl}
-\epsilon_\ell & \text{if } &\ell \mid N_f \\
\alpha_\ell(f) & \text{if  }& \ell \mid N, \ell \nmid N_f, \ell \not \in L \\
\beta_\ell(f) & \text{if }&\ell \mid N, \ell \nmid N_f, \ell \in L
\end{array} \right.
\]
for all $\ell \mid N$. 
\end{lem}
The only part of the lemma that is not completely standard is the factor $\left( \bigoplus'_{L \subset L_N} \Z_p\right)$, which corresponds to the Eisenstein series in $M_2(N)_\mathrm{Eis}^\epsilon$. For $\ell \mid N$, if $\ell \not\in L_N$, then the $U_\ell$-eigenvalue of any such Eisenstein series must be $\ell^{\frac{\epsilon_\ell+1}{2}}$, but if $\ell \in L_N$, then the possible $U_\ell$-eigenvalues $1$ and $\ell$ are congruent, and so both appear, regardless of what $\epsilon_\ell$ is. We need to omit $L=L_N$ in the case that $\epsilon_\ell=1$ for all $\ell \not\in L_N$ because that factor corresponds to the Eisenstein series with $U_\ell$-eigenvalue $\ell$ for all $\ell \mid N$, which is the non-holomorphic one.

The normalization of $\bT_U^{0,\epsilon}$ is the same, but without the Eisenstein factor $\left( \bigoplus_{L \subset L_N} \Z_p\right)$

\subsubsection{Comparisons} We now compare the algebras $\bT_U^{0,\epsilon}$ and $\bT^{0,\epsilon}$. The following proposition gives a necessary and sufficient condition for the algebras to coincide.

\begin{prop}
\label{prop:conditions for T=T_U}
Suppose that both of the following are true:
\begin{enumerate}
\item for each $f \in \Sigma$, we have $L_f = \emptyset$; and 
\item $\bT_U^{\epsilon,0}$ is generated as a $\Z_p$-algebra by $\{T_q \ : \ q \nmid Np\}$.
\end{enumerate}
Then $\bT_U^{\epsilon,0}=\bT^{\epsilon,0}$. Moreover, if one of these conditions is false, then $\bT_U^{\epsilon,0} \ne \bT^{\epsilon,0}$.
\end{prop}
\begin{proof}
The first condition ensures that $\bT_U^{\epsilon,0}$ and $\bT^{\epsilon,0}$ have the same normalization, so it is certainly necessary. The second condition is true for $\bT^{\epsilon,0}$ by Proposition \ref{prop:R to T}, so it is necessary. Furthermore, if we assume (1) and (2), then $\bT_U^{\epsilon,0}$ and $\bT^{\epsilon,0}$ are identified with the subalgebra of $\bigoplus_{f \in \Sigma} \sO_f$ generated by $\{(a_q(f)_f): q \nmid Np\}$.
\end{proof}

We now verify these conditions in certain cases considered in this paper. 
\begin{prop}
\label{prop:T=T_U first case}
Assume that $\ell_i \not \equiv 1 \pmod{p}$ for $0 < i \leq r$ and assume that $\epsilon=(-1,1,\dots,1)$. Then $\bT_U^{\epsilon,0}=\bT^{\epsilon,0}$.
\end{prop}
\begin{proof}
We verify the conditions (1) and (2) of Proposition \ref{prop:conditions for T=T_U}.

To verify (1), assume, for a contradiction, that there is an $f \in \Sigma$ with $L_f \ne \emptyset$. By our assumptions on $\ell_i$, we must have $L_f=\{\ell_0\}$. Then the newform $\tf \in S_2(N_f; \overline{\Q}_p)$ satisfies $a_q(\tf) \equiv 1+q \pmod{p}$ for all $q \nmid N_f$ and $a_\ell(\tf)=-1$ for all $\ell \mid N_f$ (since $\ell_0 \nmid N_f$ by assumption). But this is impossible by a theorem of Ribet (see \cite[Thm.\ 2.6(ii)(b)]{BD2014}), so (1) holds.

We now turn to (2). Just as in the proof of Proposition \ref{prop:R to T}, we have a homomorphism $R_N^\epsilon \to \bT_U^{0,\epsilon}$ sending $\Tr(\rho_N^{\epsilon}(\Fr_q))$ for $q \nmid Np$ to $T_q$, and whose image is the subalgebra of $\bT_U^0$ generated by $\{T_q: q \nmid Np\}$. Note that, by (1), for each $f \in \Sigma$ we have $\ell_0 \mid N_f$, so $u_{\ell_0}(f)=1$. This implies $U_{\ell_0}=1$ in $\bT_U^0$.  Hence to verify (2), we need only show that $U_\ell$ is in the image of $R_N^\epsilon \to \bT_U^{0,\epsilon}$ for all $\ell \mid N$ with $\ell \ne \ell_0$. 

Now fix such an $\ell$ and note that, by assumption, $\ell \not \equiv 1 \pmod{p}$ and $\epsilon_\ell=1$. Let $\tilde{U} \in R_N^\epsilon$ be the root of the polynomial $x^2-\Tr(\rho_N^{\epsilon}(\sigma_\ell))x+\ell$ such that $\tilde{U}-\ell \in \m_R$; such a $\tilde{U}$ exists and is unique by Hensel's Lemma. We claim that the image of $\tilde{U}$ in $\bT^{0,\epsilon}_U$ is $U_\ell$. To prove the claim, it suffices to show that $\tilde{U} \mapsto u_\ell(f)$ under the map $R_N^\epsilon \to \sO_f$ for each $f \in \Sigma$.

First assume that $\ell \mid N_f$. By \eqref{eq:Steinberg G_ell rep}, $\Tr(\rho_f(\sigma_\ell))=-(\ell+1)$. So $\tilde{U}$ is sent to the root of
\[
x^2+(\ell+1)x+\ell = (x+1)(x+\ell)
\]
(that is, either $-1$ or $-\ell$) 
that is congruent to $\ell$ modulo $\m_f$. Since $\ell \not \equiv -\ell \pmod p$, we see that $\tilde{U}$ is sent to $-1$. (Note that this shows that if $\ell \mid N_f$, then $\ell \equiv -1 \pmod{p}$, corroborating Lemma \ref{lem:reduction to lower level}.)

Next assume that $\ell \mid N$ and $\ell \nmid N_f$, so $\Tr(\rho_f(Fr_\ell))=a_\ell(\tf)$. Then $\tilde{U}$ is sent to the root of
\[x^2-a_\ell(\tf)x+\ell\]
that is congruent to $\ell$ modulo $\m_f$, which is $\alpha_\ell(f)$ by definition. 

This shows that, for $f \in \Sigma$ the map $R_N^\epsilon \to \sO_f$ sends $\tilde{U}$ to
\[
\left\{\begin{array}{lcl}
-1 & \text{if} & \ell \mid N_f \\
\alpha_\ell(f) & \text{if} & \ell \nmid N_f
\end{array}\right.
\]
which is equal to $u_\ell(f)$. Hence $U_\ell$ is the image of $\tilde{U}$ in $\bT^{0,\epsilon}_U$ and the map 
$R_N^\epsilon \to \bT_U^{0,\epsilon}$ is surjective, verifying (2).
\end{proof}

The proof of the first part of the following proposition is almost identical, but simpler, so we leave it to the reader. The second part is an application of Theorem \ref{thm:one interesting prime}.

\begin{prop}
Assume that $N=\ell_0\ell_1$, that $\ell_1 \not \equiv 1 \pmod{p}$, and that $\epsilon=(-1,-1)$. Then $\bT^{0,\epsilon}=\bT_U^{0,\epsilon}$. If, in addition, $\ell_1$ is not a $p$-th power modulo $\ell_0$, then $\bT^{0,\epsilon}$ and $\bT_U^{0,\epsilon}$ are both identical to the Hecke algebra at level $\ell_0$ considered by Mazur.
\end{prop}

\subsection{The case $p\mid N$} 
\label{subsec:p|N}

In this section, we maintain the notation of the previous section, but we assume that $\ell_0=p$ and that $\epsilon_0=-1$ (for $0< i \leq r$, $\epsilon_i$ is arbitrary).

We consider a variant $\bT_H^\epsilon$ of the Hecke algebra that is intermediate to $\bT^\epsilon$ and $\bT_U^\epsilon$. Namely, $\bT_H^\epsilon$ is the completion of the Hecke algebra generated by the $T_q$ for $q \nmid N$, together with $U_p$ and $w_{\ell_i}$ for $0 < i \leq r$, at the ideal generated by $p$, $\bT_q-(q+1)$, $U_p-1$, and $w_{\ell_i}-\epsilon_i$. Note that, as in the case of $\bT^\epsilon$, we have $w_{\ell_i}=\epsilon_i$ in $\bT_H^\epsilon$. For each $f \in \Sigma$, if $p \nmid N_f$, we let $\alpha_p(f) \in \sO_f$ be the (unique) unit root of $x^2-a_p(\tf)x+p$. 

Just as in Lemma \ref{lem:normalization of bT_U}, we can compute the normalization of $\bT^\epsilon_H$. It is the injective map 
\[
\bT_H^\epsilon \rinj \Z_p \oplus \left(\bigoplus_{f \in \Sigma} \sO_f\right)
\]
sending $T_q$ to $(1+q, a_q(f)_f)$ for $q \nmid N$ and $U_p$ as follows. The component $\bT_H^\epsilon \to \Z_p$ sends $U_p$ to $1$. The component $\bT_H^\epsilon \to \sO_f$ sends $U_p$ to $u_p(f)$ defined by
\[u_p(f):=\left\{
\begin{array}{lcl}
1 & \text{ if } & p \mid  N_f \\
\alpha_p(f) & \text{ if }& p \nmid N_f.
\end{array} \right.
\]

\begin{prop}
With the assumptions that $\ell_0=p$ and $\epsilon_0=-1$, we have $\bT_H^\epsilon=\bT^\epsilon$ as subalgebras of $\Z_p \oplus \left(\bigoplus_{f \in \Sigma} \sO_f\right)$ and $\bT_H^{\epsilon,0}=\bT^{\epsilon,0}$ as subalgebras of $\left(\bigoplus_{f \in \Sigma} \sO_f\right)$.
\end{prop}
\begin{proof}
The proof is just as in the proof of Proposition \ref{prop:T=T_U first case}, so we will be brief. We have a map $R_N^\epsilon \to \bT_H^\epsilon$ and we need only show that $U_p$ is in the image of this map. Now choose $\sigma_p\in G_p$ to be a Frobenius element such that $\omega(\sigma_p)=1$, and let $\tilde{U} \in R_N^\epsilon$ be the unique unit root of $x^2-\Tr(\rho_N^{\epsilon}(\sigma_p))x+p$. We see that $\tilde{U}$ maps to $U_p$.
\end{proof}

\begin{cor}
Let $N=p\ell$ with $\ell \equiv 1 \pmod{p}$ and $\epsilon=(-1,-1)$. Assume that $p$ is not a $p$-th power modulo $\ell$. Then the Eisenstein ideal of $\bT_H^\epsilon$ is generated by $U_p-1$. In particular, $\bT_H^\epsilon$ and $\bT_H^{0,\epsilon}$ are Gorenstein.
\end{cor}

\begin{proof}
Combine the previous proposition with Theorem \ref{thm:one interesting prime} and Mazur's good prime criterion (\S\ref{subsec:Mazur results}). 
\end{proof}

\section{Computation of some cup products}
\label{appendix:cohomology}
\subsection{Cohomology calculations}
\begin{lem}
\label{lem:Kummer injective}
If $\ell \not \equiv 0,1 \pmod{p}$, then the restriction map
\[
H^1(\Q_\ell,\F_p(1)) \to H^1(I_\ell,\F_p(1))
\]
is injective.
\end{lem}
\begin{proof}
Under the isomorphisms of Kummer theory, this map corresponds to the map $\Q_\ell^\times \otimes \F_p \to {\Q_\ell^{\mathrm{ur}}}^\times \otimes \F_p$ induced by the inclusion. Since $\ell \not \equiv 0,1 \pmod{p}$, $\Q_\ell^\times \otimes \F_p$ is generated by the class of $\ell$, which maps the class of $\ell$ in  ${\Q_\ell^{\mathrm{ur}}}^\times \otimes \F_p$, which is nonzero.
\end{proof}

\begin{lem}
\label{lem:H2 calculation}
Let $N=\ell_0 \cdots \ell_r$ be squarefree and assume $p \nmid N$. Let $V=\{i \ : \ p\mid (\ell_i-1)\}$. The local restriction maps induce an isomorphism
\[
H^2(\Z[1/Np],\F_p) \lrisom \bigoplus_{i=1}^r H^2(\Q_{\ell_i},\F_p) \cong \bigoplus_{i\in V} H^2(\Q_{\ell_i},\F_p).
\]
of vector spaces of dimension $\# V$.
\end{lem}
\begin{proof}
Just as in \cite[Lem.\ 12.1.1]{WWE3}, we know that 
\[
H^2(\Z[1/Np],\F_p) \ra H^2(\Q_p,\F_p) \oplus \bigoplus_{i=1}^r H^2(\Q_{\ell_i},\F_p)
\]
is a surjection because $H^3_{(c)}(\Z[1/Np],\F_p) \cong H^0(\Z[1/Np],\F_p(1))^* = 0$. By Tate duality, $H^2(\Q_{\ell},\F_p)=H^0(\Q_{\ell},\F_p(1))^*$, which is one-dimensional if $\ell \equiv 1\pmod{p}$ and zero otherwise. It remains to verify that $H^2(\Z[1/Np],\F_p)$ has the same dimension. This follows from the global Tate Euler characteristic computation of \cite[Lem.\ 12.1.1]{WWE3}. 
\end{proof}

The following is a consequence of Tate duality.
\begin{lem}
\label{lem:local tate desc}
Assume that $\ell \equiv 1 \pmod{p}$ is prime. Then $H^1(\Q_\ell,\F_p(-1)) = H^1(\Q_\ell,\F_p(1))$. This cohomology group is 2-dimensional, the unramified subspace is 1-dimensional, and the cup product pairing 
\[
\cup : H^1(\Q_\ell,\F_p(-1)) \times H^1(\Q_\ell,\F_p(1)) \lra H^2(\Q_\ell,\F_p)
\]
is non-degenerate and symplectic. In particular, the cup product of two unramified classes vanishes, and the cup product of a ramified class with a non-trivial unramified class does not vanish. 
\end{lem}

\subsection{Proof of Proposition \ref{prop:bc cup}} By the description of the number fields $K_i$ in Definition \ref{defn:K_i}, $\ell_0$ splits completely in $K_i$ if and only if $\ell_0 \equiv 1 \pmod{p}$ and the image $c_i|_{G_{\ell_0}}$ of $c_i$ in $H^1(\Q_{\ell_0},\F_p(-1))$ is zero. Since $b_0\vert_{G_{\ell_0}}$ is ramified and $c_i|_{G_{\ell_0}}$ is unramified, Lemma \ref{lem:local tate desc} implies that $c_i\vert_{G_{\ell_0}}=0$ if and only if $b_0\vert_{G_{\ell_0}} \cup c_i|_{G_{\ell_0}}=0$. By Lemma \ref{lem:H2 calculation}, this happens if and only if $b_0 \cup c_i=0$.

\section{Algebra}
\subsection{Some comments about Gorenstein defect}
\label{subsec:Gorenstein defect} Let $(A, \m_A, k)$ be a regular Noetherian local ring, and let $(R,\m_R)$ be a finite, flat, local $A$-algebra. 

More generally, for an $A$-module $M$, let $M^\vee = \Hom_A(M,A)$. Also, let $\bar{M}=M/\m_A M$. For a $k$-vector space $M$, let $M^* = \Hom_k(M,k)$. For an $R$-module $M$, give $M^\vee$ the $R$-module structure given by $(r \cdot f)(x)= f(rx)$ for $f \in M^\vee$ and $r \in R$, and let $g_R(M) = \dim_k(M/\m_R M)$ be the minimal number of generators of $M$. The assumptions on $A$ and $R$ imply that $R$ is a Cohen--Macaulay ring with dualizing module $R^\vee$. 

Define the \emph{Gorenstein defect} $\delta(R)$ of $R$ to be the integer $\delta(R) = g_R(R^\vee)-1$. Then $R$ is Gorenstein if and only if $\delta(R)=0$ \cite[Thm.\ 3.3.7,\ pg.\ 111]{BH1993}. If $R$ is complete intersection, then $R$ is Gorenstein \cite[Prop.\ 3.2.1, pg.\ 95]{BH1993}. Kilford and Wiese \cite[Defn.\ 1.4]{KW2008} define the Gorenstein defect of $R$ to be $\dim_k\Soc(\bar{R})-1$, where $\Soc(\bar{R})=\Ann_{\bar{R}}(\m_{\bar{R}})$. 
Our goal is Lemma \ref{lem:Gorenstein defects equality}: these definitions amount to the same thing. The proofs of the following lemmas are elementary, but we include them for completeness.

\begin{lem}
\label{lem:gorenstein defect in fin dim case}
Assume that $A=k$. Then the canonical pairing $R \times R^\vee \to k$ induces a perfect pairing $\Ann_R(\m_R) \times R^\vee/\m_R R^\vee \to k$. In particular, $\delta(R) = \dim_k(\Ann_R(\m_R))-1$.
\end{lem}
\begin{proof} 
By restriction, there is a surjective homomorphism of $R$-modules
\[
R^\vee \onto \Ann_R(\m_R)^\vee
\]
which is easily seen to factor through $R^\vee/\m_R R^\vee$. This gives the pairing. To show it is perfect, it is enough to show that the dual map $\Ann_R(\m_R) \to (R^\vee/\m_R R^\vee)^\vee$ is surjective as well. This map is induced by the canonical isomorphism $R \to R^{\vee\vee}$ given by $x \mapsto \mathrm{ev}_x$, where $\mathrm{ev}_x(f)=f(x)$ for $f \in R^\vee$.

Let $g \in (R^\vee/\m_R R^\vee)^\vee$ be an arbitrary element. Then $g$ is induced by a $R$-module homomorphism $\tilde{g}: R^\vee \to k$ such that $\tilde{g}(r.f) = 0$ for all $r \in \m_R$ and $f \in R^\vee$. By duality, we have $\tilde{g}=\mathrm{ev}_x$ for some $x \in R$. Then we have
\[
0=\tilde{g}(r.f)=\mathrm{ev}_x(r.f) =f(rx)
\]
for all $r \in \m_R$ and $f \in R^\vee$. This implies that $rx=0$ for all $r \in \m_R$, so $x \in \Ann_R(\m_R)$. Hence $g$ is in the image of $\Ann_R(\m_R) \to (R^\vee/\m_R R^\vee)^\vee$, so this map is surjective and the pairing is perfect.
\end{proof}

\begin{lem}
\label{lem:comparing bars and duals}
There is a canonical isomorphism of $\bar{R}$-modules $\overline{R^\vee} \cong \bar{R}^*$.
\end{lem}
\begin{proof}
Since $R$ is projective as an $A$-module, the map
\[
R^\vee = \Hom_A(R,A) \to \Hom_A(R,k) \cong \Hom_k(\bar{R},k) =  \bar{R}^*
\]
is a surjective morphism of $R$-modules. Since $\m_A$ annihilates the image, this map must factor through $\overline{R^\vee}$. Since $\overline{R^\vee}$ and $\bar{R}^*$ both have $k$-dimension equal to $\mathrm{rank}_A(R)$, the map $\overline{R^\vee} \to \bar{R}^*$ is an isomorphism.
\end{proof}
\begin{lem}
\label{lem:Gorenstein defects equality}
We have $\delta(R) = \delta(\bar{R})=\dim_k\Soc(\bar{R})-1$.
\end{lem}
\begin{proof}
We have
\[
\overline{R^\vee} \otimes_{\bar{R}} \bar{R}/\m_{\bar{R}} = (R^\vee \otimes_{R} \bar{R}) \otimes_{\bar{R}} \bar{R}/\m_{\bar{R}} = R^\vee \otimes_{R} R/\m_R
\]
so $g_{R}(R^\vee) = g_{\bar{R}}(\overline{R^\vee})$. By Lemma \ref{lem:comparing bars and duals}, we have 
\[
1+\delta(R) = g_{R}(R^\vee) = g_{\bar{R}}(\overline{R^\vee})=g_{\bar{R}}(\bar{R}^*) = 1 +\delta(\bar{R}).
\]
This shows that $\delta(R)=\delta(\bar{R})$. The equality $\delta(\bar{R})=\dim_k\Soc(\bar{R})-1$ follows from Lemma \ref{lem:gorenstein defect in fin dim case}.
\end{proof}

\subsection{Fiber products of commutative rings} Note that the category of commutative rings has all limits. The underlying set of the limit of a diagram of commutative rings is the limit of the diagram of underlying sets.

\begin{lem}
\label{lem:fiber prods}
Consider a commutative diagram
\[
\xymatrix{
A \ar[r]^-{\pi_B} \ar[d]_-{\pi_C} & B \ar[d]^-{\phi_B} \\
C \ar[r]^-{\phi_C} & D
}
\]
in the category of commutative rings. Assume that all the maps are surjective and that the map $\ker(\pi_B)\to \ker(\phi_C)$ induced by $\pi_C$ is an isomorphism. Then the canonical map $A \to B \times_D C$ is an isomorphism.
\end{lem}
The proof is a diagram chase.

\bibliographystyle{alpha}
\bibliography{CWEbib-2020-PG}

\end{document}